\documentclass[letterpaper,11pt]{article}
\usepackage[margin=1in]{geometry}
\usepackage[bookmarks, colorlinks=true, plainpages = false, colorlinks=true,
            citecolor=red,
            linkcolor=blue,
            anchorcolor=red,
            urlcolor=blue]{hyperref}
\usepackage{url}
\usepackage{amsmath,amsfonts,amsthm,amssymb,bm, verbatim,dsfont,mathtools}
\usepackage{color,graphicx,appendix}
\usepackage{etoolbox}

\usepackage{array}
\usepackage{multirow}

\usepackage{float}
\usepackage{tikz}
\usetikzlibrary{matrix,arrows,calc,shapes,backgrounds}
\usetikzlibrary{shapes.callouts,decorations.text}

\usetikzlibrary{shapes.misc}

\tikzset{cross/.style={cross out, draw=black, minimum size=2*(#1-\pgflinewidth), inner sep=0pt, outer sep=0pt},
cross/.default={1pt}}

\tikzstyle{int}=[draw, fill=blue!20, minimum size=2em]
\tikzstyle{dot}=[circle, draw, fill=blue!20, minimum size=2em]
\tikzstyle{dotred}=[circle, draw, fill=red!20, minimum size=2em]
\tikzstyle{init} = [pin edge={to-,thin,black}]
\tikzstyle{initred} = [pin edge={to-,thin,red}]
\tikzstyle{plan}=[draw, fill=blue!20, minimum size=2em, text width=5em, rounded corners,align=center]
\tikzstyle{planwide}=[draw, fill=blue!20, minimum size=2em, text width=8em, rounded corners,align=center]
\tikzstyle{nodedot}=[circle, draw, fill=white, minimum size=0.3cm,inner sep=0pt]
\tikzstyle{vertexdot}=[circle, draw, fill=white, minimum size=3,inner sep=0pt]
\tikzstyle{Medge}=[green!60!black, thick]
\tikzstyle{Bedge}=[red, thick]
\tikzstyle{Cedge}=[blue, thick]
\tikzstyle{Sedge}=[black, thick]
\tikzstyle{Mgiantedge}=[green!60!black, line width=3.0pt]
\tikzstyle{Bgiantedge}=[red,line width=3.0pt]
\tikzstyle{Cgiantedge}=[blue,line width=3.0pt]
\tikzstyle{Sgiantedge}=[black,line width=3.0pt]
\tikzstyle{shadedgiantnode}=[circle, draw, fill=black!10, minimum size=1cm, inner sep=0pt]
\tikzstyle{unshadedgiantnode}=[circle, draw, fill=white, minimum size=1cm, inner sep=0pt]
\makeatletter
\tikzset{my loop/.style =  {to path={
  \pgfextra{}
  [looseness=5,min distance=10mm]
  \tikz@to@curve@path},font=\sffamily\small
  }}  
\makeatletter 
\newcolumntype{C}[1]{>{\centering\arraybackslash}p{#1}}

\usepackage{xr,xspace}
\usepackage{todonotes}
\usepackage{paralist}
\usepackage{enumitem}

\usepackage{caption,subcaption,soul}
\usepackage{algorithm}
\usepackage{algpseudocode}
\makeatletter
\usepackage{romanbar}

\theoremstyle{plain}
\newtheorem{theorem}{Theorem}
\newtheorem{lemma}{Lemma}
\newtheorem{proposition}{Proposition}

\theoremstyle{definition}
\newtheorem{definition}{Definition}

\newtheorem{problem}{Problem}

\newtheorem{remark}{Remark}

\newtheorem*{remark*}{Remark}

\newcommand{\floor}[1]{\left\lfloor #1 \right\rfloor}
\newcommand{\ceil}[1]{\left\lceil #1 \right\rceil}

\usepackage{xspace,prettyref}

\newcommand{\Esingle}{\calE_{\rm single}}
\newcommand{\Edouble}{\calE_{\rm double}}
\newcommand{\ex}{\mathrm{ex}}

\newcommand{\lcm}{\mathrm{lcm}}
\newcommand{\termI}{\text{(I)}}
\newcommand{\termII}{\text{(II)}}

\newcommand{\sigmae}{\sigma^{\sf E}} 
\newcommand{\pie}{\pi^{\sf E}} 
\newcommand{\tpie}{\tpi^{\sf E}} 

\newcommand{\diverge}{\to\infty}

\newcommand{\iiddistr}{{\stackrel{\text{\iid}}{\sim}}}
\newcommand{\inddistr}{{\stackrel{\text{ind.}}{\sim}}}

\newcommand{\reals}{{\mathbb{R}}}
\newcommand{\integers}{{\mathbb{Z}}}
\newcommand{\naturals}{{\mathbb{N}}}

\newcommand{\Expect}{\mathbb{E}}
\newcommand{\expect}[1]{\mathbb{E}\left[ #1 \right]}

\newcommand{\Prob}{\mathbb{P}}

\newcommand{\prob}[1]{ \mathbb{P}\left\{ #1 \right\} }

\newcommand{\Cov}{\text{Cov}}
\newcommand\indep{\protect\mathpalette{\protect\independenT}{\perp}}
\def\independenT#1#2{\mathrel{\rlap{$#1#2$}\mkern2mu{#1#2}}}
\def\ci{\perp\!\!\!\perp} 
\def\Var{\mathrm{Var}}

\newcommand{\Bern}{{\rm Bern}}
\newcommand{\Binom}{{\rm Binom}}
\newcommand{\Uniform}{\mathrm{Uniform}}
\newcommand{\Pois}{{\rm Poi}}

\newcommand{\ie}{i.e.\xspace}
\newcommand{\iid}{i.i.d.\xspace}
\newrefformat{eq}{(\ref{#1})}
\newrefformat{chap}{Chapter~\ref{#1}}
\newrefformat{sec}{Section~\ref{#1}}
\newrefformat{alg}{Algorithm~\ref{#1}}
\newrefformat{fig}{Fig.~\ref{#1}}
\newrefformat{tab}{Table~\ref{#1}}
\newrefformat{rmk}{Remark~\ref{#1}}
\newrefformat{clm}{Claim~\ref{#1}}
\newrefformat{def}{Definition~\ref{#1}}
\newrefformat{cor}{Corollary~\ref{#1}}
\newrefformat{lmm}{Lemma~\ref{#1}}
\newrefformat{prop}{Proposition~\ref{#1}}
\newrefformat{prob}{Problem~\ref{#1}}
\newrefformat{app}{Appendix~\ref{#1}}
\newrefformat{hyp}{Hypothesis~\ref{#1}}
\newrefformat{thm}{Theorem~\ref{#1}}

\newcommand{\pth}[1]{\left( #1 \right)}
\newcommand{\qth}[1]{\left[ #1 \right]}
\newcommand{\sth}[1]{\left\{ #1 \right\}}

\newcommand{\indc}[1]{{\mathbf{1}_{\left\{{#1}\right\}}}}
\newcommand{\Indc}{\mathbf{1}}

\newcommand{\tpi}{{\widetilde{\pi}}}
\newcommand{\harmonic}{{h}}

\newcommand{\sfB}{{\mathsf{B}}}
\newcommand{\sfC}{{\mathsf{C}}}

\newcommand{\sfM}{{\mathsf{M}}}

\newcommand{\sfS}{{\mathsf{S}}}

\newcommand{\calE}{{\mathcal{E}}}
\newcommand{\calF}{{\mathcal{F}}}
\newcommand{\calG}{{\mathcal{G}}}
\newcommand{\calH}{{\mathcal{H}}}

\newcommand{\calJ}{{\mathcal{J}}}

\newcommand{\calL}{{\mathcal{L}}}

\newcommand{\calN}{{\mathcal{N}}}
\newcommand{\calO}{{\mathcal{O}}}
\newcommand{\calP}{{\mathcal{P}}}
\newcommand{\calQ}{{\mathcal{Q}}}

\newcommand{\calS}{{\mathcal{S}}}
\newcommand{\calT}{{\mathcal{T}}}

\newcommand{\orbit}{O}

\newcommand{\ER}{Erd\H{o}s-R\'enyi\xspace}

\newcommand{\Tr}{\mathsf{Tr}}

\begin{document}
\title{Testing correlation of unlabeled random graphs}
\author{Yihong Wu, Jiaming Xu, and Sophie H. Yu \thanks{
Y.\ Wu is with Department of Statistics and Data Science, Yale University, New Haven CT, USA, 
\texttt{yihong.wu@yale.edu}.
J.\ Xu and S.\ H.\ Yu are with The Fuqua School of Business, Duke University, Durham NC, USA, 
\texttt{\{jx77,haoyang.yu\}@duke.edu}.
Y.~Wu is supported in part by the NSF Grant CCF-1900507, an NSF CAREER award CCF-1651588, and an Alfred Sloan fellowship.
J.~Xu is supported by the NSF Grants IIS-1838124, CCF-1850743, and CCF-1856424.
}}

\date{\today}


\maketitle
\begin{abstract}
We study the problem of detecting the edge correlation between two random graphs with $n$ unlabeled nodes. This is formalized as a 
hypothesis testing problem, where under the null hypothesis, the two graphs are independently generated; 
under the alternative, the two graphs are edge-correlated under some latent node correspondence, but have the same marginal distributions as the null.
For both Gaussian-weighted complete graphs and dense \ER graphs (with edge probability $n^{-o(1)}$), 
we determine the sharp threshold at which the optimal testing error probability exhibits a phase transition from zero to one as $n\to \infty$.
For sparse \ER graphs with edge probability $n^{-\Omega(1)}$, we determine the threshold within a constant factor. 

The proof of the impossibility results is an application of the conditional second-moment
method, where we bound the truncated second moment of the likelihood ratio by carefully
conditioning on the typical behavior of the intersection graph (consisting of edges in both
observed graphs) and taking into account the cycle structure of the induced random permutation on the edges.
Notably, in the sparse regime, this is accomplished by leveraging the pseudoforest structure of subcritical \ER graphs and a careful enumeration of subpseudoforests that can be assembled from short orbits of the edge permutation.
\end{abstract}

\tableofcontents

\section{Introduction} \label{sec:intro}

Understanding and quantifying the correlation between datasets are among the most fundamental tasks in statistics.
In many modern applications, the observations may not be in the familiar form of vectors but rather graphs.
Furthermore, the node labels may be absent or scrambled, 
in which case one needs to decide the similarity between these \emph{unlabeled graphs} on the sheer basis of their topological structures. 
Equivalently, it amounts to determining whether there exists a node correspondence under which the (weighted) edges 
of the two graphs are correlated. This problem arises naturally in a wide range of fields:
\begin{itemize}
	\item In social network analysis, one is interested in 
deciding whether two friendship networks on different social platforms
share structural similarities, where the node labels are frequently anonymized due to privacy considerations 
\cite{narayanan2008robust,narayanan2009anonymizing}.

\item In computer vision, 
3-D shapes are commonly represented by geometric graphs, where nodes are subregions and edges encode the adjacency relationships between different regions.
A key building block for pattern recognition and image processing is to determine whether two graphs 
correspond of the same object that undergoes different rotations or deformations (changes in pose or topology) \cite{cour2007balanced,berg2005shape}.

\item In computational biology, an important task is to assess the correlation of two biological networks in two different species so as to enrich one dataset using the other \cite{singh2008global,vogelstein2011large}. 


\item In natural language processing, the so-called ontology alignment problem refers to uncovering the correlation between two knowledge graphs that are in either different languages \cite{haghighi2005robust} or different domains (e.g.~Library of Congress versus Wikipedia \cite{bayati2013message}).

\end{itemize}


Inspired by the hypothesis testing model proposed by Barak et al~\cite{barak2019nearly}, we formulate a general problem of testing network correlation as 
follows. Let  $G=([n], W)$
denote a weighted undirected graph on the node set $[n]\triangleq \{1,\ldots,n\}$ with weighted
adjacency matrix $W$, where $W_{ii}=0$ and 
for any $1 \le i < j \le n$,  $W_{ij}=1$ (or the edge weight) if $i$ and $j$
are adjacent and $W_{ij}=0$ otherwise. 
Recall that two weighted graphs $G=([n], W)$ and $H=([n], W')$ are \emph{isomorphic}
and denoted by $G\cong H$ if there exists a permutation (called a graph isomorphism) $\pi$ on $[n]$ such that 
$W_{ij}= W'_{\pi(i)\pi(j)}$ for all $i,j$.
Given a weighted graph $G$, its \emph{isomorphism class}
$\overline{G}$ is the equivalence class $\overline{G}= \{H: H \cong G\}$. 
We refer to an isomorphism class as an unlabeled graph and $\overline{G}$ as the unlabeled version of $G$.


\begin{problem}[Testing correlation of unlabeled graphs]\label{prob:HypTesting}
Let $G_1=([n], W)$ and $G_2=([n],W')$ be two weighted random graphs,
where the edge weights $\{(W_{ij}, W'_{ij}): 1 \le i < j  \le n\}$ are i.i.d.~pairs of random variables, 
and $W_{ij}$ and $W'_{ij}$ have the same marginal distribution.
Under the null hypothesis  $\calH_0$, $W_{ij}$ and $W'_{ij}$ are independent; 
under the alternative hypothesis  $\calH_1$, $W_{ij}$ and $W'_{ij}$ are correlated.
Given the unlabeled versions  of $G_1$ and $G_2$,  i.e., their isomorphism classes $\overline{G}_1=\{G: G \cong G_1\}$ and 
$\overline{G}_2=\{G: G \cong G_2\}$, the goal is to test  $\calH_0$ versus  $\calH_1$.
\end{problem}

Note that were the node labels of $G_1$ and $G_2$ observed, one could stack all the edge weights as a vector and reduce the problem to
simply testing the correlation of two random vectors. However, when the node labels are unobserved, 
the inherent correlation between $G_1$ and $G_2$ is obscured by the latent node correspondence, making the testing problem significantly more challenging. 
Indeed, since the observed graphs are unlabeled, the test needs to rely on \emph{graph invariants} 
(\ie, graph properties that are invariant under graph isomorphisms), such as subgraphs counts (e.g.~the number of edges and triangles) and spectral information (e.g.~eigenvalues of adjacency matrices or Laplacians).



 
In this work, we focus on the following two special cases of particular interests:
\begin{itemize}
\item (Gaussian Wigner model). 
Suppose that under  $\calH_1$, each pair of edge weights 
$W_{ij}$ and $W'_{ij}$ are jointly normal with zero mean, unit variance and correlation coefficient $\rho \in [0,1]$; 
under $\calH_0$, $W_{ij}$ and  $W'_{ij}$ are independent standard normals. 
Note that marginally  $W, W'$ are two Gaussian Wigner random matrices under both  $\calH_0$ and  $\calH_1$. 
 The correlated Gaussian Wigner model is proposed in \cite{ding2018efficient} as a prototypical model for random graph matching and further studied in \cite{FMWX19a,ganassali2019spectral}.

\item (\ER random graph).  Let $\calG(n,p)$ denote the \ER random graph model with edge probability $p \in [0,1]$. 
Consider the edge sampling process that generates a children graph from a given parent graph by keeping each edge independently with probability $s \in [0,1]$. 
Suppose that under  $\calH_1$, $G_1$ and $G_2$ are independently subsampled 
from a common parent graph $G \sim \calG(n,p)$; under  $\calH_0$, $G_1$ and $G_2$ are independently subsampled 
from two independent parent graphs $G, G' \sim \calG(n,p)$, respectively. See \prettyref{fig:test} for
an illustration.

Note that $G_1$ and $G_2$ are both instances of $\calG(n,ps)$ that are independent under $\calH_0$ and correlated under $\calH_1$. This specific model of correlated \ER random graphs is initially proposed by \cite{pedarsani2011privacy} and has been widely used for studying the problem of matching random graphs~\cite{cullina2016improved,cullina2017exact,barak2019nearly,mossel2019seeded,dai2019analysis,cullina2019partial,ding2018efficient,FMWX19b,ganassali2020tree,hall2020partial}. 


  \end{itemize}

\begin{figure}[H]
\centering
\begin{subfigure}{0.75\textwidth}
  \centering
  \includegraphics[trim= 10mm 10mm 10mm 10mm, width=1\linewidth]{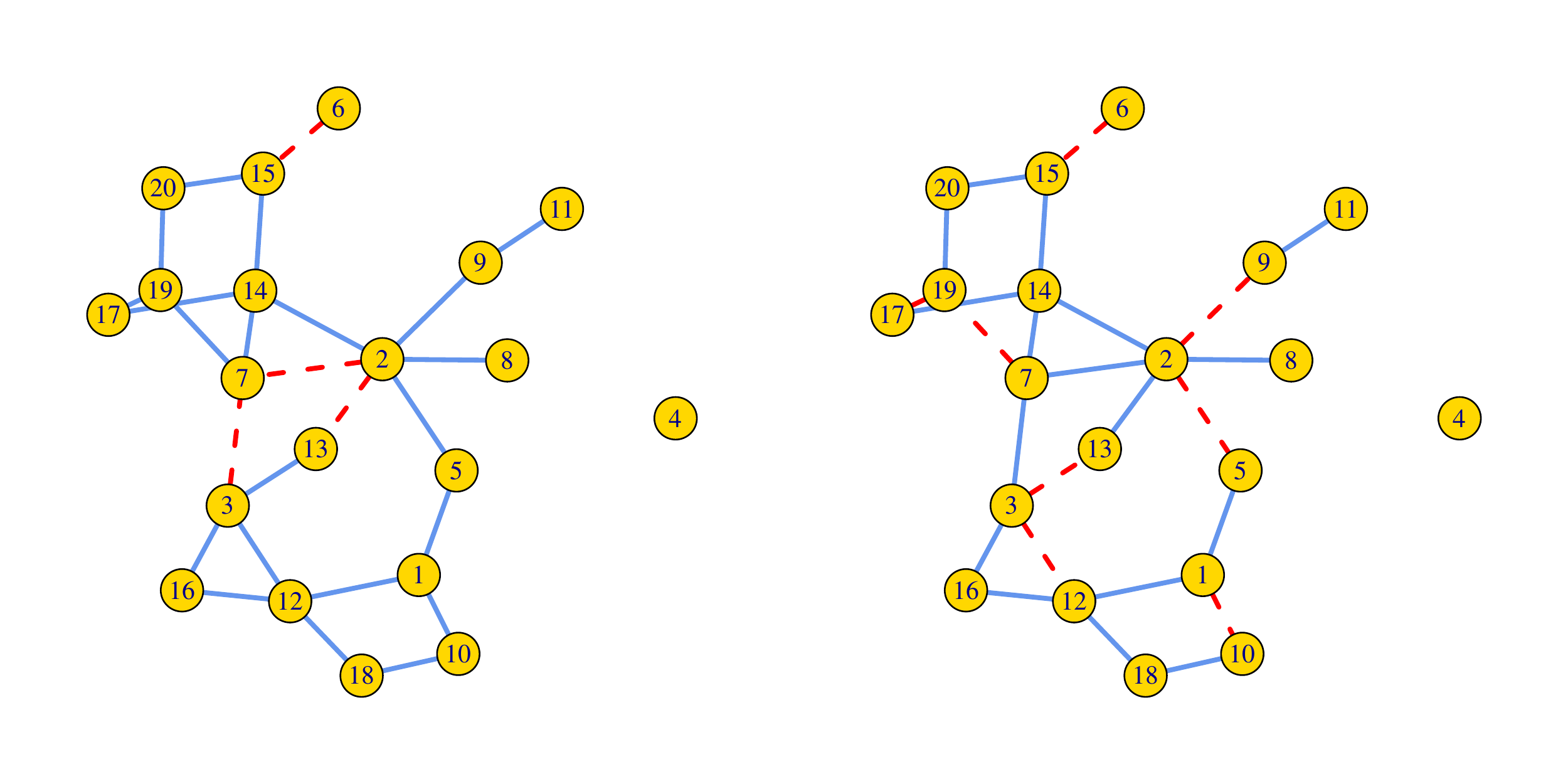}
  \caption{Two labeled graphs $G_1$ and $G_2$ are subsampled from a common parent graph  according to the correlated \ER graph model with $n=20$, $p=0.1$, and $s=0.8$ under $\calH_1$, where blue edges are edges sampled from the parent graph, and red, dashed edges are edges deleted from the parent graph.
}
  \label{fig:sub1}
\end{subfigure}%
\vskip\baselineskip
\begin{subfigure}{0.75\textwidth}
  \centering
  \includegraphics[
  trim= 10mm 10mm 10mm 10mm, width=1\linewidth]{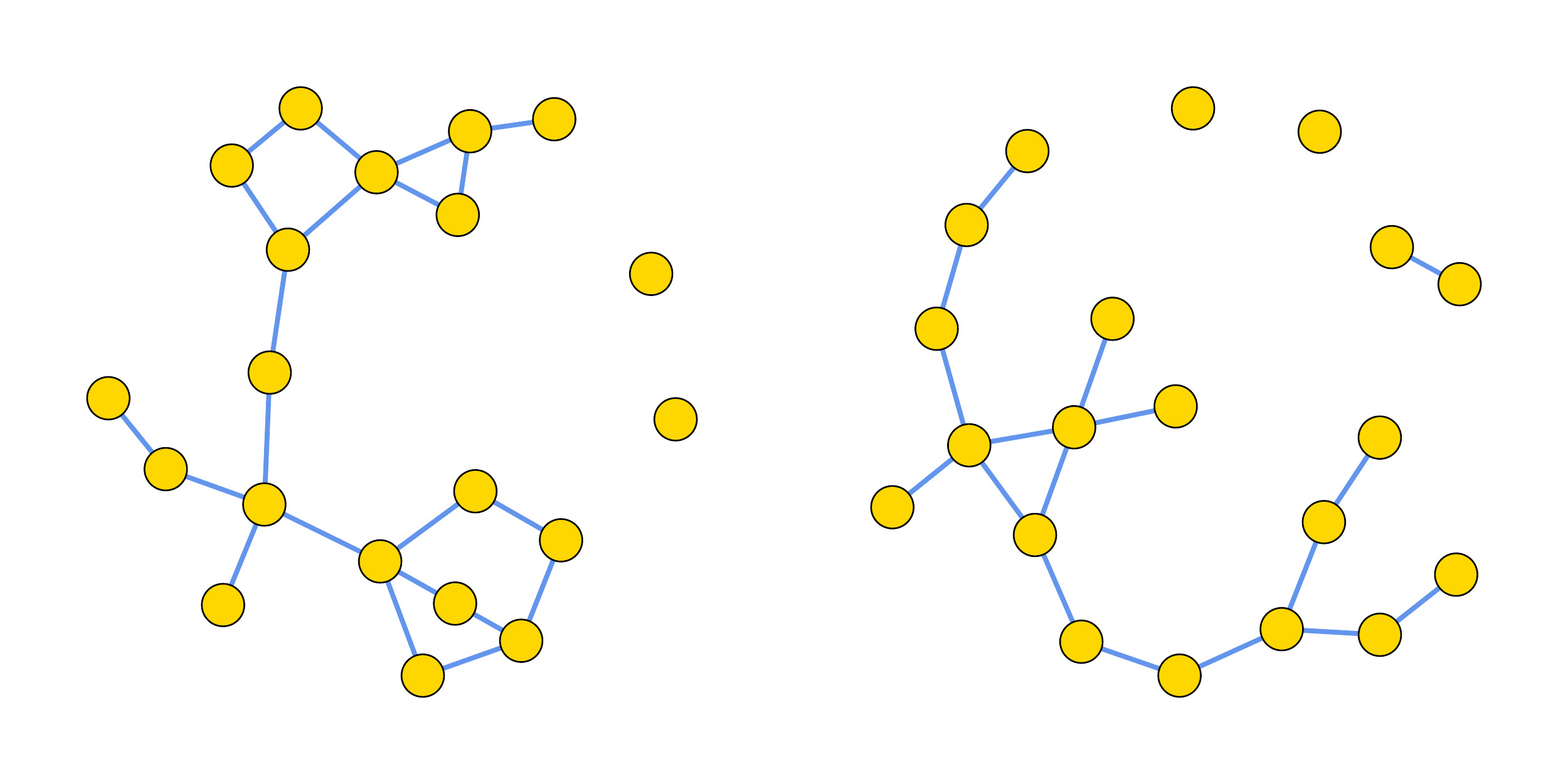}
  \caption{Two observed graphs $\overline{G}_1$ and $\overline{G}_2$ are the unlabeled versions  of $G_1$ and $G_2$, respectively. }
  \label{fig:sub2}
\end{subfigure}
\caption{
Example of testing correlation of two \ER random graphs. The task is to test the underlying hypothesis ($\calH_0$ or $\calH_1$) based on the two unlabeled graphs in panel (b).
}
\label{fig:test}
\end{figure}


We further focus on the following two natural types of testing guarantees.

\begin{definition}[Strong and weak detection]
Let $\calQ$ and $\calP$ denote the probability measure under  $\calH_0$  and $\calH_1$, respectively. 
We say a test statistic $\calT \left( \overline{G}_1,
\overline{G}_2\right)$ with threshold $\tau$ achieves 
\begin{itemize}
\item \emph{strong detection} if the sum of type I and type II error converges to 0 as $n\diverge$, that is, 
\begin{align}
    \lim_{n\diverge} \left[\calP\left(\calT\left(\overline{G}_1,\overline{G}_2\right)<\tau\right)+\calQ\left(\calT\left(\overline{G}_1, \overline{G}_2\right) \geq \tau\right)\right] = 0; \label{eq:detection}
\end{align}
\item \emph{weak detection}, if the sum of type I and type II error is bounded away from $1$ as $n\diverge$, that is, 
\begin{align}
    \limsup_{n\diverge} \left[\calP\left(\calT\left(\overline{G}_1,\overline{G}_2\right)<\tau\right)+\calQ\left(\calT\left(\overline{G}_1, \overline{G}_2\right) \geq \tau\right)\right] <1. \label{eq:weak_detection}
\end{align}
\end{itemize}
\end{definition}
Note that strong detection requires the test statistic to determine 
with high probability whether $\left( \overline{G}_1, \overline{G}_2\right)$ is drawn from $\calQ$ or $\calP$, while
weak detection only aims at strictly outperforming random guessing. It is well-known that the minimal sum of type I and type II error is $1-\mathrm{TV}\left(\calP,\calQ\right)$, achieved by the likelihood ratio test, where $\mathrm{TV}\left(\calP,\calQ\right)=\frac{1}{2} \int\left|\mathrm{d}\calP-\mathrm{d}\calQ\right|$ denotes the total variation distance between $\calP$ and $\calQ$. Thus strong and weak detection are equivalent to $\lim_{n \to \infty}\mathrm{TV}\left(\calP,\calQ\right) =1$ and $\liminf_{n \to \infty} \mathrm{TV}(\calP,\calQ)>0$, respectively.

Recent work~\cite{barak2019nearly} 
developed a polynomial-time test based on counting certain subgraphs
 that correctly distinguishes between 
 $\calH_0$ and  $\calH_1$ with probability at least $0.9$, 
 provided that the edge subsampling probability $s = \Omega \left(1\right)$ and the average degree satisfies certain conditions; however, the fundamental limit of detection remains elusive. The main objective of this paper is to obtain tight necessary and sufficient conditions for strong and weak detection.


\subsection{Main results}
\begin{theorem}[Gaussian Wigner model] \label{thm:gw}
If 
\begin{align}
\rho^2 \geq \frac{4\log n}{n-1},
\label{eq:1upperbound}
\end{align} 
then $\mathrm{TV}\left(\calP,\calQ\right) = 1+o\left(1\right)$.
Conversely, if 
\begin{align}
\rho^2 \leq \frac{(4-\epsilon) \log n}{n}
    \label{eq:1lowerbound}
\end{align}
for any constant $\epsilon>0$, then $\mathrm{TV}\left(\calP,\calQ\right) = o(1)$.
\end{theorem}


 

\begin{theorem}[\ER graphs]\label{thm:er}
If
\begin{align}
    s^2 \geq  \frac{2 \log n }{ (n-1) p\left( \log\frac{1}{p}-1+p\right) } , \label{eq:2upperbound}
\end{align} 
then $ \mathrm{TV}\left(\calP,\calQ\right) = 1-o\left(1\right)$.

Conversely, assume that $p$ is bounded away from $1$.
\begin{itemize}
\item (Dense regime): If $p=n^{-o(1)}$ and
\begin{align}
    s^2 \leq  \frac{(2-\epsilon) \log n}{np\left( \log\frac{1}{p}-1+p\right) } \label{eq:2lowerbound}
\end{align} 
for any constant $\epsilon>0$,
then $\mathrm{TV}\left(\calP,\calQ\right) = o(1)$.
\item (Sparse regime): If $p=n^{-\Omega(1)}$ and 
\begin{align}
s^2 \le \frac{1- \omega( n^{-1/3})}{np} \wedge c_0 \label{eq:2lowerbound_sparse_strong}
\end{align} 
for some universal constant $c_0$ ($c_0=0.01$ works),
then $\mathrm{TV}\left(\calP,\calQ\right) =1-\Omega(1)$.
In addition, if \prettyref{eq:2lowerbound_sparse_strong} holds and $s=o(1)$, 
then $\mathrm{TV}\left(\calP,\calQ\right) =o(1)$.
\end{itemize}
\end{theorem}

%


For the Gaussian Wigner model, \prettyref{thm:gw} shows that the fundamental limit of detection in terms of the limiting value of $ \frac{n\rho^2}{\log n}$ exhibits a sharp threshold at $4$, above which strong detection is possible and below which weak detection is impossible, 
a phenomenon known as the ``all-or-nothing'' phase transition~\cite{reeves2019all}.
In the \ER model, for dense parent graphs with $p=n^{-o(1)}$ and bounded away from $1$, \prettyref{thm:er} shows that a similar sharp threshold for $\frac{nps^2( \log (1/p) -1+p)}{ \log n}$ exists at $2$. 
Curiously, the function $p \mapsto p( \log\frac{1}{p}-1+p)$ is \emph{not} monotone and uniquely maximized at $p_* \approx 0.203$, the solution to the equation $\log \frac{1}{p}=2(1-p)$.
This shows the counterintuitive fact that the parent graph with edge density $p_*$ is the ``easiest'' for detection as it requires the lowest sampling probability $s$;
nevertheless, such non-monotonicity in the detection threshold can be anticipated by noting that in the extreme cases of $p=0$ and $p=1$, the observed two graphs are always independent and the two hypotheses are identical.

For sparse parent graphs with $p=n^{-\Omega(1)}$, the picture is less clear:
\begin{itemize}
	\item Unbounded average degree $np=\omega(1)$: 
	For simplicity, assume that $p=n^{-\alpha+o(1)}$ for some constant $\alpha \in (0,1]$. 	
	\prettyref{thm:er} implies that strong detection is possible if $\liminf nps^2 >  \frac{2}{\alpha}$ and weak detection is impossible if $\limsup nps^2 < 1$; these two conditions differ by a constant factor.

	\item Bounded average degree $np=\Theta(1)$: For simplicity, assume that $p=d/n$ for some constant $d>0$. 
	\prettyref{thm:er} shows that strong detection is possible if $s^2 >  \frac{2}{d}$ and impossible if $s^2 < c_0 \wedge \frac{1}{d}$.

\end{itemize}
For both cases, it is an open problem to determine the  sharp threshold for detection (or the existence thereof).

\begin{remark}[Simple test for weak detection]
\label{rmk:weak_detection}
In the non-trivial case of $p= \omega(1/n^2)$ and $p$ bounded away from $1$, as long as the sampling problem $s$ is any constant, weak detection can be achieved in linear time 
by simply comparing the number of edges of the two observed graphs. 
Intuitively, their difference behaves like a centered Gaussian with slightly different scale parameters under the two hypotheses which can then be distinguished non-trivially (see \prettyref{sec:pf-weak_detection} for a rigorous justification).
In view of the negative result for weak detection in \prettyref{thm:er}, we conclude that for parent graph with bounded average degree $np=O(1)$, 
weak detection is possible if and only if $s=\Omega(1)$. 
\end{remark}

As discussed in the next subsection, the testing procedure used to achieve strong detection in both \prettyref{thm:gw} and 
\prettyref{thm:er} involves a combinatorial optimization that is intractable in the worst case. 
Thus it is of interest to compare the optimal threshold to the performance of existing computationally efficient algorithms. These methods are based on subgraph counts that extend the simple idea of counting edges in \prettyref{rmk:weak_detection}. For \ER graphs, the polynomial-time test in \cite[Theorem 2.2]{barak2019nearly} (based on counting certain probabilistically constructed subgraphs) correctly distinguishes between 
 $\calH_0$ and  $\calH_1$ with probability at least $0.9$, 
 provided that the edge subsampling probability $s = \Omega \left(1\right)$
and
$nps \in \left[n^{\epsilon},n^{1/153}\right] \cup \left[n^{2/3},n^{1-\epsilon}\right]$
for some small constant $\epsilon>0$.
This performance guarantee is highly suboptimal compared to $s^2 = \Omega(\frac{\log n}{np \log(1/p) })$ given by \prettyref{thm:er}. 
In a companion paper \cite{MWXY20}, we propose a polynomial-time algorithm based on counting trees that achieves strong detection, 
 provided that $np \ge n^{-o(1)}$ and $\rho^2\triangleq \frac{s^2(1-p)^2}{(1-ps)^2} > \frac{1}{\beta}$ where $\beta \triangleq \lim_{k\to\infty} [t(k)]^{1/k} \approx 2.956$ and 
$t(k)$ is the number of unlabeled trees with $k$ vertices~\cite{otter1948number}.
Achieving the optimal threshold with polynomial-time tests remains an open problem.

\subsection{Test statistic and proof techniques} \label{sec:analysis_technique}
To introduce our testing procedure and the analysis, we first reformulate the testing problem given in \prettyref{prob:HypTesting} in a more convenient form. 
Due to the exchangeability of the (i.i.d.)~edge weights, observing the unlabeled version  is equivalent to observing its randomly relabeled version. Indeed, 
let $\pi_1$ and $\pi_2$ be two independent random permutations uniformly drawn from the set $\calS_n$ of all permutations on $[n]$.
Consider the relabeled version of $G_1=([n],W)$ with weighted adjacency matrix $A$, where $A_{ij} = W_{\pi_1(i)\pi_1(j)}$; similarly, let $B$ correspond to the relabeled version of $G_2=([n],W')$ with $B_{ij} = W'_{\pi_2(i)\pi_2(j)}$.
It is clear that observing the unlabeled graphs $\overline{G}_1$ and $\overline{G}_2$ is equivalent to observing the labeled graphs $A$ and $B$.
Since $\pi_2^{-1} \circ \pi_1$ is also a uniform random permutation, we arrive at the following formulation that is equivalent to \prettyref{prob:HypTesting}:

\begin{problem}[Reformulation of \prettyref{prob:HypTesting}]
\label{prob:HypTesting2}
	Let $A$ and $B$ denote the weighted adjacency matrices of two weighted graphs on the vertex set $[n]$,	both consisting of i.i.d.~edge weights. 
	Under  $\calH_0$, $A$ and $B$
are independent; under  $\calH_1$, conditional on a latent permutation $\pi$ drawn uniformly at random from $\calS_n$, 
$\{(A_{ij},B_{\pi(i)\pi(j)}): 1 \leq i<j\leq n\}$ are i.i.d.~and each pair $A_{ij}$ and $B_{\pi(i)\pi(j)}$ are correlated. 
Upon observing $A$ and $B$, the goal is to test $\calH_0$ versus $\calH_1$.
\end{problem}


Note that under  $\calH_1$, the latent random permutation $\pi$ represents the hidden node
correspondence under which $A$ and $B$ are correlated. For this reason, we refer to $\calH_1$ as the \emph{planted model} and $\calH_0$ as the \emph{null model}. 
The likelihood ratio is given by:
\begin{align}
\frac{\mathcal{P}\left(A,B \right) }{\mathcal{Q} \left(A,B\right)}
= \frac{1}{n!} \sum_{\pi \in \calS_n}  \frac{\mathcal{P}\left(A,B \mid \pi\right) }{\mathcal{Q} \left(A,B\right)},
\end{align}
which is the optimal test statistic but difficult to analyze due to the averaging over all $n!$ permutations. 
Instead, we consider the \emph{generalized likelihood ratio} by replacing the average with the maximum:
\begin{align}
 \max_{\pi \in \calS_n} \frac{\calP\left(A,B|\pi\right)}{\calQ\left(A,B\right)}. \label{eq:generalized_ratio}
\end{align}
As shown later in \prettyref{sec:3st}, for both the Gaussian Wigner and the \ER graph model, \prettyref{eq:generalized_ratio} is equivalent to 
\begin{align}
\calT(A, B) \triangleq \max_{\pi \in \calS_n} \; \calT_{\pi},\quad \mathrm{where\ } \calT_\pi \triangleq  \sum_{i<j} 
A_{ij}B_{\pi\left(i\right),\pi\left( j\right)}.
\label{eq:QAP_test}
\end{align}
which amounts to computing the maximal edge correlation over
all possible node correspondences between $A$ and $B$. 
As desired, the test statistic $\calT(A, B)$ is invariant to the relabeling of both $A$ and $B$ and can be applied to their unlabeled versions.
The combinatorial optimization problem \prettyref{eq:QAP_test} is an instance of
 the \textit{quadratic assignment problem} \cite{rendl1994quadratic}, 
 which is known to be NP-hard to solve 
  or to approximate within a growing factor \cite{makarychev2010maximum}.

To show the test statistic $\calT(A, B)$ achieves detection, 
first observe that in the planted model with hidden permutation $\pi$, 
$\calT(A, B)$ is trivially bounded 
from below
by $\sum_{i<j } A_{ij} B_{\pi(i)\pi(j)}$, which can be further shown to exceed some threshold $\tau$ with high probability by concentration inequalities. 
For the null model, we use a simple first moment argument (union bound) to show that $\calQ\left(\calT(A, B) \ge \tau\right) =o(1)$. 
Together we conclude that $\calT(A,B)$ with threshold $\tau$ achieves strong detection and $\mathrm{TV}\left( \calP, \calQ\right)=1-o(1)$. 


Next we provide an overview of the impossibility proof, which constitutes the bulk of the paper.  
To this end, we bound the second moment of the likelihood ratio. 
It is well-known that\footnote{Indeed, \prettyref{eq:secondmoment-tv1} follows from, e.g.,~
\cite[Lemma 2.6 and 2.7]{Tsybakov09} and \prettyref{eq:secondmoment-tv2} is by Cauchy-Schwarz inequality.}
\begin{align}
&\Expect_{\calQ}\left[\left( \frac{\calP(A,B)}{\calQ(A,B)} \right)^2\right] = O(1) \quad \Longrightarrow \quad  \mathrm{TV}(\calP(A,B),\calQ(A,B)) \le  1- \Omega(1)
\label{eq:secondmoment-tv1}\\
& \Expect_{\calQ}\left[\left( \frac{\calP(A,B)}{\calQ(A,B)} \right)^2\right] = 1+o(1) \quad \Longrightarrow \quad  \mathrm{TV}(\calP(A,B),\calQ(A,B)) =  o(1),
\label{eq:secondmoment-tv2}
 \end{align}
 which correspond to the impossibility of strong and weak detection, respectively. 

To compute the second moment, we introduce an independent copy $\tpi$ of the latent permutation $\pi$ and express the squared likelihood ratio as
\begin{align*}
\pth{\frac{\calP(A,B)}{\calQ(A,B)}}^2 
= \Expect_{\tpi\ci\pi}\qth{\prod_{i<j} X_{ij} }, \quad  \mathrm{where\ }X_{ij} \triangleq \frac{P( A_{ij},B_{\pi(i)\pi(j)}) P(A_{ij},B_{\tpi(i)\tpi(j)})}{
Q(A_{ij},B_{\pi(i)\pi(j)}) Q(A_{ij},B_{\tpi(i)\tpi(j)})},
\end{align*}
where for any $(i,j)\in \left[\binom{n}{2}\right]$, $Q$ denotes the joint density function of $A_{ij}$ and $B_{ij}$ under $\calQ$, and $P$ denotes the joint density function of $A_{ij}$ and $B_{\pi(i)\pi(j)}$ under $\calP$ given its latent permutation $\pi$.
Fixing $\pi$ and $\tpi$, we then decompose this as a product over independent randomness indexed by the so-called \emph{edge orbits}.
Specifically, the permutation $\sigma \triangleq \pi^{-1}\circ \tpi$ on the node set naturally induces a permutation $\sigmae$ on the edge set of the complete graph by permuting the end points.
Denoting by $\calO$ the collection of edge orbits (orbit in the cycle decomposition of the edge permutation $\sigmae$), we show that
\begin{align*}
\pth{\frac{\calP(A,B)}{\calQ(A,B)}}^2 
= \Expect_{\tpi\ci\pi}\qth{\prod_{O \in\calO} X_O },  \qquad X_O \triangleq \prod_{(i,j)\in O} X_{ij},
\end{align*}
where $X_O$'s are mutually independent under $\calQ$ conditioned on $\pi$ and $\tpi$.

Then, we take expectation $\Expect_{(A,B)\sim\calQ}$ on the right-hand side and interchange the two expectations.
For both the Gaussian and \ER models, 
this calculation can be explicitly carried out 
by evaluating the trace of certain operators.
In particular, this computation shows the following dichotomy: the second moment is $1+o(1)$ when $\rho^2 \le \frac{(2-\epsilon)\log n}{n}$, but unbounded
when $\rho^2 \ge  \frac{(2+\epsilon)\log n}{n}$, where $\rho$ is the correlation coefficient in the Gaussian case and $\rho=\frac{s(1-p)}{1-ps}$ in the \ER case. 
Compared with Theorems \ref{thm:gw} and \ref{thm:er}, we see that directly applying the second-moment method fails to capture the sharp threshold: 
The impossibility condition $\rho^2 \le \frac{(2-\epsilon)\log n}{n}$ is suboptimal by a multiplicative factor of $2$ in the Gaussian case 
and by an unbounded factor in the \ER case when $p=o(1)$. 

It turns out that the second moment is mostly influenced by those short edge orbits of length $k=O(\log n)$ for which $\prod_{|O|=k} X_O$ has a large expectation (see~\prettyref{sec:obstruction} for a detailed explanation).
Fortunately, the atypically large magnitude of $\prod_{|O|=k} X_\orbit$ can be attributed to certain rare events associated with the \emph{intersection graph} (edges that are included in both $A$ and $B^\pi=(B_{\pi(i)\pi(j)})$), which is distributed as $\calG(n,ps^2)$ under the planted model $\calP$.
This observation prompts us to apply the \emph{conditional second moment method}, which truncates the squared likelihood ratio on appropriately chosen \emph{global event} that has high probability under $\calP$.
Specifically,
\begin{itemize}
	\item In the dense case (including Gaussian model and dense \ER graphs), the dominating contribution comes from fixed points ($k=1$) which can be regulated by conditioning on the edge density of large induced subgraphs of the intersection graph. 
Note that for \ER graphs, even though the density of small induced subgraphs (e.g.\ induced by $\Theta(\log n)$ vertices) can deviate significantly from their expectations~\cite{balister2019dense}, fortunately we only need to consider sufficiently large subgraphs here. 

		\item For sparse \ER graphs, the argument is much more involved and combinatorial, as one need to control the contribution of not only fixed points, but all edge orbits of length $O(\log n)$. Crucially, the major contribution is due to those edge orbits that are subgraphs of the intersection graph. 
		Under the impossibility condition of \prettyref{thm:er}, the intersection graph $\calG(n,ps^2)$ is subcritical and a pseudoforest (each component having at most one cycle) with high probability. This global structure significantly limits the co-occurrence of edge orbits in the intersection graph.
		We thus truncate the squared likelihood ratio on the global event that intersection graph is a pseudoforest. To compute the conditional second moment, 
		we first study the graph structure of edge orbits, then reduce the problem to enumerating pseudoforests that are disjoint union of edge orbits and bounding their generating functions, and finally average over the cycle lengths of the random permutation $\sigma$.
This is the most challenging part of the paper.
\end{itemize}

\subsection{Connection to the literature}

This work joins an emerging line of research which examines inference problems on networks from statistical and computational perspectives.
We discuss some particularly relevant work below.

\paragraph{Random graph matching}
Given a pair of graphs, the  problem of \textit{graph matching} (or network alignment) 
refers to finding a node correspondence that maximizes the edge correlation \cite{conte2004thirty,livi2013graph},
which amounts to solving the QAP in \prettyref{eq:QAP_test}. 
Due to the worst-case intractability of the QAP, 
there is a recent surge of interests in the average-case analysis of matching two correlated random graphs~\cite{cullina2016improved,cullina2017exact,barak2019nearly,mossel2019seeded,dai2019analysis,cullina2019partial,ding2018efficient,FMWX19b,ganassali2020tree,hall2020partial},
where the goal is to reconstruct the hidden node correspondence between the two graphs accurately with high probability. 
To this end, the correlated \ER graph model (the alternative hypothesis  $\calH_1$ in \prettyref{prob:HypTesting2}) has been used as a popular model, for which 
the solution to the QAP \prettyref{eq:QAP_test} is the maximal likelihood estimator. 
It is shown in~\cite{cullina2017exact} that exact recovery of the hidden node correspondence with high probability  is information-theoretically possible 
if $nps^2-\log n \to +\infty $ and $p=O(\log^{-3}(n))$, and impossible
if $nps^2- \log n =O(1)$. In contrast,  the state of the art of polynomial-time algorithms achieve the exact recovery only when
$np=\text{poly}(\log n)$ and $1-s=1/\text{poly}(\log n)$~\cite{ding2018efficient,FMWX19a,FMWX19b}.

Recent work~\cite{cullina2019partial} initiated the study of almost exact recovery, that is, to obtain a matching 
(possibly imperfect) of size $n-o(n)$ that is contained in the true matching with high probability. 
It shows that the almost exact recovery is information-theoretically possible if 
$nps^2=\omega(1)$ and $p\le n^{-\Omega(1)}$, and impossible if $nps^2=O(1)$. 
Another work~\cite{ganassali2020tree} considers 
a weaker objective of partial recovery, that is, to output a matching 
that
contains $\Theta(n)$ correctly matched vertex pairs with high probability. 
It is shown that the partial recovery can be attained in polynomial time
by a neighborhood tree matching algorithm in the sparse graph regime where 
$nps \in (1,\lambda_0]$ for some constant $\lambda_0$ close to $1$ and $s \in (s_0,1]$ for some constant $s_0$ close to $1$. 
More recently, the partial recovery is shown to be information-theoretically impossible if 
$nps^2 \log \frac{1}{p} = o(1)$ when $p=o(1)$~\cite{hall2020partial}. 
For ease of comparison, we summarize the different thresholds
under various performance metrics in \prettyref{tab:thresholds} in the \ER model.

\begin{table}[!ht]
\centering
\begin{tabular}{ >{\centering\arraybackslash}m{1.5in}  | >{\centering\arraybackslash}m{2in}  |  >{\centering\arraybackslash}m{2in} }
\hline
 Performance metric & Positive result & Negative result \\[5pt]
\hline
Exact recovery & $nps^2 \ge \log n+\omega(1)$ \& $p=O(\log^{-3}(n))$~\cite{cullina2017exact} & $nps^2 \le \log n- \omega(1)$~\cite{cullina2016improved}  \\[10pt]
\hline
Almost exact recovery & $nps^2=\omega(1)$ \& $p\le n^{-\Omega(1)}$~\cite{cullina2019partial}  & $nps^2=O(1)$~\cite{cullina2019partial} \\[10pt]
\hline
Partial recovery & $nps \in (1,\lambda_0]$ \& $s \in (s_0,1]$~\cite{ganassali2020tree}  & $nps^2 \log \frac{1}{p} = o(1)$~\cite{hall2020partial} \\[10pt]
\hline
\shortstack{Detection \\(This paper)}
& $nps^2 \log \frac{1}{p} \ge (2+\epsilon) \log n$
%
& $p=n^{-o(1)}$ \& $nps^2\log \frac{1}{p} \le (2-\epsilon) \log n$, or $p = n^{-\Omega(1)}$ \& $s^2 \le \frac{1- \omega( n^{-1/3})}{np} \wedge 0.01$ \\[10pt]
%
\hline
\end{tabular}
\caption{Thresholds for various recovery criteria in the correlated \ER graph model when $p=o(1)$.}
\label{tab:thresholds}
\end{table}

In contrast to the aforementioned work focusing on recovering the latent matching, this work studies the hypothesis testing aspect of graph matching, which, nevertheless, has direct consequences on the recovery problem. 
%
As an application of the truncated second moment calculation, in a companion paper \cite{wu2021settling} we resolve the sharp threshold of recovery by characterizing the asymptotic mutual information $I(A, B; \pi)$. 
In particular, we show that in the dense regime with $p = n^{-o(1)}$, the sharp threshold of recovery exactly matches the detection threshold above which almost exact recovery is possible and below which partial recovery is impossible,
thereby closing the gap in \prettyref{tab:thresholds}. In the sparse regime with
$p=n^{-\Omega(1)}$, we show that the information-theoretic threshold for partial recovery is at $nps^2 \asymp 1$, which coincides with the detection threshold up to a constant factor.

%


\paragraph{Detection problems in networks}
There is a recent flurry of work using the first and second-moment methods to study hypothesis testing problems on networks with latent structures
such as community detection under stochastic block models~\cite{Mossel12,arias2014community,verzelen2015community,banks2016information}. 
Notably, a conditional second moment argument was applied by Arias-Castro and Verzelen in \cite{arias2014community} and \cite{verzelen2015community} to study the problem of detecting the presence of a planted community in dense and sparse \ER graphs, respectively. 
Similar to our work, for dense graphs, they condition on the edge density of induced subgraphs (see also the earlier work \cite{butucea2013} for the Gaussian model);
for sparse graphs, they condition on the planted community being a forest and bound the truncated second-moment by enumerating subforests using Cayley's formula.
However, the crucial difference is that in our setting simply enumerating the pseudoforests is inadequate for proving \prettyref{thm:er}. Instead, we need to take into account the cycle structure of permutations and enumerate \emph{orbit pseudoforests}, i.e., pseudoforests assembled from edge orbits (see the discussion before \prettyref{thm:jk} in \prettyref{sec:lb-sparse} for details). By separately accounting for orbits of different lengths and their graph properties, we are able to obtain a much finer control on the generating function of orbit pseudoforests that allows the conditional second moment to be bounded after averaging over the random permutation. 
This proof technique is of particular interest, and likely to be useful for other detection problems regarding permutations. 

Finally, we mention that the recent work~\cite{racz2020correlated} studied a related correlation detection problem, where 
the observed two graphs are either independent, or correlated randomly growing graphs (which grow together until time $t_*$ and grow independently afterwards
according to either uniform and preferential attachment models). Sufficient conditions are obtained for both weak detection and strong detection as $t_* \to \infty$.
However, the problem setup, main results, and proof techniques
are very different from the current paper.



\subsection{Notation and paper organization}
	\label{sec:notation}

For any $n \in \naturals$, let $[n]=\{1,2,\cdots,n\}$ and $\calS_n$ denote the set of all permutations on $[n]$. 
For a given graph $G$, let $V(G)$ denote its vertex set and $E(G)$ its edge set. 
For two graphs on $[n]$ with (weighted) adjacency matrices $A$ and $B$, their \emph{intersection graph} is a graph on $[n]$ with (weighted) adjacency matrix $A \wedge B$, where 
\begin{equation}
(A \wedge B)_{ij} \triangleq A_{ij}B_{\pi(i)\pi(j)};
\label{eq:intersection}
\end{equation}
 in the unweighted case, the edge set of $A\wedge B$ is the intersection of those of $A$ and $B$.
Given a permutation $\pi \in \calS_n$, let $ B^{\pi} = (B_{\pi(i)\pi(j)})$ denote the relabeled version of $B$ according to $\pi$. 
For any $S\subset[n]$, define $e_A(S) \triangleq \sum_{i<j\in S}A_{ij}$ as the total edge weights in the subgraph induced by $S$.

For any $a,b \in \reals$, let $a\wedge b = \min\{a,b\}$ and $a\vee b = \max\{a,b\}$. Given any $n,m\in\naturals$, let $\lcm(n,m)$ denote the least common multiple of $n$ and $m$. Given any $n,m \in \naturals$, and some nonnegative integers $\{k_{i}\}_{i=1}^m $ such that $\sum_{i=1}^m k_i = n$, let $\binom{n}{k_1,\ k_2,\ ,\cdots,k_m} = \frac{n!}{k_1!k_2!\cdots k_m!}$ be a multinomial coefficient.
We use standard asymptotic notation: for two positive sequences $\{a_n\}$ and $\{b_n\}$, we write $a_n = O(b_n)$ or $a_n \lesssim b_n$, if $a_n \le C b_n$ for some an absolute constant $C$ and for all $n$; $a_n = \Omega(b_n)$ or $a_n \gtrsim b_n$, if $b_n = O(a_n)$; $a_n = \Theta(b_n)$ or $a_n \asymp b_n$, if $a_n = O(b_n)$ and $a_n = \Omega(b_n)$; 
$a_n = o(b_n)$ or $b_n = \omega(a_n)$, if $a_n / b_n \to 0$ as $n\diverge$.


The rest of the paper is organized as follows. 
In \prettyref{sec:3st} we prove the positive result of strong detection for both Gaussian Wigner model and \ER random graphs. 
To lay the groundwork for the conditional second moment method,
in \prettyref{sec:uncond}
we present the unconditional second moment calculation and discuss the key reasons for its looseness. 
 \prettyref{sec:lb-dense} presents the conditional second-moment proof for weak detection in the dense regime. Due to their similarity, the proof for the Gaussian Wigner model 
is given in \prettyref{sec:gaussian_second_moment} and the (more technical) proof for dense \ER graphs is postponed till \prettyref{sec:er_dense}. 
\prettyref{sec:lb-sparse} provides 
the impossibility proofs of both strong and weak detection for sparse \ER random graphs. 
Several other technical proofs are also relegated to supplementary materials in \ref{sec:supp_1_3_4} and  \ref{sec:supp_lb_sparse}.
Some useful concentration inequalities and facts about random permutations are collected
in appendices for readers' convenience.

\section{First Moment Method for Detection}
\label{sec:3st}

In this section we prove the positive parts of Theorems \ref{thm:gw} and \ref{thm:er} by analyzing the test statistic \prettyref{eq:generalized_ratio}. 
Recall from \prettyref{sec:analysis_technique}
the reformulated \prettyref{prob:HypTesting2} 
with observations $A$ and $B$, whose distributions are specified as follows:

\begin{itemize}
    \item Gaussian Wigner model. 
    \begin{align}
     &\calH_0:  \left(A_{ij},B_{ij}\right) \iiddistr  \calN  \Big ( \left(\begin{smallmatrix} 0\\ 0\end{smallmatrix}\right) , \left(\begin{smallmatrix} 1 & 0 \\ 0 & 1 \end{smallmatrix}\right) \Big), \label{eq:gaussianh0} \\ 
     &\calH_1:  \left(A_{ij}, B_{\pi\left(i\right)\pi\left( j\right)}\right) \iiddistr \calN  \Big( \left(\begin{smallmatrix} 0\\ 0\end{smallmatrix}\right) , \left(\begin{smallmatrix} 1 & \rho \\ \rho & 1 \end{smallmatrix}\right)  \Big )  \text{ conditional on $\pi \sim \Uniform(\calS_n)$}.
      \label{eq:gaussianh1}
    \end{align}
    \item \ER random graph
    \begin{align}
   \calH_0 &:  \left(A_{ij},B_{ij}\right)  \iiddistr  \; \Bern(ps) \otimes \Bern(ps) ,  \label{eq:erh0} \\
   \calH_1 &:  \left(A_{ij},B_{ \pi\left(i\right)\pi\left(j\right)}\right) \iiddistr  \text{ pair of correlated $\Bern\left(p s\right)$ conditional on $\pi \sim \Uniform(\calS_n)$, where}\nonumber \\
    & \quad \quad \quad \quad \quad \quad \ A_{ij}\sim \Bern\left(p s\right) \text{ and }  B_{ \pi\left(i\right)\pi\left(j\right)}  \sim  
        \begin{cases} 
            \Bern\left(s\right) & \text{if } A_{ij}=1 \\
            \Bern\left(\frac{ps(1-s)}{1-ps}\right) & \text{if } A_{ij}=0
        \end{cases} . \label{eq:erh1}
\end{align}
\end{itemize}

Then we get
\begin{align}
\frac{\mathcal{P}\left(A,B|\pi\right) }{\mathcal{Q} \left(A,B\right)}
= \prod_{1\le i<j\le n} L\left (A_{ij}, B_{\pi(i)\pi(j)}\right), \label{eq:ratio} 
\end{align}
where 
\begin{align}
    L\left (A_{ij}, B_{\pi(i)\pi(j)}\right)
    & \triangleq \frac{P\left (A_{ij}, B_{\pi(i)\pi(j)}\right)}{Q\left (A_{ij}, B_{\pi(i)\pi(j)}\right)}. \label{eq:L_a_b}
    \end{align}
For the Gaussian Wigner model,
we have 
\begin{equation}
L(a,b) = \frac{1}{\sqrt{1-\rho^2}}  \exp \left( \frac{-\rho^2 \left( b^2 + a^2 \right)+2\rho ab }{2\left(1-\rho^2\right)}\right).
\label{eq:Lba-gaussian}
\end{equation}
For the \ER graph model, we 
have
\begin{align}
L(a,b) 
&= \begin{cases}
\frac{1}{p}  & a=1,b=1 \\
\frac{1-s}{1-ps}  & a=1,b=0\text{ or } a=0,b=1\\
\frac{1-2ps+ps^2}{(1-ps)^2}  & a=0,b=0 \\
\end{cases}
\,.
\label{eq:Lba-ER1}
\end{align}
Then, it yields that the generalized likelihood ratio test \prettyref{eq:generalized_ratio} is equivalent to 
\begin{align}
    \max_{\pi \in \calS_n} \frac{\calP\left(A,B|\pi\right)}{\calQ\left(A,B\right)} 
    & \iff \max_{\pi \in \calS_n} \sum_{i<j}A_{ij} B_{\pi(i)\pi(j)} \label{eq:generalized}
\end{align}
for both Gaussian Wigner model and \ER random graphs.

\subsection{Proof of \prettyref{thm:gw}: positive part}
Throughout the proof, denote $m=\binom{n}{2}$ for brevity. 
Without loss of generality, we assume that 
\prettyref{eq:1upperbound} holds with equality, i.e., $\rho^2 = \frac{2n\log n}{m}$; 
otherwise, one can apply the test to $(A',B')$ where $A'=\cos(\theta) A + \sin(\theta) Z$, $B'=\cos(\theta) B + \sin(\theta) Z$ 
with an appropriately chosen $\theta$, and $Z$ is standard normal and independent of $(A,B)$.
Define
\begin{equation}
\tau = \rho m - a_n
\label{eq:tau-gaussian}
\end{equation}
where $a_n$ is some sequence to be chosen satisfying $a_n=\omega(n)$ and $a_n=O(n^{3/2})$.

We first analyze the error event under the alternative hypothesis
\prettyref{eq:gaussianh1}.
Let $\pi$ denote the latent permutation such that $(A_{ij},B_{\pi(i)\pi(j)})$ are iid pairs of standard normals with correlation coefficient $\rho$.
Applying the Hanson-Wright inequality (see \prettyref{lmm:hw} in \prettyref{app:concentration}) with $M=I_m$ to $\calT_{\pi} = \sum_{1\le i<j\le n} A_{ij} B_{\pi(i)\pi(j) }$,
we get that 
\begin{align*}
\calP\pth{ \calT_{\pi} \le \tau } = \calP\pth{ \calT_{\pi} \le \rho m -a_n }  \le e^{- c a_n}+e^{-ca_n^2/m},
\end{align*}
for some universal constant $c$. Since  by definition $\calT \ge \calT_{\pi}$,  it follows that $\calP\left( \calT \le \tau \right) = o(1)$. 

To analyze the error event under the null hypothesis \prettyref{eq:gaussianh0},
in which case for each $\pi\in \calS_n$, $(A_{ij},B_{\pi(i)\pi(j)})$ are iid pairs of independent standard normals. Note that for $X,Y\iiddistr \calN(0,1)$ and any $\lambda \in (-1,1)$, we have
\begin{align}
    \Expect \left[\exp\left(\lambda X Y\right) \right] = \Expect \left[ \exp\left(\frac{\lambda^2Y^2}{2}\right) \right]
    = \frac{1}{\sqrt{1-\lambda^2}}. \label{eq:MGF_XY}
\end{align}
Then by the Chernoff bound, for any $\lambda \in (0,1)$,
\begin{align*}
    \calQ\left( \calT_{\pi} \ge  \tau \right) = \calQ\left( \exp(\lambda \calT_{\pi})\ge  \exp (\lambda\tau )\right)
    \le \exp\sth{ - \frac{m}{2} \log(1-\lambda^2) -  \lambda \tau}.
\end{align*}
Choosing $\lambda = \frac{\tau}{m}$, which satisfies $0 < \lambda=o(1)$ in view of \prettyref{eq:1upperbound} and \prettyref{eq:tau-gaussian}, we have
$\calQ( \calT_{\pi} \ge  \tau ) \le  e^{ - \frac{\tau^2}{2m} + O(\tau^4/m^3) }$.
Finally by the union bound and Stirling approximation that $n! \le e n^{n+\frac{1}{2}} e^{-n}$, $\calQ( \calT \ge  \tau ) \le  n! e^{ - \frac{\tau^2}{2m} + O(\tau^4/m^3) } = o(1)$, provided that 
$\frac{\rho^2 m}{2} -  \rho a_n - O(\rho^4 m) - n \log\frac{n}{e} - \frac{\log n}{2} \to + \infty$. 
This is ensured by the assumption that $\rho^2 = \frac{2n\log n}{m}$ and the choice of $a_n=n^{1.1}$.

\subsection{Proof of \prettyref{thm:er}: positive part}

Throughout the proof, denote $m=\binom{n}{2}$ for brevity. 
Without loss of generality, we assume that 
\prettyref{eq:2upperbound} holds with equality, i.e., 
\begin{align}
mps^2 \left( \log \frac{1}{p}  -1 +p  \right)  = n \log n \label{eq:2upperbound_eq};
\end{align}
otherwise, one can apply the test to $(A',B')$ where $A'$ ($B'$) are edge-subsampled from $A$ ($B$) with an appropriately chosen 
subsampling probability $s'$. It follows from \prettyref{eq:2upperbound_eq} that $p\ge 1/n$ and $mps^2 =\Omega(n)$. 
Define
\begin{equation}
\tau =  m ps^2 \left( 1-  \delta_n\right)
\label{eq:tau-er}
\end{equation}
where $0<\delta_n<1$ is some sequence to be chosen satisfying $\delta_n=\omega\left( 1/\sqrt{ m ps^2}  \right)$.

We first analyze the error event under the alternative hypothesis \prettyref{eq:erh1}. Let $\pi$ denote the latent permutation such that $(A_{ij} B_{\pi(i)\pi(j)})\iiddistr \Bern(ps^2)$. 
Then $\calT_{\pi}\sim \Binom(m,ps^2)$. 
By applying the Chernoff bound~\prettyref{eq:chernoff_binom_left}:
$
    \calP\left(\calT_{\pi}\le \tau \right)
    \le \exp\left(- \delta_n^2 mps^2 /2 \right)
    = o(1).
$
Since by definition $\calT \ge  \calT_{\pi}$, it follow that $\calP\left( \calT\le \tau \right)=o(1)$.


Next, we analyze the error event under the null hypothesis \prettyref{eq:erh0}, in which case for each $\pi\in \calS_n$, $(A_{ij}B_{\pi(i)\pi(j)}) \iiddistr \Bern(p^2s^2)$ and thus $\calT_{\pi}\sim \Binom\left(m,p^2 s^2\right)$. 
Using the multiplicative Chernoff bound for Binomial distributions~\prettyref{eq:chernoff_binom_right},
we obtain that 
\begin{align*}
\calQ\left( \calT_{\pi}\ge  \tau \right) & \le \exp \left( -  \tau \log \frac{\tau}{e\mu}  - \mu \right) \\
& = \exp \left( -  mps^2  (1-\delta_n)  \log \frac{1-\delta_n}{e p}   - mp^2s^2 \right)  \\
& \le  \exp \left( -  mps^2  \left( \log\frac{1}{p} -1 + p  \right) + mps^2 \delta_n \log \frac{1}{p}  \right),
\end{align*}
where $\mu = mp^2s^2$, and the last inequality holds due to $ \left(1-\delta_n\right)\log \frac{1-\delta_n}{e } \ge  -1$. 

Then by applying union bound and Stirling approximation that $n!\le en^{n+\frac{1}{2}}e^{-n},$ we have that 
\begin{align*}
\calQ\left( \calT \ge  \tau\right) 
& \le e  \exp \left( -  mps^2  \left( \log\frac{1}{p} -1 + p  \right) + mps^2 \delta_n \log \frac{1}{p}   + n \log \frac{n}{e}  + \frac{1}{2} \log n  \right) \\
& = e  \exp \left(  \left( mps^2\right)^{0.6}  \log \frac{1}{p}  - n  + \frac{1}{2} \log n  \right)  =o(1),
\end{align*}
where  the first equality holds by the assumption $mps^2 \left( \log \frac{1}{p}  -1 +p  \right)  = n \log n$ and choosing $\delta_n = 1/(mps^2)^{0.4}$;
the last equality holds by the claim that $\left(mps^2\right)^{0.6}  \log \frac{1}{p} =o(n)$. 
To finish the proof, it suffices to verify the claim, which is done separately in the following two cases.  

Suppose $1-p=\Omega(1)$. Then  $\log(1/p) -1 +p =\Omega(\log(1/p))$. Thus in view of  assumption~\prettyref{eq:2upperbound_eq}
and $p \ge 1/n$, we get that 
$ \left( mps^2\right)^{0.6}  \log \frac{1}{p} \le O \left( (n\log n)^{0.6} \log^{0.4}(1/p) \right) = o(n)$.

Suppose $1-p=o(1)$.  As $\log(1/p)-1+p  \ge (1-p)^2/2$, it follows from assumption~\prettyref{eq:2upperbound_eq}
that $(1-p)^2=\Omega(\log n/n)$. Furthermore, $\log(1/p) \le  \frac{1-p}{p}$.
Thus by assumption~\prettyref{eq:2upperbound_eq}, 
$ \left( mps^2\right)^{0.6}  \log \frac{1}{p} \le O \left(  (n\log n)^{0.6} (1-p)^{-0.2} \right) = O\left( n^{0.7} \log^{0.5} n\right)=o(n)$. 

\section{Unconditional Second Moment Method and Obstructions}
\label{sec:uncond}

In this section, we apply the unconditional second moment method 
to derive impossibility conditions for detection. As mentioned in \prettyref{sec:analysis_technique} (and described in details in \prettyref{sec:obstruction}), these conditions do not match the positive results in \prettyref{sec:3st}, due to the obstructions presented by the short edge orbits.
To overcome these difficulties, 
in Sections \ref{sec:lb-dense} and \ref{sec:lb-sparse}, we apply the conditional second moment method
by building upon the second moment computation in this section. 
We start by introducing some preliminary definitions associated with
permutations. 


\subsection{Node permutation, edge permutation, and cycle decomposition}
\label{sec:orbit}

Let $\sigma\in S_n$ be a permutation on $[n]$.
For each element $a\in[n]$, its \emph{orbit} is a cycle  $(a_0,\ldots,a_{k-1})$ for some $k\le n$, where $a_i=\sigma^{i}(a), i=0,\ldots,k-1$ and $\sigma(a_{k-1})=a$.
Each permutation can be decomposed as disjoint orbits.
For example, consider the permutation $\sigma\in S_8$ that swaps $1$ with $2$, swaps $3$ with $4$, and cyclically shifts $5678$. Then $\sigma$ consists of three orbits represented in canonical notation as $\sigma=(12)(34)(5678)$.

Consider the complete graph $K_n$ with vertex set $[n]$. Each permutation $\sigma\in S_n$ naturally induces a permutation $\sigmae$ on the edge set of $K_n$, the set $\binom{[n]}{2}$ of all unordered pairs,
according to
\begin{equation}
\sigmae((i,j)) \triangleq (\sigma(i),\sigma(j)).
\label{eq:sigmae}
\end{equation}
We refer to $\sigma$ and $\sigmae$ as \emph{node permutation} and \emph{edge permutation}, whose orbits are refereed to as \emph{node orbits} and \emph{edge orbits}, respectively.
For each edge $(i,j)$, let $\orbit_{ij}$ denotes its orbit under $\sigmae$.
As a concrete example, consider again the permutation $\sigma=(12)(34)(5678)$.
Then $\orbit_{12}=\{(1,2)\}$ and $\orbit_{34}=\{(3,4)\}$ are $1$-edge orbits (fixed point of $\sigmae$) and $\orbit_{56}=\{(5,6),(6,7),(7,8),(8,5)\}$ is a $4$-edge orbit.
(See \prettyref{tab:orbits} in \prettyref{sec:classification_edge_orbits} for more examples.)

%
%

The cycle structure of the edge permutation is determined by that of the node permutation. 
Let $n_k$ (resp. $N_k$) denote the number of $k$-node (resp. $k$-edge) orbits in $\sigma$ (resp. $\sigmae$). 
For example, 
\begin{equation}
N_1 = \binom{n_1}{2}+ n_2, \quad N_2 = \binom{n_2}{2}\times 2+n_1 n_2 + n_4. 
\label{eq:N1N2}
\end{equation}
This is due to the following reasoning.
\begin{itemize}
    \item Consider a $1$-edge orbit given by $\{(i,j)\}$. Since $(i,j)$ is unordered, it follows that either both $i,j$ are fixed points of $\sigma$ or $i,j$ form a $2$-node orbit of $\sigma$. Thus, $N_1= \binom{n_1}{2}+n_2$. 
    \item Consider a $2$-edge orbit given by $\{(i,j),(\sigma(i),\sigma(j))\}$. Then there are three cases: 
		(a) $i,j$ belong to two different $2$-node orbits;
		(b) $i$ is a fixed point and $j$ lies in a $2$-node orbit;
		(c) $i,j$ belong to a common $4$-node orbit of the form $(i*j*)$.
		Thus $N_2 = \binom{n_2}{2}\times 2 + n_1 n_2 + n_4$. 
\end{itemize}

\subsection{Second moment calculation} \label{sec:second-moment-calculation}

Recall the second-moment method described in~\prettyref{sec:analysis_technique}.
When $\calP$ is a mixture distribution, the calculation of the second moment can proceed as follows. 
Note that the likelihood ratio is $\frac{\calP(A,B)}{\calQ(A,B)} = \Expect_\pi[\frac{\calP(A,B|\pi)}{\calQ(A,B)}]$, where $\pi$ is a random permutation 
uniformly distributed over $\calS_n$. 
Introducing an independent copy $\tpi$ of $\pi$  and noting that $B$ has the same marginal distribution under both $\calP$ and $\calQ$, the squared likelihood ratio can be expressed as 
\begin{align}
\pth{\frac{\calP(A,B)}{\calQ(A,B)}}^2 
= & ~ \Expect_{\tpi\ci\pi}\qth{\frac{\calP(A,B|\pi)}{\calQ(A,B)} \frac{ \calP(A,B|\tpi)}{\calQ(A,B)}}  \overset{\prettyref{eq:L_a_b}}{=}  \Expect_{\tpi\ci\pi}\qth{\prod_{i<j} X_{ij} }, \label{eq:LR-sq}
\end{align}
where  $\pi \indep{} \tpi$ denotes that $\pi$ and $\tpi$ are independent, and
\begin{align}
X_{ij}\triangleq L\left(A_{ij},B_{\pi(i)\pi(j)} \right) L\left(A_{ij},B_{\tpi(i)\tpi(j)}\right). \label{eq:X_ij}
\end{align}

Interchanging the expectations yields 
\begin{align}
    \Expect_{\mathcal{Q}}\left[\left(\frac{\mathcal{P}  \left(A,B\right)}{\mathcal{Q}  \left(A,B\right)}\right)^2\right]
    =
    \Expect_{\tpi\ci\pi}\qth{    \Expect_{\left(A,B\right)\sim\calQ}    \qth{\prod_{i<j} X_{ij} }}. \label{eq:second1}
\end{align}


Fixing $\pi$ and $\tpi$, we first compute the inner expectation in \prettyref{eq:second1}. 
Observe that $X_{ij}$ may not be independent across different pairs of $(i,j)$. For example, suppose $(i_1,j_1)\neq (i_2,j_2)$ and $(\pi(i_1),\pi(j_1)) = (\tpi(i_2),\tpi(j_2))$, then $B_{\pi(i_1),\pi(j_1)} = B_{\tpi(i_2),\tpi(j_2)}$ and $X_{i_1j_1}$ is not independent of $X_{i_2j_2}$. 
In order to decompose $\prod_{i<j} X_{ij}$ as a product over independent randomness, we use the notion of cycle decomposition introduced in \prettyref{sec:orbit}.
Define 
\begin{equation}
\sigma \triangleq \pi^{-1}\circ \tpi,
\label{eq:sigma}
\end{equation}
which is also uniformly distributed on $\calS_n$.
Let $\sigmae$ denote the edge permutation induced by $\sigma$ as in \prettyref{eq:sigmae}, i.e., $\sigmae(i,j)=(\sigma(i),\sigma(j))$.
For each edge orbit $\orbit$ of $\sigmae$, define 
\begin{align}
    X_{\orbit}
    & \triangleq  \prod_{(i,j)\in \orbit}X_{ij} =  \prod_{(i,j)\in \orbit} L \left( A_{ij} , B_{\pi(i)\pi(j)}\right) L\left( A_{ij},B_{\tpi(i)\tpi(j)} \right). \label{eq:xo}
\end{align} 
Importantly, observe that $X_\orbit$ is a function of $\left(A_{ij}, B_{\pi(i)\pi(j)}\right)_{(i,j)\in \orbit}$.
Indeed, since $(\tpi(i),\tpi(j)) = \pi (\sigma(i), \sigma(j))$, or equivalently in terms of edge permutation, 
$\tpie = \pie\circ \sigmae$, 
and $\orbit$ is an orbit of $\sigmae$, we have $\{B_{\pi(i)\pi(j)}\}_{(i,j)\in \orbit}=\{B_{\tpi(i)\tpi(j)}\}_{(i,j)\in \orbit}$. 

Let $\calO$ denote the collection of all edge orbits of $\sigma$. 
Since edge orbits are disjoint, we have 
\begin{align}
\prod_{i<j}X_{ij} = \prod_{\orbit \in \calO} X_{\orbit}.
\label{eq:LR-orbit}
\end{align}
Since $\{A_{ij}\}_{i<j}$ and $\{B_{ij}\}_{i<j}$ are $\iid$ under $\calQ$, we conclude $\{X_{\orbit}\}_{\orbit \in \calO}$ are mutually independent under $\calQ$. 
Therefore, by \prettyref{eq:second1},
\begin{align}
\Expect_{\mathcal{Q}}\left[\left(\frac{\mathcal{P}  \left(A,B\right)}{\mathcal{Q}  \left(A,B\right)}\right)^2\right]
    & =
    \Expect_{\pi \indep{} \tpi}\left[\prod_{\orbit\in\calO}  \Expect_{\left(A,B\right)\sim\calQ} \left[X_{\orbit}\right]\right].
    \label{eq:second3}
\end{align}

Recall that, for the Gaussian Wigner model, $\rho$ denotes the correlation coefficient of edge weights in the planted model $\calP$. 
For \ER graph model, the correlation parameter in the planted model $\calP$ is defined as:
\begin{equation}
\rho \triangleq \frac{\Cov(A_{ij}B_{\pi(i)\pi(j)})}{\sqrt{\Var{(A_{ij})}}\sqrt{\Var{(B_{\pi(i)\pi(j)})}}}=\frac{s(1-p)}{1-ps}.
\label{eq:rho-ER}
\end{equation}


\begin{proposition}\label{prop:cycle}
Fixing $\pi$ and $\tpi$, for any edge orbit $\orbit$ of $\sigma=\pi^{-1} \circ \tpi$, we have  
\begin{itemize}
    \item for Gaussian Wigner models,
    \begin{align}
     \Expect_{\left(A,B\right)\sim\calQ} \left[X_{\orbit}\right] = \frac{1}{1-\rho^{2|\orbit|}}\, , \label{eq:wmcycle}
    \end{align}
    \item for \ER random graphs, 
    \begin{align}
          \Expect_{\left(A,B\right)\sim\calQ} \left[X_{\orbit}\right] = 1+\rho^{2|\orbit|} \, . \label{eq:ercycle}
    \end{align}
\end{itemize}

\end{proposition}
\begin{proof}
Recall from \prettyref{eq:L_a_b} that $L(x,y) = \frac{P(x,y)}{Q(x,y)}= \frac{P(x,y)}{Q(x)Q(y)}$.
This kernel defines an operator as follows:
for any square-integrable function $f$ under $Q$, 
\begin{equation}
(L f) (x) \triangleq \mathbb{E}_{Y\sim Q}\left[L(x,Y) f(Y)\right] = \mathbb{E}_{(X,Y)\sim P}\left[f(Y) \mid X=x\right].
\label{eq:}
\end{equation}
In addition, $L^2=L \circ L$ is given by 
$L^2(x,y) = \Expect_{Z\sim Q}[L(x,Z)L(Z,y)]$ and $L^k$ is similarly defined.
For both Gaussian and Bernoulli model, we have $L(x,y)=L(y,x)$ and hence
$L$ is self-adjoint. Furthermore, since $\iint L(x,y)^2 Q(dx) Q(dy) < \infty$, $L$ is Hilbert-Schmidt.
Thus $L$ is diagonazable with eigenvalues $\lambda_i$'s and the trace of $L$ is given by 
$\mathrm{tr}(L)=\Expect_{Y\sim Q}[L(Y,Y)] = \sum \lambda_i$.

Let $k=|O|$.
To simplify the notation, let $a_i$'s and $b_i$'s be independent sequences of iid random variables drawn from $Q$. 
Since $O$ is an edge orbit of $\sigmae$, we have $\{B_{\pi(i)\pi(j)}\}_{(i,j)\in \orbit}=\{B_{\tpi(i)\tpi(j)}\}_{(i,j)\in \orbit}$ and $(\tpi(i),\tpi(j)) = \pi (\sigma(i), \sigma(j))$.
By \prettyref{eq:xo},
\begin{align*}
\Expect_{A,B\sim \calQ}\left[ X_{\orbit} \right]
    & = \Expect_{A,B\sim \calQ}\left[\prod_{(i,j)\in \orbit}  L\left(A_{ij},B_{\pi(i)\pi(j)}\right) L\left(A_{ij},B_{\tpi(i)\tpi(j)} \right )\right]   \\
    &  = \Expect \qth{
    \prod_{\ell=1}^{k} 
    L\left(a_\ell,b_\ell\right) L\left( a_\ell,b_{(\ell+1) \bmod k}\right)}\\
    &= \Expect \qth{
    \prod_{\ell=1}^{k} 
    L^2\left(b_\ell,b_{(\ell+1) \bmod k}\right)}\\
    & =\mathrm{tr}\left( L^{2k} \right)=\sum \lambda_i^{2k}.
\end{align*}

For Gaussian Wigner model,
$L(x,y)$ given in~\prettyref{eq:Lba-gaussian} is known as Mehler's kernel and can be diagonalized by Hermite polynomials as
$$
L(x,y)= \sum_{i=0}^\infty \frac{\rho^i}{i!} H_i(x) H_i (y),
$$
where $\mathbb{E}_{Y \sim \calN(0,1)}\left[H_i(Y)H_j(Y)\right]=i! \indc{i=j}$~\cite{kibble1945extension}.
It follow that the eigenvalues of operator $L$ are given by $\lambda_i=\rho^i$ for $i \geq 0$ and thus
$
\mathrm{tr}\left( L^{2k} \right)
= \sum_{i=0}^\infty \rho^{2ki}
=\frac{1}{1-\rho^{2k}}.
$

For \ER graphs,
$$
(L f)(x)= \sum_{y \in \{0,1\}}
\frac{P(x,y)}{Q(x)Q(y)} f(y) Q(y) 
=\frac{1}{Q(x)} \sum_{y \in \{0,1\}} P(x,y)f(y).
$$
Thus the eigenvalues of $L$ are given by the eigenvalues of  the following $2 \times 2$ row-stochastic matrix $M$ with rows and columns indexed by $\{0,1\}$ and
$
M(x, y) = \frac{P(x, y)}{Q(x)} \, .
$
Explicitly, by \prettyref{eq:Lba-ER1} we have
\[
M = \begin{pmatrix}
\frac{1-ps(2-s)}{1-ps} &  \frac{ps (1-s)}{1-ps} \\
 1-s & s 
\end{pmatrix} \, .
\]
The eigenvalues of $M$ are $1$ and $\rho=\frac{s(1-p)}{1 - ps}$,
so $
\mathrm{tr} \left(L^{2k}\right) = 1 + \rho^{2 k}.$
\end{proof}

In view of Propositions \ref{prop:cycle}, $\Expect_{\left(A,B\right)\sim\calQ} \left[X_{\orbit}\right] $ decreases when the orbit length $|\orbit|$ increases.
Let $n_k$ denote the total number of $k$-node orbits in the cycle decomposition of node permutation $\sigma$, and 
let $N_k$ denote the total number of $k$-edge orbits in the cycle decomposition of edge permutation $\sigmae$, for $k\in\naturals$. 
For the Gaussian Wigner model, by \prettyref{eq:second3} and \prettyref{eq:wmcycle}, we get 
\begin{align}
     \Expect_{\mathcal{Q}}\left[\left(\frac{\mathcal{P}  \left(A,B\right)}{\mathcal{Q}  \left(A,B\right)}\right)^2\right] =    
     \Expect_{\pi\ci\widetilde{\pi}}\left[\prod_{\orbit\in\calO}\left(\frac{1}{1-\rho^{2|\orbit|}}\right)\right] =\Expect_{\pi\ci\widetilde{\pi}}\left[\prod_{k=1}^{\binom{n}{2}}\left(\frac{1}{1-\rho^{2k}}\right)^{N_k}\right]. \label{eq:secondwm}
\end{align}
For the \ER graphs, by \prettyref{eq:second3} and \prettyref{eq:ercycle}, we get 
\begin{align}
     \Expect_{\mathcal{Q}}\left[\left(\frac{\mathcal{P}  \left(A,B\right)}{\mathcal{Q}  \left(A,B\right)}\right)^2\right] =    
          \Expect_{\pi\ci\widetilde{\pi}}\left[\prod_{\orbit\in\calO}\left(1+\rho^{2|\orbit|}\right)\right] =\Expect_{\pi\ci\widetilde{\pi}}\left[\prod_{k=1}^{\binom{n}{2}}\left(1+\rho^{2k}\right)^{N_k}\right]  .
       \label{eq:seconder}
\end{align}

Let us assume  
\begin{align}
n^2 \rho^6 =o(1), \label{eq:rho_mild_condition}
\end{align}
which is ensured by \prettyref{eq:1lowerbound} for Gaussian model in \prettyref{thm:gw} or \prettyref{eq:2lowerbound} for dense \ER model with $p=n^{-o(1)}$ in \prettyref{thm:er}.

For the Gaussian model, consider the orbits of length $ k \ge  3$. Since $\sum_{k=3}^{\binom{n}{2}} N_k \le \binom{n}{2}$, we have 
\begin{align}
    \prod_{k=3}^{\binom{n}{2}}
    \left( \frac{1}{1-\rho^{2k}} \right)^{N_{k}}
    \le \left( \frac{1}{1-\rho^6} \right)^{\binom{n}{2}}
    =\left( 1+ \frac{\rho^6}{1-\rho^6} \right)^{\binom{n}{2}}
    \le \exp \left( \frac{n^2 \rho^6}{2\left(1-\rho^6\right)} \right) = 1+o(1), \label{eq:1lower2}
\end{align}
where the last equality holds due to \prettyref{eq:rho_mild_condition}.
Moreover, for $1$-orbits and $2$-orbits.
\begin{align*}
    \left(\frac{1}{1-\rho^{2}} \right)^{N_1} \left(\frac{1}{1-\rho^{4}} \right)^{N_2}
     &  \le  \left( 1+ \frac{\rho^2}{1-\rho^{2}} \right)^{N_1} \left( 1+ \left( \frac{\rho^2}{1-\rho^{2}} \right)^2 \right)^{N_2} \nonumber \\
      & \le  \exp \left(  \frac{\rho^2}{1-\rho^{2}}  N_1 + \left( \frac{\rho^2}{1-\rho^{2}} \right)^2  N_2 \right).
\end{align*}      
Therefore,
\begin{align}
     \Expect_{\mathcal{Q}}\left[\left(\frac{\mathcal{P}  \left(A,B\right)}{\mathcal{Q}  \left(A,B\right)}\right)^2\right] 
       \le   \left( 1+o(1) \right) 
      \Expect_{\pi \ci{} \tpi}\left[ \exp \left(  \frac{\rho^2}{1-\rho^{2}}  N_1 + \left( \frac{\rho^2}{1-\rho^{2}} \right)^2  N_2 \right) \right]  \label{eq:secondwm_simple}.
\end{align}

For \ER graphs, analogously, for the orbits of length $ k \ge  3$,
\begin{align}
    \prod_{k=3}^{\binom{n}{2}}\left(1+\rho^{2k}\right)^{N_k} \le \left(1+\rho^6\right)^{\binom{n}{2}}\le \exp \left(\frac{n^2\rho^6}{2}\right) = 1+ o(1), \label{eq:2lower2}
\end{align}
where the last equality holds due to  \prettyref{eq:rho_mild_condition}.
For $1$-orbits and $2$-orbits, 
$
\left(1+\rho^2 \right)^{N_1} \left( 1+\rho^4 \right)^{N_2} \le \exp \left( \rho^2 N_1 + \rho^4 N_2 \right). 
$
Therefore, 
\begin{align}
     \Expect_{\mathcal{Q}}\left[\left(\frac{\mathcal{P}  \left(A,B\right)}{\mathcal{Q}  \left(A,B\right)}\right)^2\right]  \le 
         \left( 1+o(1) \right)  \Expect_{\pi \ci{} \tpi}\left[ \exp \left( \rho^2 N_1 + \rho^4 N_2 \right) \right].
\label{eq:seconder_simple}  
\end{align} 

Next, we bound the contribution of $N_1$ and $N_2$ to the second moment for both models using  the following proposition. The proof, based on Poisson approximation, is deferred till \prettyref{sec:n1n2_weak_detection}. 
\begin{proposition}\label{prop:n1n2_weak_detection}
Assume $\mu, \nu, \tau \ge  0$ such that $\tau^2 =o\left(\frac{1}{n}\right)$, and $\mu b +\nu +2 - \log b \le 0$ for some $1 \le b \le n$ such that $b=\omega(1)$.
\begin{itemize}
    \item If $a=\omega(1)$ and $\nu \le \log(a)-3$,
     \begin{align}
     \Expect_{\pi\indep{}\tpi}\left [ \exp \left( \mu n_1^2 + \nu n_1 +  \tau n_2 + \tau^2 N_2  \right) \indc{a \le n_1 \le b } \right] = o(1). \label{eq:n1n2_a_omega_1}
    \end{align}
    \item If $a=0$ and $\nu=o(1)$,
    \begin{align}
     \Expect_{\pi\indep{}\tpi}\left [ \exp \left( \mu n_1^2 + \nu n_1 +  \tau n_2 + \tau^2 N_2  \right) \indc{a \le n_1 \le b } \right] \leq  1+o(1). \label{eq:n1n2_a_0}
    \end{align}
 In particular, if $ 0 \le \tau \le \frac{2 (\log n -2) }{n}$,  then
    \begin{align}
    \Expect_{\pi\indep{}\tpi}\left [ \exp \left( \tau N_1 + \tau^2 N_2  \right)  \right] 
    = 1+o(1).
    \label{eq:n1n2_main}
    \end{align}
\end{itemize}
\end{proposition}

Finally, we arrive at a condition for bounded second moment, which turns out to be sharp.
\begin{theorem}[Impossibility condition by unconditional second moment method] \label{thm:unconditional_second_moment}
Fix any constant $\epsilon>0$. If 
\begin{align}
\rho^2 \le \frac{ (2-\epsilon) \log n}{n}, \label{eq:cond_unconditional_second_moment}
\end{align}
then for both Gaussian Wigner and \ER graphs,
$\Expect_{\calQ}[(\frac{\calP(A,B)}{\calQ(A,B)})^2]=1+o(1)$, which further implies that 
$\mathrm{TV}(\calP, \calQ) =o(1)$, the impossibility of weak detection.
\end{theorem} 
\begin{proof}
Note that \prettyref{eq:cond_unconditional_second_moment} implies \prettyref{eq:rho_mild_condition}.
Thus, by combining~\prettyref{eq:secondwm_simple} or~\prettyref{eq:seconder_simple} with \prettyref{eq:n1n2_main} in \prettyref{prop:n1n2_weak_detection},
we get $\Expect_{\calQ}[(\frac{\calP(A,B)}{\calQ(A,B)})^2]=1+o(1)$, which yields $\mathrm{TV}(\calP, \calQ)=o(1)$ in view of \prettyref{eq:secondmoment-tv2}.
\end{proof}

\subsection{Obstruction from short orbits}
	\label{sec:obstruction}



The impossibility condition in \prettyref{thm:unconditional_second_moment} is not optimal. In the Gaussian case, \prettyref{eq:cond_unconditional_second_moment} 
differs by a factor of $2$ from the positive result of $\rho^2 \ge \frac{ (4+\epsilon) \log n}{n}$ in \prettyref{thm:gw}. 
For \ER graphs the suboptimality is more severe: 
\prettyref{thm:er} shows that if $nps^2 \left( \log \frac{1}{ p } - 1 + p \right) \ge (2+\epsilon)\log n$, then strong detection is possible.
In the regime of $p=o(1)$, since $\rho=(1+o(1))s$, this translates to the condition $\rho^2 \ge  \frac{(2+\epsilon) \log n}{np \log \frac{1}{p}}$, which differs from \prettyref{eq:cond_unconditional_second_moment} by an unbounded factor. 
This is the limitation of the second moment method, as the condition $\rho^2 \le \frac{ (2-\epsilon) \log n }{n}$  is actually tight for the second moment
to be bounded. When $\rho^2 \ge \frac{ (2+\epsilon) \log n }{n}$, 
the second moment diverges because of certain rare events associated with short orbits in $\sigma=\pi^{-1} \circ \tpi$.
Below we describe the lower bound on the second moment due to short orbits,
which motivates the conditional second moment arguments in Sections \ref{sec:lb-dense} and \ref{sec:lb-sparse} that eventually overcome these obstructions.

Specifically, in view of \prettyref{eq:LR-sq} and \prettyref{eq:LR-orbit}, for both Gaussian and \ER models, the squared likelihood ratio factorizes into products over the edge orbits of $\sigma$:
$$
\left(\frac{\mathcal{P} (A, B)  }{\mathcal{Q} (A, B) }\right)^2 = \Expect_{\pi\ci\widetilde{\pi}} \left[ \prod_{\orbit \in \calO} X_\orbit \right],
$$
where $X_\orbit$ is defined in \prettyref{eq:xo}. 
Since both $\pi$ and $\tpi$ are uniform random permutations, so is $\sigma=\pi^{-1}\circ\tpi$.
For each divisor $k$ of $n$, consider the rare event that $\sigma$ decomposes into $(n/k)$ disjoint $k$-node orbits (i.e.~$n_k=n/k$ and all the other $n_j$'s are zero), which occurs with probability $\frac{1}{ (n/k)! k^{n/k} } \ge  n^{-n/k}$. 
These short node orbits create an abundance of short edge orbits, as 
each pair of distinct $k$-node orbits can form $k$ different $k$-edge orbits.\footnote{For example, $(12)$ and $(34)$ can form two edge orbits $O_{13}$ and $O_{14}$ of length 2; these edge orbits will be referred to as Type-$\sfM$; see \prettyref{sec:classification_edge_orbits} for a full classification of edge orbits.}
Thus, the following lower bound on the second moment ensues
\begin{align}
   \Expect_{(A, B) \sim \mathcal{Q}}\left[\left(\frac{\mathcal{P} (A, B)  }{\mathcal{Q} (A, B) }\right)^2\right] 
	&  = \Expect_{\pi\ci\widetilde{\pi}}\left[\prod_{\orbit\in\calO}\Expect_{\calQ}\left[X_{\orbit}\right] \right]\nonumber \\
	 &   \overset{(a)}{\ge}    \expect{  \left( 1+ \rho^{2k} \right)^{ \binom{n_k}{2} k } }  \ge n^{-n/k}  \left( 1+ \rho^{2k} \right)^{ \binom{n/k}{2} k } \nonumber \\
  &  =   \exp \left( - \frac{n}{k} \log n +\binom{n/k}{2} k  \log \left( 1+ \rho^{2k} \right)  \right),
  \label{eq:second_moment_blow_up_short_orbits}
\end{align} 
where $(a)$ holds because $\Expect_{\calQ}\left[X_{\orbit}\right] \ge 1+ \rho^{2k}$ for each 
$k$-edge orbit $\orbit$ in both Gaussian (\prettyref{eq:wmcycle}) and \ER models (\prettyref{eq:ercycle}).

Consequently, for any $k=o(n)$, 
$$
 \rho^{2k} \ge \frac{ (2+\epsilon) \log n }{n }   
 \quad \Longrightarrow \quad \Expect_{\mathcal{Q}}\left[\left(\frac{\mathcal{P} (A, B)  }{\mathcal{Q} (A, B) }\right)^2\right] \to \infty.
$$ 
In particular, the strongest obstruction comes from $k=1$ (fixed points):
\begin{align*}
\rho^2 \ge \frac{ (2+\epsilon) \log n }{n }  \quad \Longrightarrow \quad  \Expect_{\mathcal{Q}}\left[\left(\frac{\mathcal{P} (A, B)  }{\mathcal{Q} (A, B) }\right)^2\right] \to \infty.
\end{align*}
In this case, the culprit is the rare event of $\pi$ ``colliding'' with $\tpi$ ($\sigma=\text{id}$), which holds with probability $1/n!$ but has an excessive contribution of $(1+\rho^2)^{\binom{n}{2}}$ to the second moment.

In conclusion, the second moment is susceptible to the influence of short edge orbits, for which $\prod_{|O|=k} X_O$ for small $k$ has a large expectation. 
Fortunately, it turns out that the atypically large magnitude of $\prod_{|O|=k} X_\orbit$  can be attributed to certain rare events associated with the intersection graph $A\wedge B^{\pi}$ under the planted model $\calP$. 
This motivates us to condition on some appropriate high-probability event under $\calP$, so that the excessively large magnitude of $\prod_{|O| =k} X_\orbit$ is truncated. 
As we will see in \prettyref{sec:lb-dense}, in the dense regime (including Gaussian Wigner model and dense \ER graphs), 
it suffices to consider $k=1$ and regulate $\prod_{|O|=1} X_\orbit$ by
conditioning on the edge density  for all sufficiently large induced subgraphs of $A\wedge B^{\pi}$ under $\calP$.
In contrast, in the sparse regime, we need to consider all edge orbits up to length $k=\Theta(\log n)$ for which more sophisticated techniques are called for, as we will see in \prettyref{sec:lb-sparse}.  

\section{Conditional Second Moment Method: Dense regime}
\label{sec:lb-dense}


In this section, we improve \prettyref{thm:unconditional_second_moment} by applying the conditional second moment method. 
The proof of the sharp threshold for the Gaussian model is given in full details in \prettyref{sec:gaussian_second_moment}. 
The proof for dense \ER graphs uses similar ideas but is technically more involved and hence deferred to  \prettyref{sec:er_dense}. 
We start by describing the general program of conditional second moment method.
Note that sometimes certain rare events under $\calP$ can cause the second moment to explode, while $\mathrm{TV}(\calP, \calQ) $ remains bounded away from one.  To circumvent such
catastrophic events, we can compute the second moment conditioned on events that are typical under $\calP$. More precisely, 
given an event $\calE$ such that $\calP(\calE)=1+o(1)$, define the planted model conditional on $\calE$:
\begin{align*}
    \calP'\left(A,B, \pi\right)
    & \triangleq \frac{\calP\left(A,B,\pi\right)\indc{\left(A,B,\pi\right)\in \calE}}{\calP\left( \calE\right)} = \left( 1+o\left(1\right)\right)\calP\left(A,B,\pi\right)\indc{\left(A,B,\pi\right)\in \calE},
\end{align*}
the last equality holds because $\calP(\calE)=1+o(1)$.
Then the likelihood ratio between the conditioned planted model $\calP'$ and the null model $\calQ$ is given by 
\begin{align*}
    \frac{\calP'\left(A,B\right)}{\calQ\left(A,B\right)}
    = \frac{ \int \calP'\left(A,B,\pi\right) \mathrm{d} \pi } {\calQ\left(A,B\right)}
    & = \left( 1+o\left(1\right)\right) \int \frac{\calP\left(\pi\right)\calP\left(A,B \mid \pi\right)\indc{\left(A,B,\pi\right)\in \calE}}{\calQ(A,B)}\mathrm{d}\pi\\
    & = \left( 1+o\left(1\right)\right)  \Expect_{\pi}\left[\frac{\calP\left(A,B\mid\pi\right)}{\calQ\left(A,B\right)}\indc{\left(A,B,\pi\right)\in \calE}\right].
\end{align*}
By the same reasoning that led to  \prettyref{eq:second3}, 
the conditional second moment is given by 
\begin{align}
     \Expect_{\calQ}\left[\left(\frac{\calP'\left(A,B\right)}{\calQ\left(A,B\right)}\right)^2\right] 
    & = \left( 1+o\left(1\right)\right)\Expect_{\pi\indep{}\widetilde{\pi}}\left[\Expect_{Q}\left[\frac{\calP\left(A,B\mid\pi\right)}{\calQ\left(A,B\right)}\frac{\calP\left(A,B\mid\tpi\right)}{\calQ\left(A,B\right)}\indc{\left(A,B,\pi\right)\in \calE} \indc{\left(A,B,\tpi\right)\in \calE}\right]\right] \nonumber \\
		    & = \left( 1+o\left(1\right)\right)\Expect_{\pi\indep{}\widetilde{\pi}}\left[\Expect_{Q}\left[ 
				\prod_{O\in\calO}	X_O \indc{\left(A,B,\pi\right)\in \calE} \indc{ \left(A,B,\widetilde{\pi} \right)\in \calE}
				\right]\right], \label{eq:condition_second_moment}
\end{align}
where the last equality follows from the decomposition \prettyref{eq:LR-orbit} over edge orbits $O \in \calO$ of $\sigma=\pi^{-1}\circ \tpi$. 
Compared to the unconditional second moment in \prettyref{eq:LR-sq}, the extra indicators in \prettyref{eq:condition_second_moment} will be useful for ruling out those rare events causing the second moment to blow up. 

We caution the reader that, crucially, the conditioning event $\calE$ must be measurable with respect to the observed and the latent variables $(A, B, \pi)$.
 Thus we \emph{cannot} rule out the rare event that $\pi$ is close to its independent copy $\tilde{\pi}$
so that $\sigma=\pi^{-1}\circ\tpi$ induces a proliferation of short edge orbits. Instead, as we will see, by truncating certain rare events associated with the intersection graph $A \wedge B^{\pi}$, the excessively large magnitude of $\prod_{|O|=k} X_O$ can be regulated for small $k$.

By the data processing inequality of total variation, we have
$$
\mathrm{TV}\left(\calP(A,B),\calP'(A,B) \right) \le \mathrm{TV} \left(\calP(A,B,\pi),\calP'(A,B,\pi) \right) = \calP\left((A,B,\pi) \not \in \calE \right) = o(1).
$$
Combining this with the second moment bound \prettyref{eq:secondmoment-tv1}--\prettyref{eq:secondmoment-tv2} and applying the triangle inequality, we arrive at the following conditions for non-detection:
\begin{align}
&\Expect_{\calQ}\left[\left( \frac{\calP'(A,B)}{\calQ(A,B)} \right)^2\right] = O(1) \quad \Longrightarrow \quad  \mathrm{TV}(\calP(A,B),\calQ(A,B)) \le  1- \Omega(1) \label{eq:cond_second_moment_TV_impossibilit_strong} \\
& \Expect_{\calQ}\left[\left( \frac{\calP'(A,B)}{\calQ(A,B)} \right)^2\right] = 1+ o(1) \quad \Longrightarrow \quad  \mathrm{TV}(\calP(A,B),\calQ(A,B)) =  o(1). 
\label{eq:cond_second_moment_TV_impossibilit_weak}
 \end{align}



\subsection{Sharp threshold for the Gaussian model} \label{sec:gaussian_second_moment} 

%
%

In this section, we improve over the impossibility condition $\rho^2 \le \frac{(2-\epsilon) \log n}{n}$ 
established in~\prettyref{thm:unconditional_second_moment}, showing that if $\rho^2 \le \frac{(4-\epsilon) \log n}{n} $,
then weak detection is impossible. This completes the impossibility proof of \prettyref{thm:er} for the Gaussian model.

Before the rigorous analysis, we first explain the main intuition.
Let $F$ denotes the set of fixed points of $\sigma= \pi^{-1}\circ\tpi$, so that $|F|=n_1$. 
Let 
\[
\calO_1 = \binom{F}{2},
\]
which is a subset of fixed points of the edge permutation (cf.~\prettyref{eq:N1N2}).
As argued in~\prettyref{sec:obstruction},  the unconditional second moment blows up when $\rho^2 \ge \frac{(2+\epsilon) \log n}{n}$
due to the obstruction of fixed points of $\sigma$, or more precisely, an atypically large magnitude of 
$\prod_{\orbit \in \calO_1} X_O$. 
By \prettyref{eq:Lba-gaussian} and  \prettyref{eq:X_ij}, 
\begin{align}
\prod_{\orbit \in \calO_1} X_O & = \prod_{ i<j \in F } X_{ij}   \nonumber \\
& = \left(1-\rho^2\right)^{-\binom{n_1}{2} } \exp \sth{ \frac{1}{1-\rho^2} \pth{- \rho^2  \sum_{ i<j \in F } \left(A_{ij}^2+ B_{\pi(i) \pi(j) }^2\right) 
+ 2 \rho \sum_{ i<j \in F }A_{ij}  B_{\pi(i)\pi(j)} }}.
\label{eq:X_O_F_gaussian}
\end{align}
 
Recall that for any $S \subset [n]$, $ e_{A\wedge B^{\pi}}(S) = \sum_{ i<j \in S } A_{ij}  B_{\pi(i)\pi(j)}$ as defined in \prettyref{eq:intersection}.
To truncate  $\prod_{ i<j \in F } X_{ij}$, one natural idea is to condition on the typical value of $e_{A\wedge B^{\pi}}(F)$
under the planted model $\calP$ when $|F|=n_1$ is large.
More specifically, 
for each $S\subset[n]$, define
\begin{align*}
\calE_S \triangleq \bigg\{ 
& (A, B, \pi): \sum_{i<j \in S} A_{ij}^2 ,  \; \sum_{i<j \in S} B_{\pi(i) \pi(j) }^2   \ge  \binom{|S|}{2} - C n^{3/2},  e_{A\wedge B^{\pi}}(S) \le  \rho \binom{|S|}{2} + C n^{3/2} \bigg\}
\end{align*}
where $C$ is an absolute constant.
We will condition on the event 
\begin{equation}
\calE \triangleq \bigcap_{S \subset [n]: |S| \ge  n/2} \calE_S.
\label{eq:calE-gaussian}
\end{equation}
This event $\calE$ can be shown to hold with high probability under the planted model $\calP$. 
Note that here in order to truncate $\prod_{ i<j \in F } X_{ij}$, $\calE$ is defined as the intersection of $\calE_S$ over all subsets $S$ with $|S| \ge n/2$,
so that it implies $\calE_F$ when $|F| \ge n/2$. The reason that we cannot condition on $\calE_F$ directly is because the set of fixed points $F$
depends on $\sigma=\pi^{-1}\circ\tpi$ rather than $\pi$ alone, and thus is not measurable with respect to $(A,B,\pi)$. 

Let 
\begin{equation}
\zeta=  \rho \binom{n_1}{2} + C n^{3/2} 
\label{eq:zeta}
\end{equation}
 When $n_1 \ge n/2$, we have $\zeta= \rho \binom{n_1}{2} (1+o(1))$. Furthermore, 
on the event $\calE$, it follows from \prettyref{eq:X_O_F_gaussian} and $\rho=o(1)$ that
\begin{align*}
\Expect_\calQ \left[ \prod_{i<j \in F } X_{ij}   \Indc_{\calE} \right]
\le &   \exp \sth{ - (1+o(1)) \rho^2 \binom{n_1}{2} } \Expect_\calQ\qth{\exp \sth{ \frac{2\rho}{1-\rho^2} e_{A\wedge B^{\pi}}(F)   }
 \indc{ e_{A\wedge B^{\pi}}(F) \le  \zeta   }   }\\
 \le & \exp \sth{ \frac{1+o(1) }{2} \rho^2  \binom{n_1}{2} },
\end{align*}
where the last inequality is by evaluating the truncated MGF of $e_{A\wedge B^{\pi}}(F)$ (see \prettyref{eq:truncated-MGF} below). Note that without the truncation $e_{A\wedge B^{\pi}}(F) \le \zeta$, 
we recover the unconditional bound $\Expect_\calQ \left[ \prod_{i<j \in F } X_{ij} \right] =\exp \sth{ (1+o(1) ) \rho^2  \binom{n_1}{2} }$.
Thus, the conditional bound improves over the unconditional one by a multiplicative factor of $2$ in the exponent.

Finally, to ensure the second moment after conditioning is $1+o(1)$, analogous to \prettyref{eq:second_moment_blow_up_short_orbits}, 
in the extreme case of $n_1=n$, we need to ensure 
\[
\frac{1}{n!} \exp \sth{ \frac{1+o(1) }{2} \rho^2 \binom{n}{2} } = \exp\sth{ - (1+o(1)) n \log n + \frac{ 1+o(1) }{4} \rho^2 n^2  }  = o(1),
\]
which corresponds precisely to the desired condition $\rho^2 \le \frac{(4-\epsilon) \log n}{n}$.

\medskip
Next, we proceed to the rigorous proof. As the impossibility of weak detection when $\rho^2 \le \frac{\log n}{n}$ has already been shown in  \prettyref{thm:unconditional_second_moment}, 
henceforth we only need to focus on
$$
\frac{\log n}{n} \le \rho^2 \le \frac{(4-\epsilon)\log n}{n}.
$$
The following lemma proves that $\calE$ holds with high probability 
under the planted model $\calP$.
\begin{lemma}\label{lmm:condition_Gaussian}
It holds that $\calP(  (A, B, \pi) \in \calE )=1 - e^{-\Omega(n)}$.
\end{lemma}
\begin{proof}
Fix an integer $n/2 \le k \le n$ and let $m=\binom{k}{2}$. 
Let
$t =c \left( \sqrt{m \log (1/\delta)} + \log (1/\delta) \right)$,
for a universal constant $c$ and a parameter $\delta$ to be specified later.

Fix a subset $S\subset [n]$ with $|S|=k$. 
Using the Hanson-Wright inequality given in~\prettyref{lmm:hw}, with probability at least $1-3\delta$, 
\begin{align}
\sum_{i<j \in S} A_{ij}^2 \ge m-t, \; \sum_{i<j \in S} B_{\pi(i) \pi(j) }^2 
\ge m-t, \; e_{A\wedge B^{\pi}}(S) = \sum_{i<j \in S} A_{ij}B_{\pi(i)\pi(j)} \le \rho m +t. \label{eq:deviation_Hanson_wright}
\end{align}


Now, there are $\binom{n}{k}$ different choices of $S \subset[n]$ with $|S|=k$. 
Thus by choosing $1/\delta= 2^{k} \binom{n}{k}$ and applying the union bound,
we get that 
with probability at least $1- 3\sum_{k=n/2}^n 2^{-k}=1-e^{-\Omega(n)}$,  \prettyref{eq:deviation_Hanson_wright}
holds uniformly for all  $S \subset[n]$ with $|S|=k$ and all  $n/2 \le k \le n$.
By definition and the fact that $k \ge n/2$, 
$1/\delta \le 2^k \left( \frac{en}{k}\right)^k \le (4e)^k$,
and thus $t \le c \left( \sqrt{m k \log (4e) } + k\log (4e) \right) = O\left(n^{3/2}\right)$.
\end{proof}

Now, let us compute the conditional second moment. By \prettyref{lmm:condition_Gaussian}, it follows from \prettyref{eq:condition_second_moment} that 
\begin{align*}
     \Expect_{\calQ}\left[\left(\frac{\calP'(A,B) }{\calQ(A, B) }\right)^2\right] 
    & =  \left( 1+o\left(1\right)\right) \Expect_{\pi\ci\widetilde{\pi}}\left[
    \Expect_{\calQ}
    \left[\prod_{\orbit \in \calO } X_{\orbit}\indc{\left(A,B,\pi\right)\in \calE}\indc{ \left(A,B,\widetilde{\pi} \right)\in \calE}\right]\right].
\end{align*}
\medskip 
To proceed further, we fix $\pi, \tpi$ and separately consider the following two cases.

{\bf Case 1}:  $n_1 \le n/2$. In this case, we simply drop the indicators and use the unconditional second moment:
\begin{align*}
 \Expect_{\calQ}
    \left[\prod_{\orbit \in \calO } X_{\orbit}\indc{\left(A,B,\pi\right)\in \calE}\indc{ \left(A,B,\widetilde{\pi} \right)\in \calE}\right]
     \le  \Expect_{\calQ}
    \left[\prod_{\orbit \in \calO } X_{\orbit} \right] 
     = \prod_{\orbit \in \calO } \frac{1}{1-\rho^{2|\orbit|}},
  \end{align*} 
where the last equality follows from \prettyref{eq:wmcycle}.

\medskip
{\bf Case 2}:  $n_1 > n/2$.
 In this case, 
 \begin{align*}
  \Expect_{\calQ}
    \left[\prod_{\orbit \in \calO } X_{\orbit}\indc{\left(A,B,\pi\right)\in \calE}\indc{ \left(A,B,\widetilde{\pi} \right)\in \calE}\right] 
   & \overset{(a)}{\le}  \Expect_{\calQ}
    \left[\prod_{\orbit \in \calO } X_{\orbit} \indc{ (A,B,\pi) \in \calE_F  }\right]  \\
   & \overset{(b)}{=}
 \Expect_{\calQ} \left[ \prod_{\orbit \in \calO_1}   X_{\orbit} \indc{ (A,B,\pi) \in \calE_F  } \right] \prod_{\orbit \notin \calO_1 } \Expect_{\calQ}\left[X_{\orbit} \right] \\
 & \overset{(c)}{=}
  \Expect_{\calQ}   \left[  \prod_{ i<j \in F} X_{ij} \indc{ (A,B,\pi) \in \calE_F  } \right]\prod_{\orbit \notin \calO_1 } \frac{1}{1-\rho^{2|\orbit|}},
\end{align*}
where $(a)$ is due to the definition \prettyref{eq:calE-gaussian}, $\calE \subset \calE_F$ when $n_1 \ge n
/2$; $(b)$ holds because $X_\orbit$ is a function of $(A_{ij}, B_{\pi(i) \pi(j) })_{(i,j) \in \orbit}$ that are independent across different $\orbit \in \calO$, 
and $\indc{ (A,B,\pi) \in \calE_F  }$ only depends on $\left\{ (A_{ij}, B_{\pi(i) \pi(j) })_{(i,j) \in \orbit}: \orbit \in \calO_1\right\}$;
$(c)$ follows from \prettyref{eq:wmcycle}. 

On the event $\calE_F$, 
we have 
$$
\sum_{ i< j \in F} A_{ij}^2 \ge (1+o(1)) \binom{n_1}{2}, \; \sum_{ i<j \in F} B_{\pi(i) \pi(j) }^2  \ge (1+o(1)) \binom{n_1}{2}, \; e_{A\wedge B^{\pi}}(F) \le (1+o(1)) \rho \binom{n_1}{2},
$$ 
where we used the fact that $n^{3/2} = o(\rho n_1^2)$ in view of assumption $\rho^2 \ge \frac{\log n}{n}$ and $n_1>n/2$. 

It follows from \prettyref{eq:X_O_F_gaussian}
that 
\begin{align*}
& \Expect_{\calQ}   \left[  \prod_{ i<j \in F} X_{ij} \indc{ (A,B,\pi) \in \calE_F  } \right] \\
 &= \left(1-\rho^2\right)^{-\binom{n_1}{2} } \Expect_{\calQ}   \left[  \exp \sth{ \frac{1}{1-\rho^2} \pth{- \rho^2  \sum_{ i<j \in F }\left (A_{ij}^2+ B_{\pi(i) \pi(j) }^2\right) 
+ 2 \rho e_{A\wedge B^{\pi}}(F) }} \indc{ (A,B,\pi) \in \calE_F  }  \right]\\
& \le \left(1-\rho^2\right)^{-\binom{n_1}{2} } 
 \exp \sth{  - \frac{(2 +o(1)) \rho^2}{1-\rho^2}  \binom{n_1}{2} }
 \Expect_{\calQ}   \left[  \exp \sth{ \frac{2\rho e_{A\wedge B^{\pi}}(F) }{1-\rho^2 }   }  \indc{ e_{A\wedge B^{\pi}}(F) \le \zeta} \right] ,
\end{align*}
where $e_{A\wedge B^{\pi}}(F) = \sum_{ i<j \in F } A_{ij}  B_{\pi(i)\pi(j)}$ and 
$\zeta= \rho \binom{n_1}{2} (1+o(1))$. 

Let $\beta= \frac{2\rho}{1-\rho^2 }$. Then for any $\lambda \in [0,1]$,
\begin{align}
 \Expect_{\calQ}   \left[  \exp \sth{ \frac{2\rho e_{A\wedge B^{\pi}}(F) }{1-\rho^2 }   }  \indc{ e_{A\wedge B^{\pi}}(F) \le \zeta} \right]
& \le  \Expect_{\calQ}   \left[  \exp \sth{ \beta \left( \lambda e_{A\wedge B^{\pi}}(F) + (1-\lambda) \zeta \right) }   \right] \nonumber \\
& = \exp \sth{ \beta (1-\lambda) \zeta - \frac{1}{2} \binom{n_1}{2}  \log  \left(1-\beta^2\lambda^2 \right)   }, \label{eq:truncated-MGF}
\end{align}
where the equality uses the MGF expression in \prettyref{eq:MGF_XY}.
Choosing\footnote{This choice is motivated by choosing $\lambda$ to minimize
$-\beta \lambda \zeta +  \frac{1}{2} \binom{n_1}{2}  \beta^2 \lambda^2$, the first-order approximation of the exponent in \prettyref{eq:truncated-MGF}, leading to $\lambda^*= \zeta/ [ \binom{n_1}{2} \beta ] = (1+o(1)) (1-\rho^2)/2$. }
 $\lambda=(1-\rho^2)/2$ in \prettyref{eq:truncated-MGF}, we obtain
 \begin{align*}
 \exp \sth{ \beta (1-\lambda) \zeta - \frac{1}{2} \binom{n_1}{2}  \log  \left(1-\beta^2\lambda^2 \right)   }
 =
\exp \sth{ (\beta-\rho) \zeta - \frac{1}{2} \binom{n_1}{2}  \log  \left(1-\rho^2 \right)   }.
\end{align*}
Combining the last three displayed equations yields that 
\begin{align*}
\Expect_{\calQ}   \left[  \prod_{ i<j \in F} X_{ij} \indc{ (A,B,\pi) \in \calE_F  } \right] 
& \le
 \exp \sth{  - \frac{2\rho^2 (1+o(1))}{1-\rho^2}  \binom{n_1}{2}  + 
 (\beta-\rho) \zeta - \frac{3}{2} \binom{n_1}{2}  \log  \left(1-\rho^2 \right)  } \\
 & =   \exp \sth{    \frac{(1+o(1)) \rho^2}{2} \binom{n_1}{2} } \le  \exp \sth{    \frac{(1+o(1)) \rho^2 n_1^2}{4}  } ,
\end{align*}
where the equality holds under the assumption that $\rho=o(1)$ so that
$\log (1-\rho^2)=-(1+o(1))\rho^2$.

\medskip

Combining the two cases yields   that
\begin{align*}
 \Expect_{\calQ}\left[\left(\frac{\calP'(A,B) }{\calQ(A, B) }\right)^2\right] 
& \le   (1+o(1)) \expect{ \prod_{\orbit \in \calO } \frac{1}{1-\rho^{2|\orbit|}}  \indc{n_1 \le n/2} } \\
& + (1+o(1)) \expect{ \prod_{\orbit \notin \calO_1 } \frac{1}{1-\rho^{2|\orbit|} } \exp \sth{    \frac{(1+o(1)) \rho^2n_1^2 }{4}  } \indc{n_1 > n/2} }.
\end{align*}

Let $\tau=\frac{\rho^2}{1-\rho^2}$.  Note that 
\begin{align*}
\prod_{\orbit \notin \calO_1 } \frac{1}{1-\rho^{2|\orbit|}} =
\left( \frac{1}{1-\rho^2} \right)^{n_2} \prod_{k\ge 2}  
\left( \frac{1}{1-\rho^{2k} } \right)^{N_k} 
& =(1+o(1)) \left( \frac{1}{1-\rho^2} \right)^{n_2} \left( \frac{1}{1-\rho^{4} } \right)^{N_2}, \\
& \le (1+o(1)) \exp\left( \tau n_2 + \tau^2 N_2 \right) ,
\end{align*}
where the first equality follows from \prettyref{eq:N1N2}, the second equality holds by \prettyref{eq:1lower2} under the assumption $ \rho^2 \le (4-\epsilon)\log n/n $,
and the last inequality holds because 
$\frac{1}{1-\rho^2} = 1+\tau  \le \exp\left( \tau \right)$
and $\frac{1}{1-\rho^{4} }\le  1+ \tau^2 \le  \exp\left( \tau^2 \right)$. 
Similarly, 
\begin{align*}
\prod_{\orbit \in \calO_1 } \frac{1}{1-\rho^{2|\orbit|}} =
\left( \frac{1}{1-\rho^2} \right)^{ \binom{n_1}{2} }  \le \exp \left( \tau n_1^2/2 \right). 
\end{align*}

Hence,
\begin{align*}
 \Expect_{\calQ}\left[\left(\frac{\calP'(A,B) }{\calQ(A, B) }\right)^2\right] 
& \le   (1+o(1)) \expect{  \exp \left( \tau \left( n_1^2/2 +n_2\right)  +\tau^2 N_2 \right )  \indc{n_1 \le n/2} } \\
& + (1+o(1)) \expect{ \exp \left( \tau n_2 +  \tau^2 N_2 \right)   
\exp \sth{    \frac{(1+o(1)) \rho^2 n_1^2 }{4}} \indc{n_1 > n/2} }.
\end{align*}

We upper bound the two terms separately. 
To bound the first term,  we apply \prettyref{eq:n1n2_a_0} in \prettyref{prop:n1n2_weak_detection} with $\mu=\tau/2$, $\nu=0$, $a=0$, and $b=n/2$. 
Recall that  $\tau=\frac{\rho^2}{1-\rho^2}$. By assumption $\rho^2 \le (4-\epsilon)\log n / n$, we have $\tau^2=o\left(\frac{1}{n}\right)$ and  $\mu b + 2 -\log b=\frac{\rho^2 n}{4(1-\rho^2)} + 2 - \log (n/2) \le 0$ for all sufficiently large $n$. Thus it follows from \prettyref{eq:n1n2_a_0} in \prettyref{prop:n1n2_weak_detection} that 
\begin{align*}
\expect{ \exp \left( \tau (n_1^2/2 + n_2) +  \tau^2 N_2 \right)  \indc{n_1 \le n/2} } \le 1+o(1).
\end{align*}
To bound the second term, we apply \prettyref{eq:n1n2_a_omega_1} in \prettyref{prop:n1n2_weak_detection}  with $\mu=\frac{(1+o(1)) \rho^2}{4}$, $\nu=0$, $a=\frac{n}{2}$, and $b =n$.
Recall that  $\tau=\frac{\rho^2}{1-\rho^2}$. By assumption $n\rho^2 \le (4-\epsilon)\log n$, we have $\tau^2=o\left(\frac{1}{n}\right)$  and $\mu b +\nu +2 -\log b =\frac{(1+o(1)) \rho^2 n}{4} + 2 -\log n \le 0$ for sufficiently large $n$. Thus  it follows from \prettyref{eq:n1n2_a_omega_1} in \prettyref{prop:n1n2_weak_detection} that 
\begin{align*}
 \expect{ \exp \left( \tau n_2 +  \tau^2 N_2 \right)   
\exp \sth{    \frac{(1+o(1)) \rho^2 n_1^2 }{4}} \indc{n_1 > n/2} } = o(1).
\end{align*}

%
Combining the upper bounds for the two terms, we conclude that
$ \Expect_{\calQ}\left[\left(\frac{\calP'(A,B) }{\calQ(A, B) }\right)^2\right] =1+o(1)$ under the assumption that 
$\rho^2 \le (4-\epsilon)\log n/n$. Thus $\mathrm{TV}(\calP,\calQ)=o(1)$ in view of \prettyref{eq:cond_second_moment_TV_impossibilit_weak}.


%
%
%
%

\section{Conditional Second Moment Method: Sparse regime}
\label{sec:lb-sparse}


We focus on the \ER model in the sparse regime of $p=n^{-\Omega(1)}$. 
The impossibility condition previously obtained in \prettyref{thm:unconditional_second_moment} by the unconditional second moment simplifies to $s^2   \le (2-\epsilon) \frac{\log n}{n}$. In this section, we significantly improve this result
by showing that if
\begin{align}
s^2 \le \frac{1- \omega( n^{-1/3})}{np} \wedge 0.01, \label{eq:conditional_er_sparse}
\end{align}
then strong detection is impossible. Moreover, if both $s=o(1)$ and \prettyref{eq:conditional_er_sparse} hold, then weak detection is impossible.

Analogous to the proof for the dense case in \prettyref{sec:lb-dense} (see also \prettyref{sec:er_dense}), we will apply the 
conditional second moment method. However, the argument in the sparse case is much more sophisticated for the following reason. 
In the dense regime (both Gaussian and \ER graph with $p=n^{-o(1)}$), we have shown that the main contribution to the second moment is due to fixed points of $\sigma=\pi^{-1}\circ \tpi$, which can be regulated by conditioning on the edge density of large induced subgraphs in the intersection graph. 
For sparse \ER graphs with $p=n^{-\Omega(1)}$, 	we need to control the contribution of not just fixed points, but all edge orbits of length up to $k=\Theta(\log n)$.
Indeed, as argued in~\prettyref{sec:obstruction}, the unconditional second moment blows up 
when $\rho^{2k} \ge \frac{(2+\epsilon)\log n}{n}$ due to the obstructions from the $k$-edge orbits, or more
precisely, an atypically large magnitude of $\prod_{|\orbit|=k} X_\orbit$. Note that $\rho=  \frac{s(1-p)}{1-ps}=(1+o(1)) s$ in the sparse case. 
Therefore, to show the desired condition~\prettyref{eq:conditional_er_sparse}, we need to regulate 
$\prod_{|\orbit|=k} X_\orbit$ beyond $k=1$ by proper conditioning. 
In fact, for $p=\Theta(1/n)$, since \prettyref{eq:conditional_er_sparse} reduces to $\rho \leq 0.1$, it is necessary to control all $k$ up to $\Theta(\log n)$.

To this end, the crucial observation is as follows. 
We call a given edge orbit $\orbit $ of $\sigma=\pi^{-1}\circ \tpi$ \emph{complete} if it is a subgraph of the intersection graph $A \wedge B^{\pi}$, i.e.~$\orbit \subset E\left( A \wedge B^{\pi}\right) $.
For each complete orbit $O$, we have $A_{ij}=B_{\pi(i)\pi(j)}=B_{\tpi(i)\tpi(j)}=1$ for all $(i,j)\in O$ and hence, by \prettyref{eq:Lba-ER1} and \prettyref{eq:X_ij}, $X_{ij}=L(1,1)^2 = 1/p^2$, so that $X_O$  attains its maximal possible value, namely
\begin{equation}
X_O = \left(\frac{1}{p}\right)^{2|\orbit|}, \quad \forall O \subset E\left( A \wedge B^{\pi}\right).
\label{eq:complete-orbit}
\end{equation}
For incomplete orbits, it is not hard to show (see \prettyref{prop:con2} below) that 
\[
\Expect_{\calQ}\left[ X_{\orbit} \mid \orbit \not\subset A \wedge B^{\pi} \right] \le 1.
\] 
Hence, 
the key is to control the contribution of complete edge orbits $\orbit$ that are subgraphs of $A \wedge B^{\pi}$. 
Crucially, under the assumption of \prettyref{thm:er} in the sparse regime,  $nps^2$ is sufficiently small so that
$A \wedge B^{\pi}$ is subcritical and a pseudoforest (each component having at most one cycle) with high probability under the planted model $\calP$.
This global structure significantly limits the possible configurations of complete edge orbits, since many patterns of co-occurrence of edge orbits in $A \wedge B^{\pi}$ are forbidden. 
Motivated by this observation, we truncate the likelihood ratio by conditioning on the global event that $A \wedge B^{\pi}$
is a pseudoforest. 
Finally, in order to show the conditional second moment is bounded under the desired condition~\prettyref{eq:conditional_er_sparse},
we carefully control
the co-occurrence of edge orbits in $A \wedge B^{\pi}$ under the pseudoforest constraint, which involves a delicate enumeration of pseudoforests that can be assembled from edge orbits.

Next, let us proceed to the rigorous analysis. Define 
$$
\calE \triangleq \{(A,B,\pi): A\land B^{\pi} \text{ is a pseudoforest}\}. 
$$
Note that $A \wedge B^{\pi}\sim\calG(n,ps^2)$ under the planted model $\calP$. The following result 
shows that in the subcritical case $A \wedge B^{\pi}$ is a pseudoforest.

\begin{lemma}[{\cite[Lemma 2.10]{frieze2016introduction}}]\label{lmm:cond_high_prob_sparse_ER}
If $nps^2 \le 1-\omega\left(n^{-1/3}\right)$, then $\calP\left( (A, B, \pi) \in \calE \right) = 1 - o\left( \frac{1}{n^3}\right)$ as $n \to \infty$. 
\end{lemma}

Recall from \prettyref{eq:LR-sq} and \prettyref{eq:LR-orbit} in \prettyref{sec:second-moment-calculation}  the following representation of the squared likelihood ratio 
\begin{equation}
\pth{\frac{\calP(A,B)}{\calQ(A,B)}}^2 = \Expect_{\pi\ci\widetilde{\pi}}\qth{\prod_{\orbit \in \calO }X_{\orbit}},
\label{eq:squaredLR}
\end{equation}
where for each edge orbit $O$ of $\sigma=\pi^{-1}\circ\tpi$, 
\[
X_O = \prod_{ij \in O} X_{ij}, \quad X_{ij}=  L\left(A_{ij} , B_{\pi(i)\pi(j)}\right) L\left(A_{ij}, B_{\tpi(i)\tpi(j)} \right),
\]
with $L(\cdot,\cdot)$ is defined in \prettyref{eq:Lba-ER1}. 
In order to decompose \prettyref{eq:squaredLR} further, let us introduce the following key definitions.
Recall from \prettyref{sec:orbit} that $O_i$ denotes the node-orbit of $i$ (under the node permutation $\sigma$) and $O_{ij}$ denotes the edge-orbit of $(i,j)$ (under the edge permutation $\sigmae$).
Fix some $k$ to be specified later.

\begin{itemize}
	\item Define $\calO_k$ as the set of edge orbits of length at most $k$ that are formed by node orbits with length at most $k$, that is, 
	\begin{align*}
	   \calO_k = \{O_{ij}: \left| O_i \right| \le k, \left|O_j \right| \le k, \left| O_{ij} \right|\le k, 1 \le i < j \le n\}.
   	\end{align*}
\item Define $\calJ_k$ as the set of edge orbits $\orbit \in \calO_k$ that are subgraphs of $A \wedge B^{\pi}$, \ie, 
\begin{align*}
\calJ_k & = \{\orbit\in \calO_k :  A_{ij}=1,B_{\pi(i)\pi(j)}=1, \forall (i,j)\in \orbit\} \nonumber \\
& =  \{\orbit\in \calO_k :  A_{ij}=1,B_{\tpi(i)\tpi(j)}=1, \forall (i,j)\in \orbit\},
\end{align*}
where the second equality holds because $\{B_{\pi(i)\pi(j)}\}_{(i,j)\in \orbit}=\{B_{\tpi(i)\tpi(j)}\}_{(i,j)\in \orbit}$. 
\item Define 
\begin{equation}
H_k = \bigcup_{\orbit\in \calJ_k} \orbit.
\label{eq:Hk}
\end{equation}
\end{itemize}
Note that while $\calO_k$ depends only on the random permutation $\sigma=\pi^{-1}\circ\tpi$, both $\calJ_k$ and $H_k$ depend in addition on the random graph $A \wedge B^{\pi}$.

As will be discussed at length in \prettyref{sec:classification_edge_orbits}, each edge orbit can be viewed as a subgraph of the complete graph $K_n$. 
Different edge orbits are by definition edge disjoint, and the union of all edge orbits is the edge set of $K_n$. 
We shall call a graph an \emph{orbit graph} if it is union of edge orbits. Importantly, by definition, the orbit graph $H_k$ is a subgraph of $A\wedge B^{\pi}$.


To  compute the conditional second moment, by \prettyref{lmm:cond_high_prob_sparse_ER}, it follows from \prettyref{eq:condition_second_moment} that 
\begin{align}
    \Expect_{\calQ}\left[\left(\frac{\calP'\left(A,B\right)}{\calQ\left(A,B\right)}\right)^2\right] 
    & = \left( 1+o\left(1\right)\right) \Expect_{\pi\ci\widetilde{\pi}}\left[
    \Expect_{\calQ}
    \left[\prod_{\orbit \in \calO }X_{\orbit}\indc{\left(A,B,\pi\right)\in \calE}\indc{ \left(A,B,\widetilde{\pi} \right)\in \calE}\right]\right] \nonumber \\
    & \le \left( 1+o\left(1\right)\right)  \Expect_{\pi\ci\widetilde{\pi}}\left[
    \Expect_{\calQ}
    \left[\prod_{\orbit \in \calO }X_{\orbit}\indc{H_k \text{ is a pseudoforest}}\right]\right],
    \label{eq:beforedecompose}
\end{align}
where the last inequality holds because on the event that $A\land B^{\pi}$ is a pseudoforest, its subgraph $H_k$ is also one. 


To further upper bound the right hand side of \prettyref{eq:beforedecompose}, we decompose the product over edge orbits into three terms:
\begin{align*}
\prod_{O\in\calO} X_O = \prod_{O\notin\calO_k} X_O \times \prod_{O\in\calO_k \backslash \calJ_k} X_O \times \prod_{O\in\calJ_k} X_O
\end{align*}
which correspond to the contributions of \emph{long orbits}, \emph{short incomplete orbits} (that are not subgraphs of $A \wedge B^{\pi}$), and \emph{short complete orbits} (that are subgraphs), respectively. 
As shown earlier in \prettyref{eq:complete-orbit}, for each complete edge orbit $O$, we have $X_O=(1/p)^{2|O|}$. Therefore in view of \prettyref{eq:Hk}, the collective contribution of short complete orbits are 
\begin{equation}
\prod_{O\in\calJ_k} X_O = \pth{\frac{1}{p}}^{2e(H_k)}. 
\label{eq:XO-Jk}
\end{equation}
Thus, fixing $\sigma=\pi^{-1}\circ \tpi$, we have
\begin{align}
&  \Expect_{\calQ}
    \left[\prod_{\orbit \in \calO }X_{\orbit}\indc{H_k \text{ is a pseudoforest}}\right] \nonumber  \\
   & =\Expect_{\calQ}\left[\prod_{\orbit\notin \calO_k}X_{\orbit}\right] \Expect_{\calQ}
    \left[\prod_{\orbit\in \calO_k}X_{\orbit}\indc{H_k \text{ is a pseudoforest}}\right]  \nonumber \\
    & = \Expect_{\calQ}\left[\prod_{\orbit\notin \calO_k}X_{\orbit}\right] 
    \Expect_{\calJ_k }\left[ \left( \frac{1}{p} \right)^{2 e(H_k)} \indc{H_k \text{ is a pseudoforest}}\Expect_{\calQ}\left[\prod_{\orbit\in \calO_k\backslash \calJ_k }X_{\orbit} \; \Big | \; \calJ_k \right]\right],  \label{eq:thm11}
\end{align}
where the first equality holds 
because $\{X_{\orbit}\}_{\orbit\in\calO}$ are mutually independent and $\calJ_k \subset \calO_k$, 
so that $\{X_{\orbit}\}_{\orbit \in \calO \backslash \calO_k}$ is independent of $\{X_{\orbit}\}_{\orbit \in \calO_k}$ and the event that $H_k$ is a pseudoforest;
the second equality holds because $H_k $ is measurable with respect to $\calJ_k$.

The contributions of long orbits and incomplete orbits can be readily bounded as follows whose proofs are deferred till Sections \ref{sec:con1} and \ref{sec:con2}.

\begin{proposition}[Long orbits]\label{prop:con1}
Fix any $\sigma=\pi^{-1}\circ \tpi$. For any $k\in \naturals$,  
\begin{align*}
    \Expect_{\calQ}\left[\prod_{\orbit \in \calO \backslash \calO_k} X_{\orbit}\right] \le \left(1+\rho^k\right)^{\frac{n^2}{k}}. 
\end{align*}
\end{proposition}

\begin{proposition}[Incomplete orbits]\label{prop:con2}
Fix any $\sigma=\pi^{-1}\circ \tpi$. If  $p\le 1/2$ and $s\le 1/2$, then
\begin{align*}
    \Expect_{\calQ}\left[\prod_{\orbit \in \calO_k \backslash \calJ_k }X_{\orbit} \; \Bigg| \; \calJ_k \right] \le 1. 
\end{align*}
\end{proposition}

Applying \prettyref{prop:con1} and  \prettyref{prop:con2} to \prettyref{eq:thm11}, we get that for  any $\sigma=\pi^{-1}\circ \tpi$, 
\begin{align}
 \Expect_{\calQ}
    \left[\prod_{\orbit \in \calO }X_{\orbit}\indc{H_k \text{ is a pseudoforest}}\right]
    \le   \left(1+\rho^k\right)^{\frac{n^2}{k}} \Expect_{\calJ_k }\left[ \left( \frac{1}{p} \right)^{2 e(H_k)} \indc{H_k \text{ is a pseudoforest}} \right]
    \label{eq:short_orbit_subgraph}
\end{align}

It remains to further upper bound the RHS of \prettyref{eq:short_orbit_subgraph}.
Let $\calH_k$ denote the set of all orbit graphs that consist of edge orbits in $\calO_k$ and are pseudoforests -- we call such graphs \emph{orbit pseudoforests}.
As such $\calH_k$ depends only on $\sigma$ but not the graph $A$ and $B$. Therefore, 
\begin{align}
  \Expect_{\calJ_k}\left[ \left(\frac{1}{p}\right)^{2 e(H_k)} \indc{H_k\text{ is a pseudoforest}} 
   \right] 
    & = \sum_{H\in \calH_k}
    \calQ\left(H_k = H \right)    \left(\frac{1}{p}\right)^{2 e(H )} \indc{H \text{ is a pseudoforest}} \nonumber \\
    & \le \sum_{H \in \calH_k} s^{2 e(H)},
		\label{eq:gf-pseudoforest}
\end{align}
where the last step holds because 
$$ 
\calQ \left(H_k = H\right)  \le  
\calQ \left(
A_{ij}=1,B_{\pi(i)\pi(j)}=1,\forall (i,j)\in E(H) \right)= (ps)^{2|e(H)|}.
$$


In view of \prettyref{eq:gf-pseudoforest}, to further upper bound the second moment, it boils down to bounding the the generating function of the class $\calH_k$ of orbit pseudoforests. This is done in the following theorem in terms of the cycle type of $\sigma$. The proof involves a delicate enumeration of orbit pseudoforests, which constitutes the most crucial part of the analysis. 
We note that if we ignore the orbit structure and treat $\calH_k$ as arbitrary pseudoforests, the resulting bound will be too crude to be useful.

\begin{theorem}[Generating function of orbit pseudoforests]\label{thm:jk}
For any $k\in\naturals$, $\sigma=\pi^{-1}\circ \tpi$, and any $s \in [0,1]$,
\begin{align}
  \sum_{H \in \calH_k} s^{2 e(H)}
    \le \prod_{m=1}^k 
    \left( 1 +  s^{m} n_{m}   \indc{m:\mathrm{even}}   +2 s^{2m}  
    \sum_{\ell =1}^m \ell n_{\ell}+ s^{4m} m n_{2m}  \indc{2m \le k}  \right)^{n_m},\label{eq:jk}
\end{align} 
where $n_m$ is the number of $m$-node orbits in $\sigma=\pi^{-1}\circ\tpi$ for $1\le m\le k$.
\end{theorem}

Combining \prettyref{eq:beforedecompose}, \prettyref{eq:short_orbit_subgraph}, \prettyref{eq:gf-pseudoforest},  and \prettyref{eq:jk}, 
we get that 
\begin{align}
     &\Expect_{\calQ}\left[\left(\frac{\calP'\left(A,B\right)}{\calQ\left(A,B\right)}\right)^2\right]  \nonumber \\
     & \le \left( 1+ o(1) \right) \left(1+\rho^k\right)^{\frac{n^2}{k}} 
     \Expect_{\pi\ci\widetilde{\pi}}\left[ \prod_{m=1}^k \left( 1 +  s^m n_{m}  \indc{m:\mathrm{even}}  
     + 2 s^{2m} \sum_{\ell \le m} \ell n_{\ell}+ s^{4m} m n_{2m}  \indc{2m \le k}  \right)^{n_m}\right], \label{eq:short_orbit_subgraph_3}
\end{align}
which is further bounded by the next result.

\begin{proposition}\label{prop:jkbound} 
Suppose $k(\log k)^4 = o(n)$.  If $s\le 0.1$, 
            \begin{align}
                \Expect_{\pi\ci\widetilde{\pi}}\left[\prod_{m=1}^k \left( 1 +  s^m n_{m}  \indc{m:\mathrm{even}}  
                 + 2 s^{2m} \sum_{\ell \le m} \ell n_{\ell}+ s^{4m} m n_{2m}  \indc{2m \le k}  \right)^{n_m} \right] = O(1).
            		\label{eq:jkbound1}
            \end{align}
Furthermore, if $s= o\left(1\right)$,
            \begin{align}
            \Expect_{\pi\ci\widetilde{\pi}}\left[ \prod_{\ell=m}^k \left( 1 +  s^{m} n_{m}  \indc{m:\mathrm{even}}  +  2 s^{2m} \sum_{\ell \le  m} \ell n_{\ell}+ s^{4m} m n_{2m}  \indc{2m \le k} \right)^{n_m} \right] = 1+ o(1)  . \label{eq:jkbound2}
            \end{align} 
\end{proposition}



The proof of \prettyref{prop:jkbound} is involved and deferred to 
\prettyref{sec:jkbound}. To provide some concrete idea, the following simple calculation shows that $s=o(1)$ is necessary for \prettyref{eq:jkbound2} to hold. Indeed, consider $k=1$ for which the LHS reduces to
$
  \expect{\left( 1+ 2s^2n_1\right)^{n_1} }. 
$
By Poisson approximation (see \prettyref{app:perm}), replacing $n_1$ by $\Pois(1)$ yields 
$$
\expect{\left( 1+2 s^2 n_1\right)^{n_1} }
\approx e^{-1} \sum_{a=0}^\infty \left( 1+ 2s^2 a\right)^{a} \frac{1}{a!} 
\ge  e^{-1} \sum_{a=0}^\infty \left( 1+ 2s^2 \right)^{a} \frac{1}{a!} 
=e^{2s^2},
$$
which is $1+o(1)$ if and only if $s=o(1)$. 
To evaluate the full expectation in \prettyref{eq:jkbound2}, note that even if we use Poisson approximation to replace $n_m$'s by independent Poissons, the terms inside the product over $[k]$ are still dependent. To this end, we carefully partition the product into disjoint parts, and recursively peeling off the expectation backwards.

We are now ready to complete the proof of \prettyref{thm:er} in the sparse case.

\begin{proof}[Proof of \prettyref{thm:er}: Impossibility Result in Sparse Regime]
Let $k =3 \log n$. 
If $s\le \frac{1}{2}$,
then $\frac{n^2 s^k}{k} = o(1)$ and thus 
$$\left(1+\rho^{k}\right)^{\frac{n^2}{k}} \le \exp\left(\frac{n^2 \rho^{k}}{k}\right) \le  \exp\left(\frac{n^2 s^{k}}{k}\right) = 1+o(1).$$
Note that $ k (\log k)^4 = o(n)$. 
Combining \prettyref{eq:short_orbit_subgraph_3} with  \prettyref{eq:jkbound1} and \prettyref{eq:jkbound2} yields that 
 $ \Expect_{\calQ}[(\frac{\calP'(A,B)}{\calQ(A,B)})^2] =O(1)$ for $s\le 0.1$
 and $\Expect_{\calQ}[(\frac{\calP'(A,B)}{\calQ(A,B)})^2]=1+o(1)$ for $s=o(1)$, which completes the proof
 in view of \prettyref{eq:cond_second_moment_TV_impossibilit_strong} and \prettyref{eq:cond_second_moment_TV_impossibilit_weak}. 
\end{proof}

The remainder of this section is organized as follows. To prepare for the proof of \prettyref{thm:jk}, we study the graph structure and the classification of edge orbits in \prettyref{sec:classification_edge_orbits}. 
An equivalent representation of orbit graphs as backbone graphs is given in \prettyref{sec:backbone} to aid the enumeration argument. 
As a warm-up, we first enumerate orbit forests (orbit graphs that are forests) and bound their generating function in \prettyref{sec:enumeration_forst}. 
The more challenging case of orbit pseudoforests is tackled in \prettyref{sec:enumeration_pseudoforst}, completing the proof of \prettyref{thm:jk}. Sections \ref{sec:con1}--\ref{sec:jkbound} contain the proofs of Propositions \ref{prop:con1}--\ref{prop:jkbound}.

\subsection{Classification of edge orbits} \label{sec:classification_edge_orbits}

\begin{table}[!ht]
\centering
\begin{tabular}{ >{\centering\arraybackslash}m{1in}  | >{\centering\arraybackslash}m{2in}  |  >{\centering\arraybackslash}m{1in} }
Type & Edge orbit & Orbit graph \\
\hline
\multirow{2}{*}{\raisebox{-0.2in}{$\sfM$} } 
& (13,24) & \tikz[scale=1,font=\scriptsize]{
	\draw (0,0) node (a) [vertexdot] {};
	\draw (0,-0.5) node (b) [vertexdot] {};
	\draw (1,0) node (c) [vertexdot] {};
	\draw (1,-0.5) node (d) [vertexdot] {};
	\node at (a) [left] {$1$};
	\node at (b) [left] {$2$};
	\node at (c) [right] {$3$};
	\node at (d) [right] {$4$};
	\draw[Medge] (a)--(c)(b)--(d);
	}  \\ 
	\cline{2-3}
 & (14,23) & \tikz[scale=1,font=\scriptsize]{
	\draw (0,0) node (a) [vertexdot] {};
	\draw (0,-0.5) node (b) [vertexdot] {};
	\draw (1,0) node (c) [vertexdot] {};
	\draw (1,-0.5) node (d) [vertexdot] {};
	\node at (a) [left] {$1$};
	\node at (b) [left] {$2$};
	\node at (c) [right] {$3$};
	\node at (d) [right] {$4$};
	\draw[Medge] (a)--(d) (b)--(c);
	}  \\ 
\hline
\multirow{4}{*}{\raisebox{-1.5in}{$\sfB$}} 
& (15,26,17,28) & \tikz[scale=1,font=\scriptsize]{
	\draw (0,0.5) node (a) [vertexdot] {} ;
	\draw (0,0) node (b) [vertexdot] {};
	\draw (1,1) node (c) [vertexdot] {};
	\draw (1,0.5) node (d) [vertexdot] {};
    \draw (1,0) node (e) [vertexdot] {};
	\draw (1,-0.5) node (f) [vertexdot] {};
	\node at (a) [left] {$1$};
	\node at (b) [left] {$2$};
	\node at (c) [right] {$5$};
	\node at (d) [right] {$6$};
	\node at (e) [right] {$7$};
	\node at (f) [right] {$8$};
	\draw[Bedge] (a)--(c) (a)--(e) (b)--(d) (b)--(f);
	} \\ 
\cline{2-3}
& (16,27,18,25) & \tikz[scale=1,font=\scriptsize]{
	\draw (0,0.5) node (a) [vertexdot] {} ;
	\draw (0,0) node (b) [vertexdot] {};
	\draw (1,1) node (c) [vertexdot] {};
	\draw (1,0.5) node (d) [vertexdot] {};
    \draw (1,0) node (e) [vertexdot] {};
	\draw (1,-0.5) node (f) [vertexdot] {};
	\node at (a) [left] {$1$};
	\node at (b) [left] {$2$};
	\node at (c) [right] {$5$};
	\node at (d) [right] {$6$};
	\node at (e) [right] {$7$};
	\node at (f) [right] {$8$};
	\draw[Bedge] (a)--(d) (a)--(f) (b)--(c) (b)--(e);
	} \\ 
\cline{2-3}
	& (35,46,37,48) & \tikz[scale=1,font=\scriptsize]{
	\draw (0,0.5) node (a) [vertexdot] {} ;
	\draw (0,0) node (b) [vertexdot] {};
	\draw (1,1) node (c) [vertexdot] {};
	\draw (1,0.5) node (d) [vertexdot] {};
    \draw (1,0) node (e) [vertexdot] {};
	\draw (1,-0.5) node (f) [vertexdot] {};
	\node at (a) [left] {$3$};
	\node at (b) [left] {$4$};
	\node at (c) [right] {$5$};
	\node at (d) [right] {$6$};
	\node at (e) [right] {$7$};
	\node at (f) [right] {$8$};
	\draw[Bedge] (a)--(c) (a)--(e) (b)--(d) (b)--(f);
	} \\ 
\cline{2-3}
& (36,47,38,45) & \tikz[scale=1,font=\scriptsize]{
	\draw (0,0.5) node (a) [vertexdot] {} ;
	\draw (0,0) node (b) [vertexdot] {};
	\draw (1,1) node (c) [vertexdot] {};
	\draw (1,0.5) node (d) [vertexdot] {};
    \draw (1,0) node (e) [vertexdot] {};
	\draw (1,-0.5) node (f) [vertexdot] {};
	\node at (a) [left] {$3$};
	\node at (b) [left] {$4$};
	\node at (c) [right] {$5$};
	\node at (d) [right] {$6$};
	\node at (e) [right] {$7$};
	\node at (f) [right] {$8$};
	\draw[Bedge] (a)--(d) (a)--(f) (b)--(c) (b)--(e);
	} \\ 
	\hline
	$\sfC$
& (56,67,78,85) & \tikz[scale=1,font=\scriptsize]{
	\draw (0,1) node (c) [vertexdot] {};
	\draw (0,0.5) node (d) [vertexdot] {};
    \draw (0,0) node (e) [vertexdot] {};
	\draw (0,-0.5) node (f) [vertexdot] {};
	\node at (c) [left] {$5$};
	\node at (d) [left] {$6$};
	\node at (e) [left] {$7$};
	\node at (f) [left] {$8$};
	\draw[Cedge] (c)--(d) (d)--(e) (e)--(f);
	\draw (c) edge[blue, bend left, thick] (f);
	}  \\ 
\hline
\multirow{3}{*}{\raisebox{-0.5in}{$\sfS$}} 
& (12) & \tikz[scale=1,font=\scriptsize]{
	\draw (0,0.5) node (a) [vertexdot] {};
	\draw (0,0) node (b) [vertexdot] {};
	\node at (a) [left] {$1$};
	\node at (b) [left] {$2$};
	\draw[Sedge] (a)--(b);
	}  \\ 
	\cline{2-3}
& (34) & \tikz[scale=1,font=\scriptsize]{
	\draw (0,0.5) node (a) [vertexdot] {};
	\draw (0,0) node (b) [vertexdot] {};
	\node at (a) [left] {$3$};
	\node at (b) [left] {$4$};
	\draw[Sedge] (a)--(b);
	}  \\ 
	\cline{2-3}
 & (57,68) & \tikz[scale=1,font=\scriptsize]{
	\draw (0,1) node (c) [vertexdot] {};
	\draw (0,0.5) node (d) [vertexdot] {};
    \draw (0,0) node (e) [vertexdot] {};
	\draw (0,-0.5) node (f) [vertexdot] {};
	\node at (c) [left] {$5$};
	\node at (d) [left] {$6$};
	\node at (e) [left] {$7$};
	\node at (f) [left] {$8$};
	\draw (c) edge[bend left, thick] (e) (d) edge[bend left, thick] (f);
	}\\
	\hline
\end{tabular}
\caption{Edge orbits corresponding to the node permutation $\sigma=(12)(34)(5678)$. When representing an edge orbit in cycle notation, each edge $(i,j)$ is abbreviated as $ij$. As a convention, nodes in each node orbit are vertically aligned and arranged in the order of the permutation $\sigma$. For edge orbits, type $\sfM$ are in green, type $\sfB$ in red, type $\sfC$ in blue, and type $\sfS$ in black. 
}

\label{tab:orbits}
\end{table}

To prove \prettyref{thm:jk}, we are interested in orbit graphs consisting of short edge orbits, and the main task lies in enumerating those that are pseudoforests.
To this end, we need to understand the graph structure of edge orbits.

Throughout this subsection, fix a node permutation $\sigma$.
For a given edge $(i,j)$, its edge orbit can be viewed a graph with vertex set $O_i\cup O_j$ and edge set $\orbit_{ij}$. 
Let $|O_i|=\ell$ and $|O_j|=m$.
Each edge orbit can be classified into the following four categories (see \prettyref{tab:orbits} for a concrete example). 
\begin{description}
    \item[Type $\sfM$ (Matching):] $i$ and $j$ belong to different node orbits of the same length.
		In this case, $|\orbit_{ij}|= m$ and $O_{ij}$ is a perfect matching.
		We call such $O_{ij}$ an $\sfM_m$ edge orbit (or a \emph{matching}).
		Furthermore, for two distinct node orbits $O$ and $O'$ of length $m$, the total number of possible $\sfM_m$ edge orbit is $m$.

    \item[Type $\sfB$ (Bridge):] $i$ and $j$ belong to different node orbits of different lengths.
		Without loss of generality, assume that the orbit of $i$ is shorter than that of $j$, i.e. $\ell<=m$. 
		In this case, 
		let $M=\lcm(\ell,m)$. Then $|\orbit_{ij}|= M$ and $O_{ij}$ consists of $\frac{\ell m}{M}$ vertex-disjoint copies of the complete bipartite graphs $K_{M/\ell,M/m}$.
		We call such edge orbit a $\sfB_{m,\ell}$ edge orbit (or a \emph{bridge}).
		Furthermore, for two node orbits $O$ and $O'$ with $|O|=\ell<|O'|=m$, the total number of possible bridges is $\frac{\ell m}{M}$.

		Of special interest is the case where $\ell$ is a divisor of $m$ and the orbit is $\ell K_{1,\frac{m}{\ell}}$ (i.e. $\ell$ copies of $\frac{m}{\ell}$-stars).		
		These are the only bridges that are cycle-free; otherwise the bridge contains a component with \emph{at least two} cycles. 
		This observation is useful for the enumeration argument in Sections \ref{sec:enumeration_forst} and	\ref{sec:enumeration_pseudoforst} under constraints on the number of cycles.

		
		\item[Type $\sfC$ (Cycle):] $i$ and $j$ belong to the same node orbit of length $m$ and $j \neq \sigma^{m/2}(i)$. 
		In this case, $|\orbit_{ij}|= m$ and $O_{ij}$ is an $m$-cycle.
		We call such $O_{ij}$ a $\sfC_m$ edge orbit (or a \emph{cycle}), and there are a total number $\floor{\frac{m-1}{2}}{}$
		of them for the same node orbit.

		

    \item[Type $\sfS$ (Split):] $i$ and $j$ belong to the same node orbit (of even length $m$) and $j = \sigma^{m/2}(i)$.		
		In this case, $|\orbit_{ij}|= m/2$ and $O_{ij}$ is a perfect matching.
		We call such $O_{ij}$ an $\sfS_m$ edge orbit (or a \emph{split}).
		Clearly, for each node orbit of even length, there is a unique way for it to split into an $\sfS_m$ edge orbit.
				

\end{description}
In summary, matchings and bridges are edge orbits formed by two distinct node orbits, which are bipartite graphs with vertex sets $O_i$ and $O_j$.
Cycles and splits are edge orbits formed by a single node orbit $O_i$, which can either form a full cycle or split into a perfect matching.


\subsection{Orbit graph and backbone graph}
	\label{sec:backbone}

Every orbit graph $H$ can be equivalently and succinctly represented as a 
\emph{backbone} graph $\Gamma$ defined as follows.

\begin{definition}[Backbone graph]
\label{def:backbone}
Given an orbit graph $H$, its \emph{backbone} graph is an undirected labeled multigraph, 
whose nodes and edges (referred to as \emph{giant nodes} and \emph{giant edges}) correspond to node orbits and edge orbits in $H$, respectively.
Each giant node carries a binary label (represented as shaded or non-shaded) indicating whether the node orbit forms a Type $\sfS$ edge orbit (split) or not.
Each giant edge carries a label (an integer) encoding the specific realization of the edge orbit.
Specially,
\begin{itemize}
	\item A Type $\sfS$ edge orbit (split) is represented by a shaded giant node. 
	\item A Type $\sfC_m$ edge orbit (cycle) is represented by a self-loop,
	whose edge label takes values in $\left[\floor{\frac{m-1}{2}}{}\right]$. 
\item A Type $\sfM_m$ edge orbit (matching) is represented by a giant edge between two $m$-node orbits, with edge label taking values in  $[m]$.
\item A Type $\sfB_{m,\ell}$ edge orbit (bridge) is represented by a giant edge between a $\ell$-node orbit and a $m$-node orbit ($\ell<m$), with edge label taking values in  $\left[\frac{\ell m}{\lcm(\ell,m)}\right]$.
\end{itemize}

\end{definition}

See \prettyref{fig:example} for an example of an orbit graph and its corresponding  backbone graph.
As a convention, for backbone graph, the labeled giant edges representing Type $\sfM$, Type $\sfB$, and Type $\sfC$ edge orbits are colored green, red, and blue, respectively. 
Each shaded giant node represents a Type $\sfS$ edge orbit. For convenience, the number inside each giant node represents the length of its corresponding node orbit. 

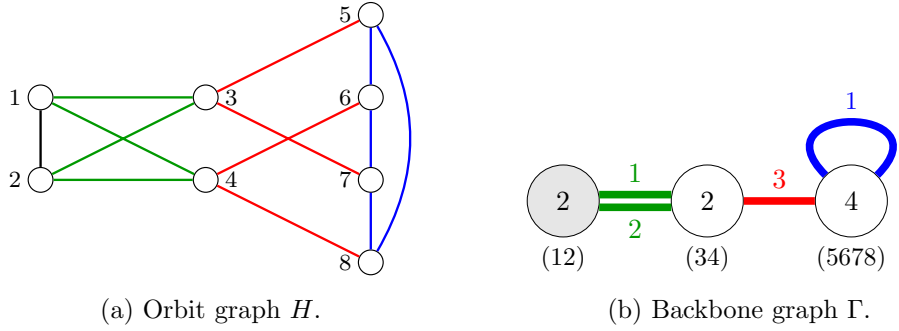
\begin{figure}[H]
\centering
\begin{tabular}{C{.4\textwidth}C{.4\textwidth}}
\subcaptionbox{Orbit graph $H$.}{
\resizebox{0.35\textwidth}{!}{%
\centering
\begin{tikzpicture}[scale=1,font=\scriptsize]
    \draw (0,1) node (a) [nodedot] {} ;
	\draw (0,0) node (b) [nodedot] {} ;
	\draw (2,1) node (c) [nodedot] {} ;
	\draw (2,0) node (d) [nodedot] {} ;
	\draw (4,2) node (e) [nodedot] {} ;
	\draw (4,1) node (f) [nodedot] {} ;
    \draw (4,0) node (g) [nodedot] {} ;
	\draw (4,-1) node (h) [nodedot] {} ;
	\node at (a) [left=0.1] {$1$};
	\node at (b) [left=0.1] {$2$};
	\node at (c) [right=0.1] {$3$};
	\node at (d) [right=0.1] {$4$};
	\node at (e) [left=0.1]{$5$};
	\node at (f) [left=0.1] {$6$};
	\node at (g) [left=0.1] {$7$};
	\node at (h) [left=0.1] {$8$};
	\draw[Medge] (a)--(c) (b)--(d) (a)--(d) (b)--(c);
	\draw[Bedge]  (c)--(e) (c)--(g) (d)--(f) (d)--(h);
	\draw[Cedge] (e)--(f) (f)--(g) (g)--(h);
	\draw (e) edge[blue, bend left, thick] (h);
	\draw[Sedge] (a)--(b);
\end{tikzpicture}
}}&
\subcaptionbox{Backbone graph $\Gamma$.}{
    \resizebox{0.35\textwidth}{!}{%
    \centering
    \begin{tikzpicture}
		\draw (0,0) node (a) [shadedgiantnode] {$2$};
	\draw (2,0) node (b)[unshadedgiantnode]{$2$};
	\draw (4,0) node (c)[unshadedgiantnode] {$4$};
	\draw (b.170)  edge[auto=right, green!60!black, line width=3.0pt] node {$1$} (a.10);
	\draw (a.-10)  edge[auto=right, green!60!black, line width=3.0pt] node {$2$} (b.190);
	\draw (c) edge[auto=right, red , line width=3.0pt] node {$3$}  (b); 
	\path (c) edge[my loop, blue, line width=3.0pt] node[above]  {$1$} (c);
	\draw(0,-0.8) node (c) {\small $(12)$};
	\draw(2,-0.8) node (c) {\small $(34)$};
	\draw(4,-0.8) node (c) {\small $(5678)$};
    \end{tikzpicture}
}}
\end{tabular}
\caption{Example of an orbit graph and its corresponding backbone graph for $\sigma = (12)(34)(5678)$. 
The labels of giant edges are determined based on the enumeration of edge orbits in \prettyref{tab:orbits}. 
For instance, the two green giant edges correspond to the two Type $\sfM$ (perfect matchings) between node orbits $(12)$ and $(34)$, and the red giant edge corresponds to the Type $\sfB$ edge orbits (bridge) between node orbits $(34)$ and $(5678)$. 
}
\label{fig:example}
\end{figure}

Recall that $\calH_k$ denotes the collection of orbit pseudoforests consisting of edge orbits of length at most $k$ formed by node orbits of size at most $k$. 
To enumerate $H \in \calH_k$, it is equivalent to enumerating the corresponding backbone graph $\Gamma$. 
To facilitate the enumeration, we introduce the following definitions:
\begin{itemize}
	\item Let $S_m$ denote the set of giant nodes corresponding to $m$-node orbits. 

\item 
Let  $\Gamma_m=\Gamma[S_m]$ denote the subgraph of $\Gamma$ induced by node set $S_m$ for $1 \le m \le k$.
Let $\Gamma_{ m, \ell }= \Gamma[ S_m, S_\ell ]$ denote the (bipartite) subgraph of $\Gamma$ induced by edges between 
$S_m$ and $S_\ell$, for $1 \le \ell <m \le k$. 
Each giant edge in $\Gamma_{m,\ell} $ corresponds to a $\sfB_{m,\ell}$ edge orbit (bridge). 
\item A  connected component of $\Gamma_m$ is called  \emph{plain} if it contains no split and is not incident to any bridge in $\cup_{\ell<m}\Gamma_{m,\ell}$.

\end{itemize}

Following \cite[page 112]{janson2011random}, 
we define the \emph{excess} of a graph $G$, denoted by $\ex(G)$, as its number of edges minus its number of nodes.
Given a connected component $C$ in $\Gamma_m$, let $H_C$ denote the orbit graph consisting of edge orbits (including splits, matchings, and cycles) 
in $C$, as well as bridges in $\cup_{\ell<m}\Gamma_{m,\ell}$ that are incident to $C$.
The following two operations can be recursively applied to $C$ to increase $\ex(H_C)$:
\begin{enumerate}[label=(O\arabic*)]
\item \label{O:split} Adding one
split in $C$  increases $\ex(H_C)$ by $m/2$;
\item \label{O:1bridge} Adding one $\sfB_{m,\ell}$ bridge ($\ell<m$) to $C$ increases $\ex(H_C)$ by at least $\lcm(\ell,m)-\ell$.
\end{enumerate}

In addition, we need the following fact about the excess of an orbit graph:
\begin{lemma}\label{lmm:plain}
For any connected component $C$ in $\Gamma_m$, $\ex(H_C)\ge  -m$, where the equality holds if and only if $C$ is a plain tree component in $\Gamma_m$.
\end{lemma}
\begin{proof}
Given a connected component $C$ in $\Gamma_m$, let $a$ and $b$ denote the total number of giant edges and giant nodes in $C$, respectively. If $C$ is a plain tree component, we have $a +1 = b$. Since each giant edge in $\Gamma_m$ represents an $m$-edge orbit, and each giant node represent an $m$-node orbit, we have $\ex(H_C) = am - bm =-m$. By \ref{O:split}, \ref{O:1bridge}, and 
the fact that adding one self-loop in $C$ increases $\ex(H_C)$ by $m$, we have $\ex(H_C) \ge  -m$, 
where the equality holds if and only if $C$ does not contain any split or self-loop and is not incident to any bridge in $\cup_{\ell< m}\sfB_{m,\ell}$
that is, $C$ is a plain tree component in $\Gamma_m$. 
\end{proof}
As we will see next, the pseudoforest (forest) constraint of $H$ restricts the possible configurations of $\Gamma_m$ and forbids certain operations on its components (which would otherwise generate too many cycles).

\subsection{Warm-up: Generating function of orbit forests}\label{sec:enumeration_forst}

Fix  $\sigma =\pi^{-1}\circ\tpi$ and recall that $n_m$ denotes the number of $m$-node orbits in $\sigma$.
Our enumeration scheme crucially exploits the classification of edge orbits and orbit graphs in \prettyref{sec:classification_edge_orbits} and the representation of orbit graphs as backbone graphs introduced in \prettyref{sec:backbone}. 
As a warm-up, in this section we bound the generating function of orbit forests, which is much simpler than orbit pseudoforests. 
Restricting the summation to the set $\calF_k$ of orbit forests, a strict subset of $\calH_k$, we show the following improved version of \prettyref{eq:jk}:
\begin{align}
\sum_{H \in \calF_k } s^{2 e(H)} 
\le &\prod_{1\le m\le k} \left( 1+  s^m \indc{m:\mathrm{even}} + s^{2m} \sum_{\ell \le m} \ell n_\ell \right)^{n_m}.
\label{eq:jk-forest}
\end{align}

When the orbit graph $H$ is a forest, its corresponding backbone graph $\Gamma$ must satisfy the following four conditions:
\begin{enumerate}[label=(T\arabic*)]
 \item \label{pt:forest} For each $1 \le m \le k$, $\Gamma_m$ is a forest with simple edges (of multiplicity $1$);
		\item \label{pt:forest-divisor} For each $1 \le \ell < m \le k$, $\Gamma_{m,\ell}$ is empty unless $\ell$ is a divisor of $m$; 
   
		\item \label{pt:forest-noC} There is no  self-loop;

     \item \label{pt:forest-sb} For each $1 \le m \le k$,  each component of $\Gamma_m$ either contains at most $1$  split
     or 
     is incident to at most $1$ bridge in $\cup_{\ell<m} \Gamma_{ m,\ell}$, 
      but not both.

\end{enumerate}
Otherwise, $H$ contains at least one cycle. Indeed, \ref{pt:forest}-\ref{pt:forest-noC} can be readily verified based on the classification of edge orbits and orbit graphs in \prettyref{sec:classification_edge_orbits}.
Suppose the condition in \ref{pt:forest-sb} does not hold.  Then by \ref{O:split}, \ref{O:1bridge} and \prettyref{lmm:plain}, there exists a component $C$ in $\Gamma_m$ such that $\ex(H_C)\ge  0$, contradicting $H$ being a forest.
See~\prettyref{fig:tree_forbidden_pattern} for an illustration of forbidden patterns that
violate   \ref{pt:forest-sb} for $m=4$ and $\ell=2$. 
\begin{figure}[H]
\centering
\begin{tabular}{C{.3\textwidth}C{.3\textwidth}C{.3\textwidth}}
\subcaptionbox{A component in $\Gamma_4$ is incident to $2$ bridges in $\Gamma_{4,2}$.
}%
[0.9\linewidth]{
\resizebox{0.20\textwidth}{!}{%
\centering
\begin{tikzpicture}[scale=1,font=\large]
    \draw (0,1) node (a) [nodedot] {} ;
	\draw (0,0) node (b) [nodedot] {} ;
	\draw (2,2) node (c) [nodedot] {} ;
	\draw (2,1) node (d) [nodedot] {} ;
    \draw (2,0) node (e) [nodedot] {} ;
	\draw (2,-1) node (f) [nodedot] {} ;
	\draw (4,1) node (g) [nodedot] {} ;
	\draw (4,0) node (h) [nodedot] {} ;
	\draw[Bedge]  (a)--(c) (b)--(d) (a)--(e) (b)--(f) 
	(g)--(c) (h)--(d) (g)--(e) (h)--(f) ;
	\draw (0,-2) node (A)[unshadedgiantnode] {$2$};
	\draw (2,-2) node (B)[unshadedgiantnode] {$4$};
	\draw (4,-2) node (C)[unshadedgiantnode] {$2$};
	\draw[Bgiantedge]  (A)--(B) (C)--(B);
    \end{tikzpicture}
}}&
\subcaptionbox{A component in $\Gamma_4$ contains $2$ splits.
}[0.9\linewidth]{
\resizebox{0.20\textwidth}{!}{%
\centering
\begin{tikzpicture}[scale=1,font=\large]
	\draw (0,2) node (a) [nodedot] {} ;
	\draw (0,1) node (b) [nodedot] {} ;
    \draw (0,0) node (c) [nodedot] {} ;
	\draw (0,-1) node (d) [nodedot] {} ;
	\draw (2,2) node (e) [nodedot] {} ;
	\draw (2,1) node (f) [nodedot] {} ;
    \draw (2,0) node (g) [nodedot] {} ;
	\draw (2,-1) node (h) [nodedot] {} ;
	\draw (4,2) node (i) [nodedot] {} ;
	\draw (4,1) node (j) [nodedot] {} ;
    \draw (4,0) node (k) [nodedot] {} ;
	\draw (4,-1) node (l) [nodedot] {} ;
	\draw[Medge]  (a)--(e) (b)--(f) (c)--(g) (d)--(h) 
	(e)--(i) (f)--(j) (g)--(k) (h)--(l) ;
	\draw (a) edge[thick, bend right] (c) (b) edge[thick, bend right] (d);
	\draw (i) edge[thick, bend left] (k) (j) edge[thick, bend left] (l);
	\draw (0,-2) node (A)[shadedgiantnode] {$4$};
	\draw (2,-2) node (B)[unshadedgiantnode] {$4$};
	\draw (4,-2) node (C)[shadedgiantnode] {$4$};
	\draw[Mgiantedge]  (A)--(B) (C)--(B);
    \end{tikzpicture}
}}&
\subcaptionbox{A component in $\Gamma_4$ contains $1$ split and is incident to $1$ bridge in $\Gamma_{4,2}$.
}[.9\linewidth]{
\resizebox{0.20\textwidth}{!}{%
\centering
\begin{tikzpicture}[scale=1,font=\large]
    \draw (0,1) node (a) [nodedot] {} ;
	\draw (0,0) node (b) [nodedot] {} ;
	\draw (2,2) node (c) [nodedot] {} ;
	\draw (2,1) node (d) [nodedot] {} ;
    \draw (2,0) node (e) [nodedot] {} ;
	\draw (2,-1) node (f) [nodedot] {} ;
	\draw (4,2) node (g) [nodedot] {} ;
	\draw (4,1) node (h) [nodedot] {} ;
    \draw (4,0) node (i) [nodedot] {} ;
	\draw (4,-1) node (j) [nodedot] {} ;
	\draw[Bedge]  (a)--(c) (b)--(d) (a)--(e) (b)--(f);
	\draw[Medge]  (c)--(g) (d)--(h) (e)--(i) (f)--(j);
	 \draw (g) edge[thick, bend left] (i) (h) edge[thick, bend left] (j);
	\draw (0,-2) node (A)[unshadedgiantnode] {$2$};
	\draw (2,-2) node (B)[unshadedgiantnode] {$4$};
	\draw (4,-2) node (C)[shadedgiantnode] {$4$};
	\draw[Bgiantedge]  (A)--(B);
	\draw[Mgiantedge]  (C)--(B);
    \end{tikzpicture}
}}%
\end{tabular}
\caption{Examples of backbone graphs violating  \ref{pt:forest-sb}, whose corresponding orbit graphs contain cycles.}
\label{fig:tree_forbidden_pattern}
\end{figure}
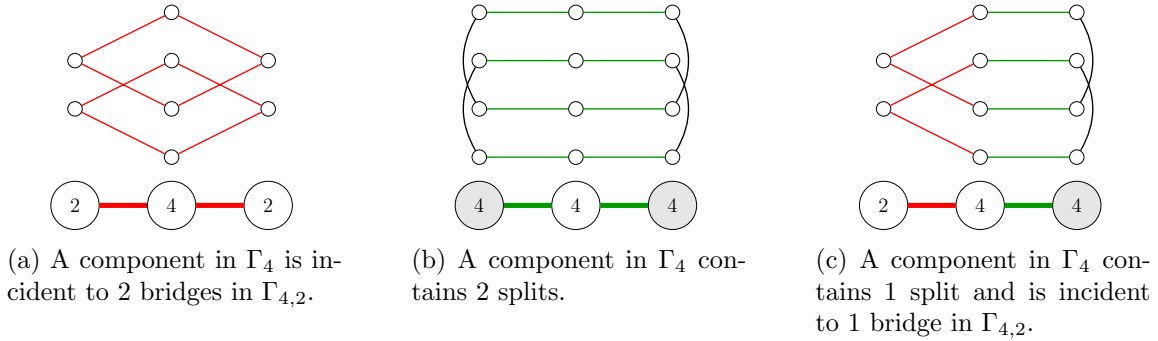

Next, we describe an algorithm for generating all possible backbone graphs $\Gamma$ that satisfy the aforementioned conditions \ref{pt:forest}--\ref{pt:forest-sb}.
Given a sequence of integers $(\mathbf{a}, \mathbf{b}, \mathbf{c})=(a_m, b_m, c_m)_{1 \le m \le k}$ with $b_m=0$ for odd $m$, we construct $\Gamma$ as follows: 
\begin{algorithm}[H]
\caption{Forest enumeration algorithm} \label{alg:f}
\begin{algorithmic}[1]
\For{each $t=1,\ldots,k$}
   \State{\bfseries Step 1: Matching stage.}
    Construct a rooted forest $\Gamma_t$  with $n_t$ giant nodes and $a_t$ giant edges; 
    Attach a label from $[t]$ to each giant edge; 
    \State{\bfseries Step 2: Splitting stage.}
		Choose $b_t$ components from $n_t-a_t$ tree components of $\Gamma_t$, and within each chosen component,
    add a split to the root;
    \State{\bfseries Step 3: Bridging stage.}
     Choose $c_t$ out of the remaining $n_t-a_t-b_t$ tree components of $\Gamma_t$, and for each chosen component,
    add a bridge connecting its root to a giant node in $\Gamma_\ell$ 
    for some $\ell<t$ that is a divisor of $t$. Attach a label from $[\ell]$ to the added bridge. 
\EndFor
\end{algorithmic}
\end{algorithm}

We claim that any orbit forest can be generated by 
\prettyref{alg:f}. To verify this claim formally, let $H$ be an orbit forest and $\Gamma$ denote the its  corresponding backbone graph in \prettyref{def:backbone}, which, for $1\le m\le k$, contains
\begin{itemize}
\item $a_m$ matchings corresponding to Type $\sfM_m$ edge orbits; 
 \item $b_m$ splits corresponding to Type $\sfS_m$ edge orbits;
 \item  $c_m$ bridges corresponding to
 Type $\sfB_{m,\ell}$ edge orbits for some $\ell<m$ that is a divisor of $m$.
 \end{itemize}
 
For each $\Gamma_m$, we arbitrarily choose the root for each plain tree component, and specify the root in each non-plain tree component  as the 
\emph{unique} giant node that either splits or is incident to a bridge in $\cup_{\ell<m} \Gamma_{m,\ell}$.
Then clearly Steps 1--3 can realize any configuration of matchings, splits and bridges in $\Gamma$, thanks to the properties \ref{pt:forest}-\ref{pt:forest-sb}.

Note that the total number of edges in the corresponding orbit forest $H$ is determined by the input parameter $(\mathbf{a}, \mathbf{b}, \mathbf{c}) $ as 
$$\sum_{m=1}^k \left[ m (a_m+c_m)  + m b_m/2 \right].$$ 
To enumerate the orbit forests, it suffices to count all possible output backbone graphs $\Gamma$ of \prettyref{alg:f} as follows. For $t=1,\ldots,k$,
  \begin{enumerate}
    \item
	It is well-known that the total number of rooted forests on $n$ vertices with $a$ edges is 
	\begin{equation}
	\binom{n-1}{a}n^{a}
	\label{eq:rooted-forest}
	\end{equation}
	(see e.g.~\cite[II.18, p.~128]{flajolet2009analytic}.)  Moreover,  each giant edge added in Step $1$ has $t$ possible labels. Therefore, the total number of  rooted  \emph{backbone} graphs $\Gamma_t$ is at most
    \begin{align}
   \binom{n_t-1}{a_t} \left( t n_t\right) ^{a_t}  \le  \binom{n_t}{a_t}\left(tn_t\right)^{a_t}.   \label{eq:at}
    \end{align}
    \item
 The total number of ways of placing $b_t$ splits is at most 
    \begin{align}
        \binom{n_t-a_t}{b_t }. \label{eq:bt}
    \end{align}
    \item
   The total number of ways of placing $c_t$ bridges is at most 
       \begin{align}
        \binom{n_t-a_t-b_t}{c_t} \left(\sum_{\ell<t}\ell n_{\ell}\right)^{c_t}. \label{eq:dt}
    \end{align}
   Note that we could further restrict the summation over $\ell$ to divisors of $t$ and get a tighter upper bound, but this is not needed for the main results.  
 
   
\end{enumerate}
Combining \prettyref{eq:at}, \prettyref{eq:bt}, and \prettyref{eq:dt}, we get that the total number of output backbone graphs $\Gamma$ with 
input parameter $(\mathbf{a}, \mathbf{b}, \mathbf{c}) $ is at most
\begin{align}
    \prod_{1\le t\le k} \indc{ b_t=0 \text{ for odd } t }
    \binom{n_t}{a_t,\, b_t, \, c_t } \left(tn_t\right)^{a_t}  \left(\sum_{\ell<t}\ell n_{\ell}\right)^{c_t} \label{eq:gforest}.
\end{align}
Then the desired \prettyref{eq:jk-forest} readily follows from
\begin{align*}
\sum_{H \in \calF_k } s^{2 e(H)}  
\le & \sum_{\mathbf{a}, \mathbf{b}, \mathbf{c} } \;  \prod_{1\le t\le k} 
\indc{ b_t=0 \text{ for odd } t }
  \binom{n_t}{a_t,\, b_t, \, c_t }  \left(t n_t\right)^{a_t}  \left(\sum_{\ell<t }\ell n_{\ell}\right)^{c_t}  s^{ 2t a_t + t b_t  + 2t c_t } \\
\le &\prod_{1\le t\le k} \left( 1+  s^t \indc{t:\mathrm{even}} + s^{2t} \sum_{\ell \le t} \ell n_\ell \right)^{n_t}.
\end{align*}

\subsection{Proof of \prettyref{thm:jk}: Generating function of orbit pseudoforests}\label{sec:enumeration_pseudoforst}

Fix  $\sigma =\pi^{-1}\circ\tpi$ and recall $n_m$ denotes the number of $m$-node orbits in $\sigma$.
In this section we bound the generating function \prettyref{eq:gf-pseudoforest} 
of orbit pseudoforests $H \in \calH_k$ and prove \prettyref{thm:jk}. Recall that each orbit graph $H$ can be equivalently represented as a backbone graph $\Gamma$ as in \prettyref{def:backbone}.
In addition, we need the following vocabularies:
For $1\le m\le k$ and each $u\in S_m$, let $C(u)$ denote the connected component in $\Gamma_m$ containing $u$.

Similar to the reasoning in \prettyref{sec:enumeration_forst}, when $H$ is a pseudoforest, its backbone graph $\Gamma$ must satisfy the following properties:
\begin{enumerate}[label=(P\arabic*)]
    \item \label{pt:pseudo} For each $1 \le m \le k$, $\Gamma_m$ is a pseudoforest (with self-loops and parallel edges counted as cycles);
    \item  \label{pt:divisor} For each $1 \le \ell < m \le k$, $\Gamma_{m, \ell}$ is empty unless $\ell$ is a divisor of $m$.
    \item \label{pt:unicyclic} Each unicyclic component of  $\Gamma_m$ is plain.
    \item \label{pt:sss} A tree component in $\Gamma_m$ contains at most two splits.
    \item \label{pt:bb} 
    Let $(u,v) \in \Gamma_{m,\ell_1}$ and $(u',v')\in \Gamma_{m,\ell_2}$ be two bridges with $\ell_1,\ell_2<m$, such that $u$ and $u'$ belong to the same tree component in $\Gamma_m$. 
    Then $m$ must be even and $\ell_1=\ell_2=m/2$. 
    \item \label{pt:sb} 
    Let $(u,v) \in \Gamma_{m,\ell}$ be a bridge with $\ell<m$ such that $u$ belongs to a tree component that contains a split in $\Gamma_m$. Then $m$ must be even and $\ell=m/2$. Furthermore, $v$ must belong to a plain tree component in $\Gamma_{m/2}$. 
    \item \label{pt:hybrid} For each $(u,v)$ and $(u',v')$ that satisfy either \ref{pt:bb} or \ref{pt:sb} where $v \neq v'$,
    the ending points $v$ and $v'$ must belong to distinct plain tree components in $\Gamma_{m/2}$.
\end{enumerate}
Otherwise, $H$ contains a component with at least two cycles, violating the pseudoforest constraint.
See \prettyref{fig:forbidden_bb} - \prettyref{fig:forbidden_hybrid} for illustrations of forbidden patterns that violate  \ref{pt:sss} - \ref{pt:hybrid}.
   \begin{figure}[ht]
    \centering
    \begin{tabular}{C{.4\textwidth}C{.4\textwidth}}
    \subcaptionbox{A component in $\Gamma_4$ contains $3$ splits.
    }[1\linewidth]{
    \resizebox{0.27\textwidth}{!}{%
    \centering
    \begin{tikzpicture}[scale=1,font=\large]
    	\draw (0,2) node (a) [nodedot] {} ;
    	\draw (0,1) node (b) [nodedot] {} ;
        \draw (0,0) node (c) [nodedot] {} ;
    	\draw (0,-1) node (d) [nodedot] {} ;
    	\draw (2,2) node (e) [nodedot] {} ;
    	\draw (2,1) node (f) [nodedot] {} ;
        \draw (2,0) node (g) [nodedot] {} ;
    	\draw (2,-1) node (h) [nodedot] {} ;
    	\draw (4,2) node (i) [nodedot] {} ;
    	\draw (4,1) node (j) [nodedot] {} ;
        \draw (4,0) node (k) [nodedot] {} ;
    	\draw (4,-1) node (l) [nodedot] {} ;
    	\draw[Medge]  (a)--(e) (b)--(f) (c)--(g) (d)--(h) 
    	(e)--(i) (f)--(j) (g)--(k) (h)--(l) ;
    	\draw (a) edge[thick, bend right] (c) (b) edge[thick, bend right] (d);
    	\draw (e) edge[thick, bend right] (g) (f) edge[thick, bend right] (h);
    	\draw (i) edge[thick, bend left] (k) (j) edge[thick, bend left] (l);
    	\draw (0,-2) node (A)[shadedgiantnode] {$4$};
    	\draw (2,-2) node (B)[shadedgiantnode] {$4$};
    	\draw (4,-2) node (C)[shadedgiantnode] {$4$};
    	\draw[Mgiantedge]  (A)--(B) (C)--(B);
    		\draw(0,-2.7) node (c) {$ $};
	    \draw(2,-2.7) node (c) {$ $};
	    \draw(4,-2.7) node (c) {$ $};
	    \draw(6,-2.7) node (c) {$ $};
        \end{tikzpicture}
    }}&
    \subcaptionbox{A component in $\Gamma_4$ is incident to a bridge $(u,v)\in \Gamma_{4,1}$ and a bridge $(u',v')\in \Gamma_{4,2}$.
    }%
    [1\linewidth]{
    \resizebox{0.27\textwidth}{!}{%
    \centering
    \begin{tikzpicture}[scale=1,font=\large]
        \draw (0,0.5) node (a) [nodedot] {} ;
    	\draw (2,2) node (b) [nodedot] {} ;
    	\draw (2,1) node (c) [nodedot] {} ;
        \draw (2,0) node (d) [nodedot] {} ;
    	\draw (2,-1) node (e) [nodedot] {} ;
    	\draw (4,2) node (f) [nodedot] {} ;
    	\draw (4,1) node (g) [nodedot] {} ;
        \draw (4,0) node (h) [nodedot] {} ;
    	\draw (4,-1) node (i) [nodedot] {} ;
    	\draw (6,1) node (j) [nodedot] {} ;
    	\draw (6,0) node (k) [nodedot] {} ;
    	\draw[Bedge]  (a)--(b) (a)--(c) (a)--(d) (a)--(e) 
    	(j)--(f) (k)--(g) (j)--(h) (k)--(i) ;
    	\draw[Medge] (b)--(f) (c)--(g) (d)--(h) (e)--(i) ;
    	\draw (0,-2) node (A)[unshadedgiantnode] {$1$};
    	\draw (2,-2) node (B)[unshadedgiantnode] {$4$};
    	\draw (4,-2) node (C)[unshadedgiantnode] {$4$};
    	\draw (6,-2) node (D)[unshadedgiantnode] {$2$};
    	\draw[Bgiantedge]  (A)--(B) (C)--(D);
    	\draw[Mgiantedge] (B)--(C);
    	\draw(0,-2.7) node (c) {$v$};
	    \draw(2,-2.7) node (c) {$u$};
	    \draw(4,-2.7) node (c) {$u'$};
	    \draw(6,-2.7) node (c) {$v'$};
        \end{tikzpicture}
    }}
    \end{tabular}
    \caption{Examples of backbone graphs violating  \ref{pt:sss} and \ref{pt:bb}, shown in  (a) and (b), respectively, and the corresponding orbit graphs.
    \label{fig:forbidden_bb}
    }
    \end{figure}
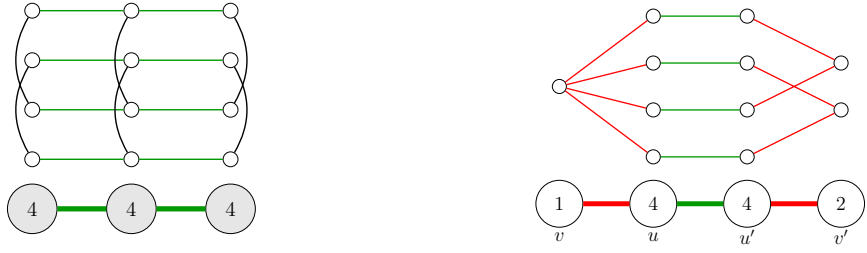
    \begin{figure}[ht]
        \centering
        \begin{tabular}{C{.3\textwidth}C{.3\textwidth}C{.3\textwidth}}
        \subcaptionbox{A component in $\Gamma_4$ contains $1$ splits and is incident to $1$ bridge $(u,v)\in \Gamma_{4,1}$. 
				}[.9\linewidth]{
        \resizebox{0.27\textwidth}{!}{%
        \centering
        \begin{tikzpicture}[scale=1,font=\large]
        	\draw (0,2) node (a) [nodedot] {} ;
        	\draw (2,2) node (b) [nodedot] {} ;
        	\draw (2,1) node (c) [nodedot] {} ;
            \draw (2,0) node (d) [nodedot] {} ;
        	\draw (2,-1) node (e) [nodedot] {} ;
        	\draw (4,2) node (f) [nodedot] {} ;
        	\draw (4,1) node (g) [nodedot] {} ;
            \draw (4,0) node (h) [nodedot] {} ;
        	\draw (4,-1) node (i) [nodedot] {} ;
        	\draw (6,2) node (j) [nodedot] {} ;
        	\draw (6,1) node (k) [nodedot] {} ;
            \draw (6,0) node (l) [nodedot] {} ;
        	\draw (6,-1) node (m) [nodedot] {} ;
        	\draw[Medge] (b)--(f) (c)--(g) (d)--(h) (e)--(i) (j)--(f) (k)--(g) (l)--(h) (m)--(i) ;
        	\draw[Bedge] (a)--(b) (a)--(c) (a)--(d) (a)--(e) ;
        	\draw (j) edge[thick, bend left] (l) (k) edge[thick, bend left] (m);
        	\draw (0,-2) node (A)[unshadedgiantnode] {$1$};
        	\draw (2,-2) node (B)[unshadedgiantnode] {$4$};
        	\draw (4,-2) node (C)[unshadedgiantnode] {$4$};
        	\draw (6,-2) node (D)[shadedgiantnode] {$4$};
        	\draw[Bgiantedge]  (A)--(B);
        	\draw[Mgiantedge]  (C)--(B) (C)--(D);
        	\draw(0,-2.7) node (c) {$v$};
	        \draw(2,-2.7) node (c) {$u$};
            \end{tikzpicture}
        }}&
        \subcaptionbox{A component in $\Gamma_4$ contains $1$ split and is incident to $1$ bridge $(u,v) \in \Gamma_{4,2}$ where $v$ is 
        in a non-plain component in $\Gamma_2$ that is incident to $1$ bridge in $\Gamma_{2,1}$.
        }[.9\linewidth]{
        \resizebox{0.27\textwidth}{!}{%
        \centering
        \begin{tikzpicture}[scale=1,font=\large]
        	\draw (-2,0.5) node (a) [nodedot] {} ;
            \draw (0,1) node (c) [nodedot] {} ;
        	\draw (0,0) node (d) [nodedot] {} ;
        	\draw (2,2) node (e) [nodedot] {} ;
        	\draw (2,1) node (f) [nodedot] {} ;
            \draw (2,0) node (g) [nodedot] {} ;
        	\draw (2,-1) node (h) [nodedot] {} ;
        	\draw (4,2) node (j) [nodedot] {} ;
            \draw (4,1) node (k) [nodedot] {} ;
            \draw (4,0) node (l) [nodedot] {} ;
            \draw (4,-1) node (m) [nodedot] {} ;
        	\draw[Bedge] (a)--(c) (a)--(d) (c)--(e) (d)--(f) (c)--(g) (d)--(h) ;	
        	\draw[Medge] (e)--(j) (f)--(k) (g)--(l) (h)--(m);
        	\draw (j) edge[thick, bend left] (l) (k) edge[thick, bend left] (m);
        	\draw (-2,-2) node (A)[unshadedgiantnode] {$1$};
        	\draw (0,-2) node (B)[unshadedgiantnode] {$2$};
        	\draw (2,-2) node (C)[unshadedgiantnode] {$4$};
        	\draw (4,-2) node (D)[shadedgiantnode] {$4$};
        	\draw[Bgiantedge] (A)--(B) (B)--(C) ;
        	\draw[Mgiantedge]  (C)--(D);
        	\draw(0,-2.7) node (c) {$v$};
	        \draw(2,-2.7) node (c) {$u$};
            \end{tikzpicture}
        }}&
        \subcaptionbox{A component in $\Gamma_4$ contains $1$ split and is incident to $1$ bridge $(u,v)\in\Gamma_{4,2}$ where $v$ is in a non-plain component in $\Gamma_2$ that contains a split.
         }[.9\linewidth]{
        \resizebox{0.27\textwidth}{!}{%
        \centering
        \begin{tikzpicture}[scale=1,font=\large]
        	\draw (-2,1) node (a) [nodedot] {} ;
        	\draw (-2,0) node (b) [nodedot] {} ;
            \draw (0,1) node (c) [nodedot] {} ;
        	\draw (0,0) node (d) [nodedot] {} ;
        	\draw (2,2) node (e) [nodedot] {} ;
        	\draw (2,1) node (f) [nodedot] {} ;
            \draw (2,0) node (g) [nodedot] {} ;
        	\draw (2,-1) node (h) [nodedot] {} ;
        	\draw (4,2) node (j) [nodedot] {} ;
            \draw (4,1) node (k) [nodedot] {} ;
            \draw (4,0) node (l) [nodedot] {} ;
            \draw (4,-1) node (m) [nodedot] {} ;
        	\draw[Bedge]  (c)--(e) (d)--(f) (c)--(g) (d)--(h) ;	
        	\draw[Medge] (a)--(c) (b)--(d) (e)--(j) (f)--(k) (g)--(l) (h)--(m);
        	\draw[Sedge] (a)--(b);
        	\draw (j) edge[thick, bend left] (l) (k) edge[thick, bend left] (m);
        	\draw (-2,-2) node (A)[shadedgiantnode] {$2$};
        	\draw (0,-2) node (B)[unshadedgiantnode] {$2$};
        	\draw (2,-2) node (C)[unshadedgiantnode] {$4$};
        	\draw (4,-2) node (D)[shadedgiantnode] {$4$};
        	\draw[Bgiantedge] (B)--(C) ;
        	\draw[Mgiantedge] (A)--(B) (C)--(D);
            \draw(0,-2.7) node (c) {$v$};
	        \draw(2,-2.7) node (c) {$u$};
            \end{tikzpicture}
        }}
        \end{tabular}
        \caption{Examples of backbone graphs  violating  \ref{pt:sb} and the corresponding  orbit graphs.}
        \label{fig:forbidden_sb}
        \end{figure}
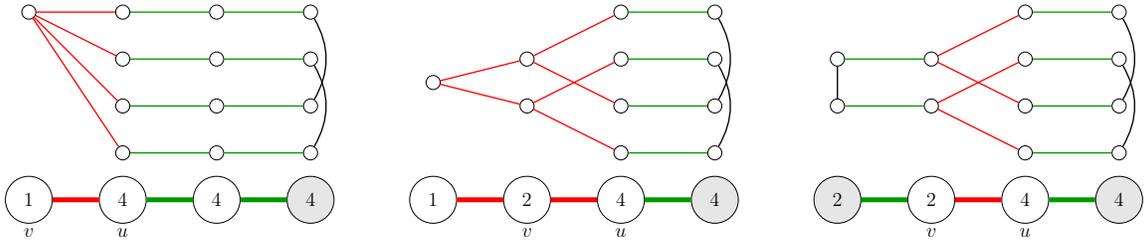 
        
    \begin{figure}[ht]
        \centering
        \begin{tabular}{C{.45\textwidth}C{.5\textwidth}}
        \subcaptionbox{$(u,v)$ and $(u,v')$ satisfy \ref{pt:bb}, while $v$ is  in a non-plain component in $\Gamma_2$ that contains a split.
        \\}[.9\linewidth]{
        \resizebox{0.27\textwidth}{!}{%
        \centering
        \begin{tikzpicture}[scale=1,font=\large]
        	\draw (-2,1) node (a) [nodedot] {} ;
        	\draw (-2,0) node (b) [nodedot] {} ;
            \draw (0,1) node (c) [nodedot] {} ;
        	\draw (0,0) node (d) [nodedot] {} ;
        	\draw (2,2) node (e) [nodedot] {} ;
        	\draw (2,1) node (f) [nodedot] {} ;
            \draw (2,0) node (g) [nodedot] {} ;
        	\draw (2,-1) node (h) [nodedot] {} ;
        	\draw (4,1) node (j) [nodedot] {} ;
            \draw (4,0) node (k) [nodedot] {} ;
        	\draw[Medge]  (a)--(c) (b)--(d);
        	\draw[Bedge]  (c)--(e) (d)--(f) (c)--(g) (d)--(h)  (j)--(e) (k)--(f) (j)--(g) (k)--(h);
        	\draw[Sedge](a)--(b);
        	\draw (-2,-2) node (A)[shadedgiantnode] {$2$};
        	\draw (0,-2) node (B)[unshadedgiantnode] {$2$};
        	\draw (2,-2) node (C)[unshadedgiantnode] {$4$};
        	\draw (4,-2) node (D)[unshadedgiantnode] {$2$};
        	\draw[Bgiantedge]  (B)--(C) (C)--(D);
        	\draw[Mgiantedge]  (A)--(B);
        	\draw(0,-2.7) node (c) {$v$};
	        \draw(2,-2.7) node (c) {$u$};
	        \draw(4,-2.7) node (c) {$v'$};
            \end{tikzpicture}
        }}&
        \subcaptionbox{$(u,v)$ and $(u,v')$ satisfy \ref{pt:bb}, while $v$ is in a non-plain component that is incident to a bridge in $\Gamma_{2,1}$
        \\}[.9\linewidth]{
        \resizebox{0.27\textwidth}{!}{%
        \centering
        \begin{tikzpicture}[scale=1,font=\large]
        	\draw (-2,0.5) node (a) [nodedot] {} ;
            \draw (0,1) node (c) [nodedot] {} ;
        	\draw (0,0) node (d) [nodedot] {} ;
        	\draw (2,2) node (e) [nodedot] {} ;
        	\draw (2,1) node (f) [nodedot] {} ;
            \draw (2,0) node (g) [nodedot] {} ;
        	\draw (2,-1) node (h) [nodedot] {} ;
        	\draw (4,1) node (j) [nodedot] {} ;
            \draw (4,0) node (k) [nodedot] {} ;
        	\draw[Bedge] (a)--(c) (a)--(d) (c)--(e) (d)--(f) (c)--(g) (d)--(h)  (j)--(e) (k)--(f) (j)--(g) (k)--(h);	\draw (-2,-2) node (A)[unshadedgiantnode] {$1$};
        	\draw (0,-2) node (B)[unshadedgiantnode] {$2$};
        	\draw (2,-2) node (C)[unshadedgiantnode] {$4$};
        	\draw (4,-2) node (D)[unshadedgiantnode] {$2$};
        	\draw[Bgiantedge]  (A)--(B) (B)--(C) (C)--(D);
        	\draw(0,-2.7) node (c) {$v$};
	        \draw(2,-2.7) node (c) {$u$};
	        \draw(4,-2.7) node (c) {$v'$};
            \end{tikzpicture}
        }}\\
        \subcaptionbox{$(u,v)$ and $(u,v')$ satisfy \ref{pt:bb}, while $v$ and $v'$ are in the same component in $\Gamma_2$.
         }[.9\linewidth]{
        \resizebox{0.35\textwidth}{!}{%
        \centering
        \begin{tikzpicture}[scale=1,font=\large]
        	\draw (-2,1) node (a) [nodedot] {} ;
        	\draw (-2,0) node (b) [nodedot] {} ;
            \draw (0,1) node (c) [nodedot] {} ;
        	\draw (0,0) node (d) [nodedot] {} ;
        	\draw (2,2) node (e) [nodedot] {} ;
        	\draw (2,1) node (f) [nodedot] {} ;
            \draw (2,0) node (g) [nodedot] {} ;
        	\draw (2,-1) node (h) [nodedot] {} ;
        	\draw (4,1) node (j) [nodedot] {} ;
            \draw (4,0) node (k) [nodedot] {} ;
            \draw (6,1) node (l) [nodedot] {} ;
            \draw (6,0) node (m) [nodedot] {} ;
        	\draw[Medge]  (a)--(c) (b)--(d) (j)--(l) (k)--(m);
        	\draw[Bedge]  (c)--(e) (d)--(f) (c)--(g) (d)--(h)  (j)--(e) (k)--(f) (j)--(g) (k)--(h);
        	\draw (a) edge[green!60!black, thick, bend left=15] (l) (b) edge[green!60!black, thick, bend right=15] (m);
        	\draw (-2,-2.5) node (A)[unshadedgiantnode] {$2$};
        	\draw (0,-2.5) node (B)[unshadedgiantnode] {$2$};
        	\draw (2,-2.5) node (C)[unshadedgiantnode] {$4$};
        	\draw (4,-2.5) node (D)[unshadedgiantnode] {$2$};
        	\draw (6,-2.5) node (E)[unshadedgiantnode] {$2$};
        	\draw (A) edge[green!60!black, line width=3.0pt, bend left=20] (E);
        	\draw[Bgiantedge]  (B)--(C) (C)--(D);
        	\draw[Mgiantedge]  (A)--(B) (D)--(E);
        	\draw(0,-3.2) node (c) {$v$};
	        \draw(2,-3.2) node (c) {$u$};
	        \draw(4,-3.2) node (c) {$v'$};
            \end{tikzpicture}
        }}& 
        \subcaptionbox{$(u,v)$ satisfies \ref{pt:bb}, $(u',v')$ satisfies \ref{pt:sb},  while $v$ and $v'$ are in the same component in $\Gamma_2$.
        }[0.9\linewidth]{
        \resizebox{0.42\textwidth}{!}{%
        \centering
        \begin{tikzpicture}[scale=1,font=\large]
            \draw (-2,1) node (a) [nodedot] {} ;
        	\draw (-2,0) node (b) [nodedot] {} ;
        	\draw (0,2) node (c) [nodedot] {} ;
        	\draw (0,1) node (d) [nodedot] {} ;
            \draw (0,0) node (e) [nodedot] {} ;
        	\draw (0,-1) node (f) [nodedot] {} ;
        	\draw (2,1) node (g) [nodedot] {} ;
            \draw (2,0) node (h) [nodedot] {} ;
            \draw (4,1) node (i) [nodedot] {} ;
            \draw (4,0) node (j) [nodedot] {} ;
            \draw (6,2) node (k) [nodedot] {} ;
            \draw (6,1) node (l) [nodedot] {} ;
            \draw (6,0) node (m) [nodedot] {} ;
            \draw (6,-1) node (n) [nodedot] {} ;
            \draw (8,2) node (o) [nodedot] {} ;
            \draw (8,1) node (p) [nodedot] {} ;
            \draw (8,0) node (q) [nodedot] {} ;
            \draw (8,-1) node (r) [nodedot] {} ;
        	\draw[Medge]  (g)--(i) (h)--(j)  (k)--(o) (l)--(p) (m)--(q) (n)--(r);
        	\draw[Bedge]  (a)--(c) (b)--(d) (a)--(e) (b)--(f)  (g)--(c) (h)--(d) (g)--(e) (h)--(f) 
        	(i)--(k) (j)--(l) (i)--(m) (j)--(n) ;
        	\draw (o) edge[thick, bend left] (q) (p) edge[thick, bend left] (r);
        	\draw (8,-2) node (A)[shadedgiantnode] {$4$};
        	\draw (6,-2) node (B)[unshadedgiantnode] {$4$};
        	\draw (4,-2) node (C)[unshadedgiantnode] {$2$};
        	\draw (2,-2) node (D)[unshadedgiantnode] {$2$};
        	\draw (0,-2) node (E)[unshadedgiantnode] {$4$};
        	\draw (-2,-2) node (F)[unshadedgiantnode] {$2$};
        	\draw[Bgiantedge]  (B)--(C)  (D)--(E) (E)--(F);
        	\draw[Mgiantedge]  (A)--(B) (C)--(D)   ;
        	\draw(0,-2.7) node (c) {$u$};
	        \draw(2,-2.7) node (c) {$v$};
	        \draw(4,-2.7) node (c) {$v'$};
	        \draw(6,-2.7) node (c) {$u'$};
            \end{tikzpicture}
        }}
        \end{tabular}
        \caption{Examples of backbone graphs violating  \ref{pt:hybrid} and the corresponding  orbit graphs.}
        \label{fig:forbidden_hybrid}
        \end{figure}
Properties \ref{pt:pseudo}--\ref{pt:hybrid} are justified by the following arguments:
\begin{itemize}
    \item Paralleling  conditions \ref{pt:forest} and \ref{pt:forest-divisor} for the forest constraint, \ref{pt:pseudo} and \ref{pt:divisor} follows from the classification of edge orbits and orbit graphs in \prettyref{sec:classification_edge_orbits}; 
    \item Suppose \ref{pt:unicyclic} does not hold. Since the excess of the corresponding orbit graph of a plain unicyclic component in $\Gamma_m$ is $0$, by \ref{O:split} and \ref{O:1bridge}, there exists a unicyclic component component $C$ in $\Gamma_m$ such that $\ex(H_C) > 0$, contradicting $H$ being a pseudo-forest;
    \item Suppose \ref{pt:sss} does not hold. Then by \ref{O:split} and \prettyref{lmm:plain}, there exists a tree component component $C$ in $\Gamma_m$ such that $\ex(H_C) > 0$, contradicting $H$ being a pseudo-forest;
    \item Suppose \ref{pt:bb} does not hold. Then by \ref{O:1bridge} and \prettyref{lmm:plain}, there exists a tree component component $C$ in $\Gamma_m$ such that $\ex(H_C) > 0$, contradicting $H$ being a pseudo-forest;
    \item  To prove \ref{pt:sb}, 
        let $G_1$ denote the orbit graph of $C(u)$ consisting of edge orbits (including splits, matchings, and cycles). Let $G_2$ denote the orbit graph of $C(v)$ consisting of edge orbits (including splits, matchings, and cycles) in $C(v)$, as well as bridges in $\cup_{\ell<m/2}\Gamma_{m/2,\ell}$ that are incident to $C(v)$. 
        Let $G_{\mathrm{new}}$ denote the edge-disjoint union of $G_1$, $G_2$, and the edge orbit corresponding to the bridge $(u,v)$. 
        Since $C(u)$ contains a split, by \ref{O:split} and \prettyref{lmm:plain}, $ \ex(G_1)\ge -m/2$. 
        Then we have 
         $$\ex(G_{\mathrm{new}}) =\ex(G_1)+\ex(G_2)+m  \ge \ex(G_2) +m/2 \ge  0,$$
        where the last inequality is met with equality if and only if $C(v)$ is a plain tree component in $\Gamma_{m/2}$ by \prettyref{lmm:plain}. Hence,  \ref{pt:sb} follows. 
    \item To prove \ref{pt:hybrid}, 
        let $C=C(u) \cup C(u')$. 
        Let $G_1$ denote the orbit graph of $C$ consisting of edge orbits (including splits, matchings, and cycles), as well as bridges in $\cup_{\ell<m} \Gamma_{m, \ell}$ except for $(u,v)$ and $(u',v')$ that are incident to $C$. 
        If $C(u) = C(u')$, then $\ex(G_1) \ge -m$ by  \prettyref{lmm:plain}.  If $C(u) \neq C(u')$, then $G_1$ is an edge-disjoint union of 
        $H_1$ and $H_1'$, where $H_1$ (resp.\ $H_1'$) is the 
        orbit graph of $C(u) $ (resp.\ $C(u')$) consisting of edge orbits (including splits, matchings, and cycles) in $C(u)$ (resp.\ $C(u')$),
        as well as bridges in $\cup_{\ell<m} \Gamma_{m, \ell}$ except for $(u,v)$ (resp.\ $(u',v')$) that are incident to $C(u)$ (resp.\ $C(u')$). 
        By assumption, together with \ref{O:split}, \ref{O:1bridge} and \prettyref{lmm:plain}, $\ex(H_1)\ge -m/2$ and $\ex(H_1') \ge -m/2$ and thus $\ex(G_1) \ge \ex(H_1)+\ex(H_2 ) \ge -m$.

        Let $C' = C(v) \cup C(v')$.
        Note that by \ref{pt:bb} and \ref{pt:sb}, both $C(v)$ and $C(v')$ are components in $\Gamma_{m/2}$.
        Let $G_2$ denote the orbit graph of $C'$ consisting of edge orbits (including splits, matchings, and cycles) in $C'$, as well as bridges in $\cup_{\ell<m/2} \Gamma_{m/2, \ell}$ that are incident to $C'$. 
        Let $G_{\mathrm{new}}$ denote the edge-disjoint union of $G_1$, $G_2$, and  the edge obits corresponding to the two bridges $(u,v)$ and $(u',v')$.
        Then,
        $$
        \ex(G_\mathrm{new}) \ge \ex(G_1) + \ex(G_2) + 2m  \ge -m + \ex(G_2) + 2 m = \ex(G_2) +m .
        $$ 
        By assumption, $G_{\mathrm{new}}$ is a pseudo-forest and thus $\ex(G_\mathrm{new}) \le 0$. 
        It follows that $\ex(G_2) \le -m$ and hence $v$ and $v'$ must be in distinct plain tree components in $\Gamma_{m/2}$ by \prettyref{lmm:plain}. 

\end{itemize}         

The implication of \ref{pt:sss}-\ref{pt:hybrid}
is the following. 
For each $m \in [k]$, define
\[
\calE(m) \triangleq \cup_{\ell < m} E(\Gamma_{m,\ell})
\]
consisting of all bridges between $m$-node orbits and shorter orbits. 
Then $\calE(m)$ can be divided into two sets of bridges as follows.
For each $u\in S_m$, recall that $C(u)$ denotes the connected component in $\Gamma_m$ containing $u$. A bridge is denoted by a giant edge $(u,v) \in \Gamma_{m,\ell}$ with $\ell<m$, where $u\in S_m$ in the longer orbit is called the \emph{starting point} and $v \in S_\ell$ in the shorter orbit is called the \emph{ending point}.
Define 
\begin{align}
\Esingle(m) \triangleq & ~ \{(u,v)\in E(\Gamma): u \in S_m, v \in \cup_{\ell < m} S_\ell, \label{eq:Esingle}\\
& ~~~\text{$C(u)$ contains no split and is not incident to any bridge in $\cup_{\ell<m}\Gamma_{m,\ell}$ other than $(u,v)$} \}	\nonumber \\
\Edouble(m) \triangleq & ~ \{(u,v)\in E(\Gamma): u \in S_m, v \in  S_{m/2}, \label{eq:Edouble}\\
& ~~~\text{$C(u)$ contains a split or is incident to some bridge in $\cup_{\ell<m}\Gamma_{m,\ell}$ other than $(u,v)$} \}	\nonumber 
\end{align}
By Properties \ref{pt:sb}--\ref{pt:hybrid}
, we have $\calE(m)=\Esingle(m) \cup \Edouble(m)$. Moreover,
\begin{itemize}
		\item For each $(u,v),(u',v') \in \Esingle(m)$, the starting points $u$ and $u'$ belong to separate tree components in $\Gamma_m$, \ie, $C(u)$ and $C(u')$ are distinct tree components in $\Gamma_m$. 
		Furthermore, $C(u)$ (resp.\ $C(u')$) contains no split and is not incident to any bridge other than $(u,v)$ (resp.\ $(u',v')$). (This is just repeating definition). 
	\item For each $(u,v),(u',v') \in \Edouble(m)$, the ending points $v$ and $v'$ belong to separate plain tree components in $\Gamma_{m/2}$, by \ref{pt:hybrid}. 
\end{itemize}
The above observation suggests that to specify the bridges in $\Esingle(m)$, one can use the \emph{forward construction} by first choosing their starting points from separate components of $\Gamma_m$ then choosing the ending points from shorter orbits $\cup_{\ell<m} S_\ell$ in an unconstrained way; to specify the bridges in $\Edouble(m)$, one can use the \emph{backward construction} by first choosing their ending points from separate components of $\Gamma_{m/2}$ then choosing the starting points from $S_m$ in an unconstrained way.
This separate account of bridges is useful in the enumeration scheme which we describe next.

Next, we describe an algorithm for generating all possible backbone graphs $\Gamma$ that satisfy the properties \ref{pt:pseudo}--\ref{pt:hybrid}.
Given a sequence of integers $(\mathbf{a}, \mathbf{b}, \mathbf{c}, \mathbf{d} ) = (a_t, b_t, c_t,d_t)_{1 \le t \le k}$ with $b_t=0$ for odd $t$ and $d_t=0$ if $2t>k$, 
we construct $\Gamma$ as follows: 
\begin{algorithm}[H]
\caption{Pseudoforest enumeration algorithm} \label{alg:pf}
\begin{algorithmic}[1]
\For{each $t=1,\ldots,k$}
   \State{\bfseries Step 1: Matching stage.}
    Construct a rooted pseudoforest $\Gamma_t$  with $n_t$ giant nodes and $a_t$ giant edges (allowing self-loops and multiple edges).     
    Attach a label from $[t]$ to each giant edge and a label from $\left[\lfloor\frac{t-1}{2}\rfloor \right]$ to each self-loop; 
    \State{\bfseries Step 2: Splitting stage.}
		Choose $b_t$ components from $n_t-a_t$ tree components of $\Gamma_t$, and within each chosen component, choose a node:
    if the node chosen is the same as the root, 
    add a split to the root; otherwise, add two splits, one at the root and the other one at the chosen node;
    \State{\bfseries Step 3: Forward bridging stage.}
     Choose $c_t$ out of the remaining $n_t-a_t-b_t$ tree components of $\Gamma_t$, and for each chosen component,
    add a bridge connecting its root to a giant node in $\Gamma_\ell$ 
    for some $\ell<t$ that is a divisor of $t$. Attach a label from $[\ell]$ to the added bridge. 
    \State{\bfseries Step 4: Backward bridging stage.}
     Choose $d_t$ from the remaining $n_t-a_t-b_t-c_t$ tree components of of $\Gamma_t$. 
     For each chosen component, add a bridge by connecting its root to a giant node in $\Gamma_{2t}$. 
     Attach a label from $[t]$ to the added bridge. 
\EndFor
\end{algorithmic}
\end{algorithm}
We note that an output graph of \prettyref{alg:pf} is not necessarily a pseudoforest; nevertheless any orbit pseudoforest can be generated by \prettyref{alg:pf}, which is what we need for upper bounding the generating function of orbit pseudoforests, $\sum_{H\in\calH_k}s^{2e(H)}$.
To verify this claim formally, let $H$ be an orbit pseudoforest and let $\Gamma$ denote its backbone graph as in \prettyref{def:backbone}, where for $1\le m\le k$ there are:
\begin{itemize}
\item $a_m$ giant edges (including self-loops) corresponding to either Type $\sfM_m$ or $\sfC_m$ edge orbits; 
 \item 
 $b_m$ components that contain splits
 corresponding to Type $\sfS_m$ edge orbits; 
 \item  $c_m$ giant edges corresponding to
 Type $\sfB_{m,\ell}$ edge orbits for some $\ell<m$ that is a divisor of $m$;
 \item  $d_m$ giant edges corresponding to Type $\sfB_{m, 2m}$ edge orbits.
 \end{itemize}
For each $\Gamma_m$ where $1\le m\le k$, we arbitrarily choose the root for each plain tree component, and specify the root in each non-plain tree component as the giant node that either splits or is incident to a bridge in $\Esingle(m) \cup \Edouble(2m)$ (when there are two giant nodes that split in a tree component, we choose any one of them as the root; otherwise, the choice of the root is unique). 
Then it is clear that Steps 1 and 2 can realize any configuration of splits and matchings in $\Gamma$, thanks to Properties \ref{pt:pseudo}--\ref{pt:sss}. 
Finally, note that 
bridges in $\Esingle(m)$ are added by Step 3 (forward bridging) at iteration $t=m$ with $c_m=|\Esingle(m)|$, and
bridges in $\Edouble(m)$ are added by Step 4 (backward bridging) at iteration $t=m/2$ with $d_{m/2}=|\Edouble(m)|$.

Next we bound the generating function $\sum_{H\in\calH_k}s^{2e(H)}$ from above.
We first state an auxiliary lemma, which extends the well-known formula \prettyref{eq:rooted-forest} for  enumerating rooted forests.
     \begin{lemma}\label{lmm:counting_pseudoforests}
      The number of rooted pseudoforests on $n$ nodes with $a$ edges (allowing self-loops and multiple edges)  is at most
\begin{equation}
         \binom{n}{a }\left(2n \right)^{a}. \label{eq:rooted-pseudoforest}
\end{equation}
\end{lemma}
\begin{proof}
    To see this, let $m$ denote the number of cycles (including self-loops and parallel edges). 
    Then the number of connected components is $n - a + m$. 
    To enumerate all such rooted pseudoforests, we first enumerate all rooted forests on $n$ vertices with $a-m$ edges,
    then choose $m$ roots out of $n-a+m$ roots, and finally add one edge to each chosen root to form a cycle within its corresponding component. 
  Each added edge can either be a self-loop at the root or connect the root to some other node, so there are at most $n$ different choices of the added edge. 
  Therefore,  the total number of rooted pseudoforests on $n$ vertices with $a$ edges is at most 
     \begin{align*}
         \sum_{m=0}^{a}\binom{n}{a-m}n^{a-m}\binom{n-a+m}{m}n^m 
         & = n^{a} \underbrace{\sum_{m=0}^{a}\binom{n}{a-m }\binom{n-a+m}{m}}_{2^{a} \binom{n}{a}} = \binom{n}{a}(2n)^{a}.
     \end{align*}
\end{proof}

Now, we can enumerate all possible output backbone graphs $\Gamma$ of \prettyref{alg:pf} as follows.
For $t=1,\ldots,k$,
\begin{enumerate}
   \item 

   Note that $\Gamma_t$ constructed in Step 1 is a  rooted pseudoforest with $n_t$ giant nodes and $a_t$ giant edges. 
   Moreover, each giant edge added in Step 1 carries at most $t$ possible labels. 
   Hence, the total number of all possible rooted pseudoforests  $\Gamma_t$ constructed in Step 1 is at most:
    \begin{align}
        \binom{n_t}{a_t}\left(2t n_t\right)^{a_t}. \label{eq:pseudoam} 
    \end{align} 
   \item 
The total number of different ways of splitting is at most 
   \begin{align}
        \binom{n_t-a_t}{b_t} n_t^{b_t} \label{eq:pseudobm}.
   \end{align}
   \item The total number of different ways of forward bridging is at most:
   \begin{align}
       \binom{n_t-a_t-b_t  }{c_t}\left(\sum_{\ell<t} \ell n_{\ell}\right)^{c_t}  \label{eq:pseudocm}.
   \end{align}
   \item  The total number of different ways of backward bridging is at most:
   \begin{align}
      \binom{n_t-a_t-b_t-c_t }{d_t } \left( t n_{2t} \right)^{d_t}   \label{eq:pseudodm}.
   \end{align}
\end{enumerate}
Combining \prettyref{eq:pseudoam}, \prettyref{eq:pseudobm}, \prettyref{eq:pseudocm}, and \prettyref{eq:pseudodm}, we conclude  that
the total number of possible
 output backbone graphs $\Gamma$ with input parameter $ (\mathbf{a}, \mathbf{b}, \mathbf{c},\mathbf{d})$ is at  most 
\begin{align*}
 \prod_{t=1}^k   \indc{b_t=0 \text{ for odd } t } \indc{d_t=0 \text{ if } 2t>k } 
  \binom{n_t}{a_t, \, b_t,\, c_t, \, d_t } \left(2t n_t\right)^{a_t}n_t^{b_t} \left(\sum_{\ell<t} \ell n_{\ell}\right)^{c_t}\left( tn_{2t} \right)^{d_t}.
\end{align*}
Note that for each output backbone graph, the total number of edges in the corresponding orbit graph $H$ satisfies
 $$
 e(H) \ge \sum_{t=1}^k t \left(a_t+b_t/2 +c_t+2d_t \right).
 $$
 Combining the above two displays, we obtain
\begin{align*}
  \sum_{H \in \calH_k}  s^{2e(H) }  & \le \sum_{\mathbf{a}, \mathbf{b}, \mathbf{c},\mathbf{d}} \;
  \prod_{t=1}^k  
  \indc{b_t=0 \text{ for odd } t }  \indc{d_t=0 \text{ if } 2t>k }
 \binom{n_t}{a_t, \, b_t , \, c_t, \, d_t }  \\
 &~~~~ \times  \left(2tn_t\right)^{a_t} n_t^{b_t    } \left(\sum_{\ell<t} \ell n_{\ell}\right)^{c_t}\left( tn_{2t} \right)^{d_t}
 \times s^{2t a_t + t b_t   + 2t d_t + 4t d_t} \\
 & \le \prod_{t=1}^k 
 \left( 1 +   s^{t} n_{t} \indc{t:\mathrm{even}}   + 2t n_t s^{2t} + s^{2t} \sum_{\ell < t} \ell n_{\ell}+ s^{4t} t n_{2t}  \indc{2t \le k}  \right)^{n_t}\\
 & \le \prod_{t=1}^k   \left( 1 +  s^{t} n_{t}  \indc{t:\mathrm{even}}  +  2 s^{2t} \sum_{\ell \le  t} \ell n_{\ell}+ s^{4t} t n_{2t}  \indc{2t \le k}  \right)^{n_t},
\end{align*}
completing the proof of \prettyref{thm:jk}.

\section{Conclusion and Open Questions}
In this paper, we formulate the general problem of testing network correlation and characterize the statistical detection limit.
For both Gaussian-weighted complete graphs and dense \ER graphs, we determine the sharp threshold at which the asymptotic optimal  testing error probability jumps from $0$ to $1$. For sparse \ER graphs, we determine the threshold within a constant factor. The proof of the impossibility results relies on a delicate application of the truncated second moment method, and in particular, leverages the pseudoforest structure of subcritical \ER graphs in the sparse setting. We conclude the paper with a few important open questions.
\begin{enumerate}
\item In a companion paper~\cite{MWXY20}, we show that a polynomial-time test based on counting trees achieves strong detection when the average degree $np\ge n^{-o(1)}$ and
the correlation $\rho \ge c$ for an explicit constant $c$. In particular, this result combined with our negative results in~\prettyref{thm:er} imply that the detection limit is attainable in polynomial-time up to a constant factor in the sparse regime when $np=\Theta(1).$
However, achieving the optimal detection threshold in  polynomial time remains largely open. 
\item 
It is of interest to study the detection limit under general weight distributions. Our proof techniques are likely to work beyond the Gaussian Wigner and \ER graphs model. For example, for general distributions $P$ and $Q$, 
as shown in the proof of~\prettyref{prop:cycle}, the second moment is determined by the eigenvalues of the kernel operator 
defined by  the likelihood ratio $L(x,y)=\frac{P(x,y)}{Q(x,y)}$. 
Another interesting direction is testing correlations between hypergraphs. 
\item Another important open problem is to determine the  sharp threshold for detection in the sparse \ER graphs with $p=n^{-\Omega(1)}$. In particular, to improve our positive result,
one may need to analyze  a more powerful test statistic beyond QAP. 
For the negative direction,
one needs to consider the case where the intersection graph $A \wedge B^\pi$ is 
supercritical and a more sophisticated conditioning beyond the pseudoforest structure may be required.

\end{enumerate}

\appendix

\section{Supplementary Proof for Sections \ref{sec:intro}, \ref{sec:uncond} and \ref{sec:lb-dense}} \label{sec:supp_1_3_4}
\subsection{Proof for \prettyref{rmk:weak_detection}}
\label{sec:pf-weak_detection}
In this subsection, we show that in the non-trivial regime of 
$p=\omega(1/n^2)$ and $p=1-\Omega(1)$, weak detection can be achieved by comparing the number of edges of the two observed graph, provided that $s=\Omega(1)$. 
To see this, 
let $X,Y$ denote the total number of edges in the observed graphs $\overline{G}_1$ and $\overline{G}_2$, respectively.
Under  $\calH_0$ (resp.\  $\calH_1$), $X-Y$ is a sum of $m$ $\iid$ random variables that equal to $-1$ or $+1$ with the equal probability $ps(1-ps)$ (resp.\ $ps(1-s)$) and $0$ otherwise. Using Gaussian approximation, this essentially reduces to testing $\calN(0,1)$ vs. $\calN\left(0,\frac{1-ps}{1-s}\right)$, which can be non-trivially separated as long as $s$ is non-vanishing.

Formally,  assume $s\le 1-\Omega(1)$ without loss of generality and  let $m=\binom{n}{2}$. Then for some $\tau$, 
\begin{align*}
    & \calQ\left(|X-Y| \ge \tau\right)-\calP \left(|X-Y| \ge \tau\right) \\
    & \overset{(a)}{\ge} \Prob\left(|\calN(0,2m ps(1-ps))| \ge \tau\right)-\Prob\left(|\calN(0,2 m ps(1-s))|\ge \tau\right) - O\left(\frac{1}{\sqrt{m p }}\right)\\
    &\overset{(b)}{=} \mathrm{TV} \left(\calN\left(0,1\right),\calN\left(0,\frac{1-ps}{1-s}\right)\right)  - O\left(\frac{1}{\sqrt{m p }}\right) 
    \overset{(c)}{=} \Omega(1),
\end{align*}
where $(a)$ follows from Berry-Esseen theorem \cite[Theorem 5.5]{petrov} and the assumption that $ \Omega(1) \le s \le 1-\Omega(1)$;
$(b)$ holds for some $\tau$ because the likelihood ratio between two centered normals is monotonic in $|x|$;
$(c)$ holds because $(1-ps)/(1-s)$ is bounded away from $1$ and $mp = \omega(1)$ under assumptions that $s = \Omega(1)$, $p$ is bounded away from $1$, and $n^2p= \omega(1)$.

\subsection{Proof of \prettyref{prop:n1n2_weak_detection}} \label{sec:n1n2_weak_detection}
\begin{proof}
Using \prettyref{eq:N1N2}, we have $N_2 = \binom{n_2}{2}\times 2+n_1 n_2 + n_4 \le (n_2 + n_1) n_2 + n_4 \le n n_2 + n $. Therefore, 
\begin{align*}
    & \Expect_{\pi\indep{}\tpi}\left [ \exp \left( \mu n_1^2 + \nu n_1 +  \tau n_2 + \tau^2 N_2  \right) \indc{a \le n_1 \le b } \right] \\
    & \le \Expect_{\pi\indep{}\tpi}\left [ \exp \left( \mu n_1^2 + \nu n_1 + \left( \tau + \tau^2 n\right)n_2 + \tau^2 n \right) \indc{a \le n_1 \le b } \right] \\
    & = \left( 1+o(1)\right)  \Expect_{\pi\indep{}\tpi} \left [ \exp \left( \mu n_1^2 + \nu n_1 +  (\tau+ \tau^2 n) n_2 \right) \indc{ a \le n_1  \le b} \right],
\end{align*}
where the last equality holds because $n\tau^2 = o(1)$ by assumption. 
Pick $\eta$ such that $\omega(1) \le \eta \le \sqrt{n}$. We decompose
$
 \Expect_{\pi\indep{}\tpi} \left [ \exp \left( \mu n_1^2 + \nu n_1 +  (\tau+ \tau^2 n) n_2   \right) \indc{ a \le n_1  \le b} \right]
$
as  $\termI+\termII$, where:
\begin{align*}
    \termI 
    & = \Expect_{\pi\indep{}\tpi}\left [ \exp \left( \mu n_1^2 + \nu n_1 +  \left( \tau +\tau^2 n\right) n_2 \right) \indc{a \le n_1 \le b } \indc{0 \le n_1< \eta}\right]
    \\
     \termII 
    & = \Expect_{\pi\indep{}\tpi}\left [ \exp \left( \mu n_1^2 + \nu n_1 +  \left( \tau +\tau^2 n\right) n_2 \right) \indc{a \le n_1 \le b } \indc{n_1\ge  \eta}  \right].
\end{align*}
\begin{itemize}
    \item To bound $\termI$, we first apply the total variation bound \prettyref{eq:cycle_decomposition_poisson_distribution_total_v} in \prettyref{lmm:cycle_decomposition_poisson_distribution_total_v} and get
        $$
        \mathrm{TV} \left(  \calL\left( n_1, n_2\right) , \calL\left( Z_1, Z_2 \right) \right) \le F\left( \frac{n}{2} \right),
        $$
        where $Z_1\sim \Pois(1)$ and $Z_2 \sim \Pois(\frac{1}{2})$ are independent Poisson random variables, and $\log F(n/2)=-(1+o(1)) \frac{n}{2}\log \left(\frac{n}{2}\right)$.  
        Then, we have 
        \begin{align*}
            \termI 
            & \le \Expect\left [ \exp \left( \mu Z_1^2 + \nu Z_1 + \left( \tau +\tau^2 n\right) Z_2\right) \indc{a \le Z_1 < b}\right]  + 2 F\left(\frac{n}{2} \right) \exp \left(\left(\mu \eta + \nu\right) \eta + \left( \tau +\tau^2 n \right)n\right) \\
            & =   \Expect\left [ \exp \left( \mu Z_1^2 + \nu Z_1 + \left( \tau +\tau^2 n\right) Z_2\right) \indc{a \le Z_1 < b}\right] + o(1),
        \end{align*}
    where the last equality holds by the claim that $\left(\mu \eta + \nu\right) \eta + \left( \tau +\tau^2 n \right)n = o(n\log n)$. Indeed, note that $\mu b +\nu +2 -\log b \le 0$ and $\omega(1)\le b \le n$.
    It follows that $\mu \le \log(b)/b=o(1)$ and $\nu \le \log b \le \log n$. Moreover, 
    $\tau^2=o\left(\frac{1}{n}\right)$ and $\omega(1) \le \eta \le \sqrt{n}$. Hence the claim follows. 
    \item To bound $\termII$, we apply \prettyref{eq:jointdis2} in \prettyref{lmm:jointpmf}: 
        \begin{align*}
             \termII 
            & \le \Expect \left [\exp \left( \mu Z_1^2 + \nu Z_1 +  \left( \tau +\tau^2 n\right) Z_2\right) \indc{\eta \le Z_1 \le b } \right]e^{\frac{3}{2}}.
        \end{align*}
\end{itemize}
By applying the moment generating function $ \Expect_{X\sim  \Pois(\lambda)} \left[ \exp\left(tX\right)\right] = \exp\left(\lambda\left(e^t-1\right)\right)$,
we then get that $ \Expect \left [ \exp \left(\left( \tau +\tau^2 n\right) Z_2\right)\right]= 1+o(1)$ given $\tau^2 = o\left(\frac{1}{n}\right)$.
Therefore, we have 
        \begin{align*}
            \termI & \le \Expect \left [ \exp \left( \mu Z_1^2 + \nu Z_1 \right) \indc{a \le Z_1 \le b } \right] (1+o(1)) + o(1)\\
            \termII & \le  \Expect \left [\exp \left( \mu Z_1^2 + \nu Z_1 \right) \indc{\eta \le Z_1 \le b } \right] (1+o(1))e^{\frac{3}{2}}.
        \end{align*}



The following intermediate result is proved in the end.
\begin{lemma}\label{lmm:exp_poisson}
Assume $\alpha, \beta \ge  0$, $\alpha m +\beta +2 - \log m \le 0$ for some $1 \le m \le n$ such that $m=\omega(1)$, and $Z\sim \Pois\left(\lambda\right)$ for some $0<\lambda\leq 1$. 
\begin{itemize}
    \item If $\ell=\omega(1)$ and $\beta \le \log(\ell)-3$, 
    \begin{align}
    \Expect \left [ \exp \left( \alpha Z^2 + \beta Z\right) \indc{\ell \le Z \le m }  \right] = o(1). \label{eq:z_ell_omega_1}
    \end{align}
    \item If $\ell=0$ and $\beta=o(1)$, 
    \begin{align}
     \Expect \left [ \exp \left( \alpha Z^2 + \beta Z\right) \indc{\ell \le Z \le m }  \right] = 1+ o(1) . 
     \label{eq:z_ell_0}
    \end{align}
\end{itemize}
\end{lemma}

If $a= \omega(1)$, applying \prettyref{eq:z_ell_omega_1} yields $\termI=o(1)$ and $\termII=o(1)$, hence the desired \prettyref{eq:n1n2_a_omega_1}.
If $a=0$, applying both \prettyref{eq:z_ell_omega_1} and \prettyref{eq:z_ell_0}, we get $\termI=1+o(1)$, $\termII=o(1)$, and hence the desired \prettyref{eq:n1n2_a_0}.

Finally to show \prettyref{eq:n1n2_main}, note that using \prettyref{eq:N1N2},  we have $N_1 = \binom{n_1}{2} + n_2 \le n_1^2/2+ n_2$. Then, we have
$\Expect_{\pi\indep{}\tpi}\left [ \exp \left( \tau N_1 + \tau^2 N_2  \right)  \right] 
    \le \Expect_{\pi\indep{}\tpi}\left [ \exp \left( \tau n_1^2/2 +  \tau n_2 + \tau^2 N_2  \right)  \right]$.
Thus the desired bound follows from applying  \prettyref{eq:n1n2_a_0} with $\mu=\tau^2/2$, $\nu=0$, $a=0$, and $b=n$. 
\end{proof}


\begin{proof}[Proof of \prettyref{lmm:exp_poisson}]
To show \prettyref{eq:z_ell_omega_1}, note that
\begin{align*}
    \Expect \left [ \exp \left( \alpha Z^2 + \beta Z\right) \indc{\ell \le Z \le m }  \right]
    & = e^{-\lambda} \sum_{a_1=\ell}^{m}  \frac{\lambda^{a_1}\exp\left( \alpha a_1^2 + \beta a_1 \right)}{a_1!} \\
    & \overset{(a)}{\le} e^{-\lambda} \sum_{a_1=\ell}^{m}  \lambda^{a_1} \exp\left( \alpha a_1^2 + \beta a_1 - a_1 \log a_1 +a_1\right) )  \\
    & \overset{(b)}{=}e^{-\lambda} \sum_{a_1=\ell}^{m} \lambda^{a_1} \exp(-a_1) = o(1),
\end{align*}
where $(a)$ holds due to $a_1!\ge   (a_1/e)^{a_1}$ for $a_1\ge  1$; $(b)$ follows from 
the claim that $\alpha x+\beta +2 -\log x \le 0$ for $\ell \le x \le m$.
  To see this, define $f(x) = \frac{\log x-\beta-2}{x}$.
Then by assumption $f(m) \ge \alpha$. Since $f'(x) \le 0$  for $x\ge  e^{\beta+3}$ and $\ell \ge e^{\beta +3}$ by assumption,   
it follows that $f(x) \ge \alpha$ for all $\ell \le x \le m$.

Next we prove \prettyref{eq:z_ell_0}. Since $\alpha m + \beta +2 - \log m \le 0$ for $m=\omega(1)$, we have $\alpha \le \log(m)/m= o(1)$.  
Moreover, $\beta=o(1)$ by assumption. Therefore there exists some $t= \omega(1)$ such that $\ell=0< t \le m$ and $ \alpha t^2 + \beta t =o(1)$.
 Then
\begin{align*}
    \Expect \left [ \exp \left( \alpha Z^2 + \beta Z\right) \indc{\ell \le Z \le m }  \right]
    & =  \Expect \left [ \exp \left( \alpha Z^2 + \beta Z\right) \indc{0 \le Z \le t }  \right]+ \Expect \left [ \exp \left( \alpha Z^2 + \beta Z\right) \indc{t <  Z \le m }  \right]\\
    & \overset{(a)}{=} e^{-\lambda} \sum_{a_1=0}^{t} \frac{\lambda^{a_1}\exp\left(\alpha a_1^2 +\beta a_1 \right)}{a_1!} + o(1) \\
    &  \overset{(b)}{=} e^{-\lambda} \sum_{a_1=0}^t \frac{\lambda^{a_1}\left(1+o(1)\right)}{a_1!}+o(1) \overset{(c)}{=} 1+o(1),
\end{align*}
where $(a)$ holds by \prettyref{eq:z_ell_omega_1}; $(b)$ holds because $ \alpha a_1^2 + \beta a_1 =o(1)$ 
for $0 \le a_1 \le t$; $(c)$ holds because 
$\sum_{a_1=0}^t \frac{\lambda^{a_1}}{a_1!}= e^{\lambda}+o(1)$ for $t=\omega(1)$.  
\end{proof}
\subsection{Sharp threshold for dense \ER graphs}  \label{sec:er_dense} 
In this subsection, we focus on the case of dense parent graph whose edge density satisfies
\begin{align}
p \le 1- \Omega(1)  \quad \text{ and } \quad p=n^{-o(1)}. \label{eq:dense_regime_er}
\end{align}
Recall that \prettyref{thm:unconditional_second_moment} implies that weak detection is impossible, if $\rho^2 =\frac{s^2 (1-p)^2 }{ (1-ps)^2 } \le (2-\epsilon) \frac{\log n}{n}$.
We improve over this condition 
by showing that if 
\begin{align}
np s^2 \left( \log \frac{1}{p} -1 + p \right) \le (2-\epsilon) \log n, 
\label{eq:lb-denseER}
\end{align}
then weak detection is impossible. 
This completes the impossibility proof for \prettyref{thm:er} in the dense regime of~\prettyref{eq:dense_regime_er}.

Without loss of generality, we assume that \prettyref{eq:lb-denseER} holds with equality (otherwise one can further subsample the edges), i.e.,
\begin{align}
np s^2 \left( \log \frac{1}{p} -1 + p \right) = (2-\epsilon) \log n. \label{eq:sharp_lower_bound_er_dense}
\end{align}
In the dense graph regime~\prettyref{eq:dense_regime_er}, 
under assumption \prettyref{eq:sharp_lower_bound_er_dense}, 
we have 
\begin{align}
nps^2=\omega(1)  \quad \text{ and } \quad s=n^{ -1/2+o(1) }. \label{eq:dense_nps} 
\end{align}

As argued in~\prettyref{sec:obstruction}, analogous to the Gaussian case, the unconditional second moment explodes when $\rho^2 \ge (2+\epsilon) \frac{\log n}{n}$,
due to the obstruction of fixed points, or more precisely, an atypically large magnitude of $\prod_{O \in \calO_1} X_O$. 
By \prettyref{eq:Lba-ER1},
\begin{align*}
L(a,b) 
&= \frac{1-\eta}{1-ps} \pth{\frac{1-s}{1-\eta}}^{a+b} \pth{\frac{s(1-\eta)}{\eta(1-s)}}^{ab},
\end{align*}
where $\eta \triangleq \frac{ps(1-s)}{1-ps}$,
and then by \prettyref{eq:X_ij}, 
\begin{align}
\prod_{O \in \calO_1} X_O & = \prod_{ i<j \in F } X_{ij}   \nonumber \\
&  =\left(\frac{1-\eta}{1-p s}\right)^{2 \binom{n_1}{2} }  \left(\frac{1-s }{1-\eta } \right)^{ 2 e_A(F) + 2 e_{ B^{\pi}} \left( F \right) } 	
	\left(\frac{s\left(1-\eta\right)}{\left(1-s\right)\eta}\right)^{2 e_{A\wedge B^{\pi}}(F)},  \label{eq:X_O_F_ER}
\end{align}
where $F$ denotes the set of fixed points of $\sigma = \pi^{-1}\circ\tpi$, and
$$
e_A(F) = \sum_{i<j \in F} A_{ij}, \quad e_{ B^{\pi}}  \left( F \right) =\sum_{i<j \in F} B_{\pi(i) \pi(j) }, \quad 
\text{ and } \quad e_{A\wedge B^{\pi}}(F) = \sum_{i<j \in F} A_{ij} B_{\pi(i) \pi(j)}.
$$ 
Since  $s>\eta$, it follows that  $\frac{1-s }{1-\eta }  <1$ and  $\frac{s\left(1-\eta\right)}{\left(1-s\right)\eta}>1$. Thus when 
$e_{A\wedge B^{\pi}}(F)$ is atypically large, 
$\prod_{ i<j \in F } X_{ij} $ becomes enormously large, driving the unconditional second moment to explode. 
Hence, we aim to truncate  $\prod_{ i<j \in F } X_{ij}$ by conditioning on the maximum possible value of $e_{A\wedge B^{\pi}}(F)$
under the planted model $\calP$ when $|F|=n_1$ is large.

Specifically, given $2 \le k \le n$, define
\begin{align}
\zeta(k) \triangleq \binom{k}{2} ps^2 \exp\sth{  1+ W \left( \frac{ 2\log (2en/k)  }{ e (k-1) ps^2} - \frac{1}{e}  \right)}, \label{eq:zeta_def}
\end{align}
where $W$ is the Lambert W function defined on $[-1,\infty)$ 
as the unique solution of $W(x) e^{W(x)} =x$ for $x \ge -1/e$.  
Let 
\begin{align}
\alpha \triangleq p \left( \log \frac{1}{p} - 1+ p \right).  \label{eq:def_alpha}
\end{align}
For each  $S\subset [n]$, define the event
\begin{align}
\calE_S \triangleq \bigg\{ & (A, B, \pi):  e_A(S), e_{ B^{\pi}}(S)  \ge \binom{|S| }{2} ps - \sqrt{2 \binom{|S|}{2}  ps |S| \log \frac{2en}{ |S|}  }, 
 e_{A\wedge B^{\pi}}(S) \le   \zeta \left( |S| \right) \bigg \}
 \label{eq:cond_high_prob_dense_ER}.
\end{align} 
We condition on the event
\begin{align}
\calE \triangleq \bigcap_{S \subset[n]: \alpha n \le |S| \le n } \calE_S  \label{eq:calE-er}.
\end{align}
As previously explained in \prettyref{sec:gaussian_second_moment} for the Gaussian case, 
since we cannot condition on the set $F$, in order to truncate $\prod_{ i<j \in F } X_{ij} $,  
$\calE$ is defined as the intersection of $\calE_S$ over all subsets $S$ with $|S| \ge \alpha n$, 
so that $\calE$ implies $\calE_F$ whenever $|F| \ge \alpha n$. 
However, in contrast to the Gaussian case, it is no longer true that
$e_{A \wedge B^{\pi}}(S)=(1+o(1))  \binom{|S|}{2} ps^2$ uniformly for all $S$ with $|S| \ge \alpha n$.
Thus, we condition on $e_{A\wedge B^{\pi}}(S) \le   \zeta \left( |S| \right) $,
where $\zeta(k)$ in \prettyref{eq:zeta_def} is defined according to the large-deviation behavior of $\Binom\left( \binom{k}{2}, ps^2 \right)$,
as we will see in the next lemma. 
 

The following lemma proves that 
$\calE$ holds with high probability under the planted model $\calP$.
\begin{lemma}\label{lmm:cond_high_prob_dense_ER}
Suppose $\alpha n =\omega(1)$. Then $\calP( (A, B, \pi) ) \in \calE )=1 - e ^{-\Omega(\alpha n) }$. 
\end{lemma} 
\begin{proof}
Fix an integer $\alpha n \le k \le n$ and let $m=\binom{k}{2}$. 
Let
$$
t = \sqrt{2 m ps \log (1/\delta) }, \qquad t' = m ps^2 \exp\sth{  1+ W \left( \frac{ \log (1/\delta)  }{ e m ps^2} - \frac{1}{e}  \right)},
$$
for a parameter $\delta$ to be specified later.

Fix a subset $S \subset[n]$ with $|S|=k$. As $e_A(S) \sim \Binom(m, ps)$ and $e_{ B^{\pi}} (S) \sim \Binom(m,ps)$, using Chernoff's bound
for Binomial distributions~\prettyref{eq:chernoff_binom_left}, we get that 
with probability at least $1-2\delta$, $e_A(S) \ge mps - t$ and $e_{ B^{\pi}}(S) \ge mps-t$. 
Moreover, since $e_{A\wedge B^{\pi}}(S) \sim \Binom(m,ps^2)$, Using the multiplicative Chernoff bound for Binomial distributions~\prettyref{eq:chernoff_binom_right_Lambert}, we get that 
with probability at least $1-\delta$, $e_{A\wedge B^{\pi}}(S) \le  t' $. 

Now, there are $\binom{n}{k} \le \left( \frac{en}{k} \right)^k$ different choices of $S \subset[n]$ with $|S|=k$. 
Thus by choosing $1/\delta= \left( \frac{2en}{k} \right)^k$ and applying the union bound,
we get that 
with probability at least $1- 3\sum_{k=\alpha n}^n 2^{-k}=1-e^{-\Omega(\alpha n )}$,  
$e_A(S) \ge m p s - t$,  $e_{ B^{\pi}}(S) \ge m p s-t$, and $e_{A\wedge B^{\pi}}(S) \le t'$
for all  $S \subset[n]$ with $|S|=k$ and all  $\alpha n \le k \le n$.
\end{proof}

Now we compute the conditional second moment. 
By \prettyref{lmm:cond_high_prob_dense_ER},  it follows from  \prettyref{eq:condition_second_moment} that 
\begin{align*}
     \Expect_{\calQ}\left[\left(\frac{\calP' (A, B) }{\calQ(A, B) } \right)^2\right] 
     =  \left( 1+o\left(1\right)\right) \Expect_{\pi\ci\widetilde{\pi}}\left[
    \Expect_{\calQ}
    \left[\prod_{\orbit \in \calO } X_{\orbit}\indc{\left(A,B,\pi\right)\in \calE}\indc{ \left(A,B,\widetilde{\pi} \right)\in \calE}\right]\right].
\end{align*}

To proceed further, we fix $\pi, \tpi$ and  separately consider the following two cases. Recall that $\alpha$ is defined in~\prettyref{eq:def_alpha}.

{\bf Case 1}:  $n_1 \le \alpha n$. In this case, we simply drop the indicators and use the unconditional second moment:
$$
 \Expect_{\calQ}
    \left[\prod_{\orbit \in \calO } X_{\orbit}\indc{\left(A,B,\pi\right)\in \calE}\indc{ \left(A,B,\widetilde{\pi} \right)\in \calE}\right]
 \le  \Expect_{\calQ}
    \left[\prod_{\orbit \in \calO } X_{\orbit} \right]
  =\prod_{\orbit \in \calO }  \left(1+\rho^{2|\orbit|} \right) ,
$$       
where the equality follows from \prettyref{eq:ercycle}.

\medskip 
{\bf Case 2}:  $n_1 > \alpha n$.
In this case, we have that 
\begin{align*}
 \Expect_{\calQ}
    \left[\prod_{\orbit \in \calO } X_{\orbit}\indc{\left(A,B,\pi\right)\in \calE} \indc{ \left(A,B,\widetilde{\pi} \right)\in \calE}\right] 
& \overset{(a)}{\le}  \Expect_{\calQ}
    \left[\prod_{\orbit \in \calO } X_{\orbit} \indc{\left(A,B,\pi\right)\in \calE_F } \right] \\ 
   & \overset{(b)}{=} \prod_{\orbit \notin \calO_1 } \Expect_{\calQ}\left[X_{\orbit} \right]
 \Expect_{\calQ} \left[ \prod_{\orbit \in \calO_1}   X_{\orbit}  \indc{\left(A,B,\pi\right)\in \calE_F} \right] \\
 & \overset{(c)}{=} \prod_{\orbit \notin \calO_1 } \left(1+\rho^{2|\orbit|} \right)
  \Expect_{\calQ}   \left[  \prod_{ i<j \in F} X_{ij} \indc{\left(A,B,\pi\right)\in \calE_F }  \right],
\end{align*}
where $(a)$ holds because by definition \prettyref{eq:calE-er}, $\calE\subset \calE_F$ when $n_1 > \alpha n$;
$(b)$ holds because $X_\orbit$ is a function of
$( A_{ij}, B_{\pi(i) \pi(j) })_{(i,j) }$
 that are independent across different $\orbit \in \calO$, 
and $\indc{ (A,B,\pi) \in \calE_F  }$ only depends on $\left\{ (A_{ij}, B_{\pi(i) \pi(j) })_{(i,j) \in \orbit}: \orbit \in \calO_1 \right\} $;
$(c)$ follows from \prettyref{eq:ercycle}.

Let $m=\binom{n_1}{2}$.  Under event $\calE_F$, 
we have that 
$e_A(F) \ge (1+o(1))m ps $ and
$e_{ B^{\pi}}(F)  \ge (1+o(1)) m ps $.
This is because by \prettyref{eq:cond_high_prob_dense_ER},
$$ 
\frac{ n_1 \log(n/n_1) }{ m ps}  = \frac{ 2 \log (n/n_1)} { (n_1 -1) ps } 
\overset{(a)}{=} O \left(  \frac{ \log (1/\alpha)}{ \alpha n  ps} \right) 
\overset{(b)}{=} \Theta\left( \frac{1}{np^2 s } \right) \overset{(c)}{=} o(1),
$$ 
where $(a)$ holds because $n_1 \ge \alpha n$; 
$(b)$ holds due to $\frac{1}{\alpha} \log \frac{1}{\alpha} =\Theta(1/p)$;
$(c)$ holds because  $p=n^{o(1)}$ and $s=n^{-1/2+o(1) }$. 
Moreover,  under event $\calE_F$, 
$e_{A\wedge B^{\pi}}(F)  \le \zeta(n_1)$.

Let 
\begin{align}
\gamma\equiv \gamma(n_1) =\frac{ 2\log (2en/n_1)  }{ (n_1-1) ps^2}.  \label{eq:def_gamma}
\end{align}
Then  $\zeta(n_1) =m ps^2 \exp\left( 1+ W\left(\frac{\gamma-1}{e}\right)  \right)$. 
The following lemma characterizes the behavior of $\zeta(n_1)$ in the following three asymptotic regimes depending on $\gamma$. 
\begin{lemma}\label{lmm:flucuation_W_function}
\begin{itemize}
\item If $\gamma=o(1)$, 
$\zeta(n_1) = (1+o(1))m ps^2$. 
\item  If $\gamma = \Theta(1)$, $\zeta(n_1)= \Theta( m ps^2 )$. 
In particular, for all $n_1 \ge \alpha n$,  we have $\zeta(n_1)= o(ms^2)$. 

\item If $\gamma=\omega(1)$, 
$\zeta (n_1) \le (e+o(1)) m ps^2 \gamma/ \log \gamma$. In particular,   for all $n_1 \ge \alpha n$,  $\zeta(n_1) = o(ms^2)$.
\end{itemize}
\end{lemma}
\begin{proof}
\begin{itemize}
    \item $\gamma=o(1)$. Using the approximation $W(\frac{\gamma-1}{e}) = -1 + \sqrt{2\gamma} + O(\gamma)$~\cite[eq.(4.22)]{corless1996lambertw}, we get that  $\zeta(n_1) = (1+o(1))m ps^2$. 
    \item $\gamma = \Theta(1)$. In this regime, $W(\frac{\gamma-1}{e}) =\Theta(1)$ and thus $\zeta(n_1)= \Theta( m ps^2 )$. In particular, for all $n_1 \ge \alpha n$,  we have $\zeta(n_1)=\Theta(mps^2) = o(ms^2)$. To see this, note that $\frac{1}{\alpha }\log \frac{1}{\alpha} = \Theta(1/p)$. Thus, $\gamma =  O \left(  \frac{ \log (1/\alpha)}{ \alpha n  p s^2 } \right)  =  O\left( \frac{1}{n p^2 s^2 } \right)$. Hence $p=O\left(  \frac{1}{n p s^2 \gamma } \right) =o(1)$, as $nps^2 =\omega(1)$ and $\gamma=\Theta(1)$. 
    \item $\gamma=\omega(1)$. Using the approximation $W(x) = \log x - \log\log x +o(1)$ as $x\to \infty$~\cite[Theorem 2.7]{hoorfar2008inequalities}, we get that
    $\zeta (n_1) \le (e+o(1)) m ps^2 \gamma/ \log \gamma=  (2e+o(1)) n_1 \log (2en/n_1) /\log(\gamma)$. 
    Moreover, for all $n_1 \ge \alpha n$, 
    $$
    \frac{ \zeta(n_1)}{ms^2} = O\left( \frac{n_1 \log (n/n_1) }{ ms^2 \log \gamma }  \right) = O \left(  \frac{ \log (1/\alpha)}{ \alpha n  s^2 \log \gamma } \right) = \Theta\left( \frac{1}{nps^2 \log \gamma} \right) = o(1).
    $$
\end{itemize}

\end{proof}

In view of~\prettyref{lmm:flucuation_W_function}, we get that $\zeta(n_1) \le mps^2 + o(ms^2)$ for all $n_1 \ge \alpha n$.
For ease of notation, we henceforth write $\zeta(n_1)$ simply as $\zeta$. 

It follows from \prettyref{eq:X_O_F_ER}
 that 
\begin{align*}
& \Expect_{\calQ}   \left[  \prod_{ i<j \in F} X_{ij}  \indc{\left(A,B,\pi\right)\in \calE_F}  \right] \\
&=  \left(\frac{1-\eta}{1-p s}\right)^{2 \binom{n_1}{2} } 
\Expect_{\calQ}   \left[  \left(\frac{1-s }{1-\eta } \right)^{2e_A(F) + 2e_{ B^{\pi}}(F) } 	
	\left(\frac{s\left(1-\eta\right)}{\left(1-s\right)\eta}\right)^{2 e_{A\wedge B^{\pi}}(F)}  \indc{\left(A,B,\pi\right)\in \calE_F}    \right] \\
 & \le \left(\frac{1-\eta}{1-p s}\right)^{2m}  \left(\frac{ 1-s }{1-\eta}\right)^{(4+o(1)) mps}
 \Expect_\calQ\qth{ \left(\frac{s\left(1-\eta\right)}{\left(1-s\right)\eta}\right)^{2 e_{A\wedge B^{\pi}}(F) }
 \indc{ e_{A\wedge B^{\pi}}(F) \le \zeta} } \\
 & =  \exp\sth{ -(2+o(1)) mps^2 (1-p )}
 \Expect_\calQ\qth{ \left(\frac{s\left(1-\eta\right)}{\left(1-s\right)\eta}\right)^{2 e_{A\wedge B^{\pi}}(F) }
 \indc{ e_{A\wedge B^{\pi}}(F) \le \zeta} },
\end{align*}
where the last equality holds because  in view of $\eta= \frac{ps(1-s)}{1-ps}$ and $s=o(1)$, 
\begin{align*}
\log \frac{1-\eta}{1-ps} &= \log \left(1 + \frac{ps^2(1-p)}{ (1-ps)^2} \right) = (1+o(1)) ps^2 (1-p) \\
\log \frac{1-\eta}{1-s}  & = \log \left( 1 +  \frac{s(1-p) }{(1-s)(1-ps) } \right) = (1+o(1)) s(1-p).
\end{align*}

Let 
$$
u=\left( \frac{s\left(1-\eta\right)}{\left(1-s\right)\eta} \right)^2
=\left( \frac{(1-\eta) (1-ps) }{ p (1-s)^2 } \right)^2 = (1+o(1))  p^{-2}.
$$
Then for any $\lambda \in [0,1]$,
\begin{align*}
\Expect_\calQ\qth{ \left(\frac{s\left(1-\eta\right)}{\left(1-s\right)\eta}\right)^{2 e_{A\wedge B^{\pi}}(F) }
 \indc{ e_{A\wedge B^{\pi}}(F) \le \zeta} }
& \le  \Expect_{\calQ}   \left[  u^{   \lambda e_{A\wedge B^{\pi}}(F) + (1-\lambda)  \zeta  }   \right]\\
& = u^{(1-\lambda) \zeta } \left( 1 + p^2 s^2 ( u^{\lambda} -1 )   \right)^{m}.
\end{align*}
Optimizing over $\lambda \in[0,1]$, or equivalently, over $y=u^{\lambda} \in [1, u]$, we get (by taking the log and differentiating) that 
\[
\inf_{1 \le y  \le u} \pth{1 + p^2s^2 (y-1)}^m y^{-\zeta} = \left( \frac{m( 1-p^2s^2) }{m-\zeta }  \right)^{m} \left( \frac{\zeta (1-p^2s^2) }{ p^2s^2 (m-\zeta) } \right)^{-\zeta},
\]
where the infimum is achieved at $y^*=\frac{\zeta (1-p^2s^2) }{ p^2s^2 (m-\zeta) }$.
Note that  
$$
1 \le y^* \le u  \Longleftrightarrow mp^2s^2 \le \zeta \le \frac{ m u p^2 s^2 }{ 1+ up^2 s^2 - p^2 s^2 } .    
$$
Since $mps^2 \le \zeta \le m ps^2 + o(ms^2)$, $u= (1+o(1)) p^{-2}$, $s=o(1)$, and 
$p$ is bounded away from $1$, it follows that $y^* \in [1, u]$.  

%

Putting everything together, we get that
\begin{align*}
& 
\Expect_{\calQ}   \left[  \prod_{ i<j \in F} X_{ij} \indc{(A,B,\pi)\in\calE_F}   \right]
\\
& \le \exp\sth{ -(2+o(1)) mps^2 (1-p ) + m \log \frac{ m (1-p^2s^2) }{m-\zeta}  + \zeta \log \frac{(m-\zeta) up^2s^2 }{\zeta(1-p^2s^2)} } \\
& \overset{(a)}{\le} \exp\sth{ - m ps^2 (2-p ) + \zeta \left( \log (s^2) +o(1) \right) + m h(\zeta/m) } \\
& \overset{(b)}{\le} \exp\sth{ - m ps^2 (2-p ) + \zeta \log \frac{ems^2}{\zeta} +o(\zeta)  },
\end{align*}
where $h(x)=-x\log x - (1-x) \log (1-x) $ is the binary entropy function; 
$(a)$ holds because in view of $s=o(1)$, we have
$\zeta= ms^2 (p+o(1))$, and $u=(1+o(1))p^{-2}$,  $(m-\zeta) \log (1-p^2s^2) = - (1+o(1)) mp^2s^2$
and $\log(up^2)=o(1)$; 
$(b)$ holds due to $h(x)=x\log(e/x)+o(x)$ by the Taylor approximation.

\medskip 
Combining the two cases yields  that
\begin{align*}
 \Expect_{\calQ}\left[\left(\frac{\calP'(A,B) }{\calQ(A, B) } \right)^2\right] 
& \le   (1+o(1)) \expect{ \prod_{\orbit \in \calO }  \left(1+\rho^{2|\orbit|} \right) 
\indc{n_1 \le \alpha n}  } \\
& + (1+o(1)) \expect{ \prod_{\orbit \notin \calO_1 }  \left(1+\rho^{2|\orbit|} \right) 
\exp\sth{ \zeta \log \frac{ems^2}{\zeta} +o(\zeta) } \indc{n_1 > \alpha n} }.
\end{align*}

Recall that $\rho=\frac{s(1-p)}{1-ps} \le s$ and by \prettyref{eq:N1N2},
\begin{align*}
\prod_{\orbit \notin \calO_1 }  \left(1+\rho^{2|\orbit|} \right) 
& =
\left( 1+ \rho^2 \right)^{n_2} \prod_{k\ge 2}  
\left(1+ \rho^{2k} \right)^{N_k}  \\
& =(1+o(1)) \left( 1+\rho^2 \right)^{n_2} \left( 1+\rho^4  \right)^{N_2} \\
& \le (1+o(1)) \exp \left( s^2 n_2 + s^4 N_2 \right), 
\end{align*}
where  the first equality follows from \prettyref{eq:N1N2};
the second equality holds because $\prod_{k\ge 3}  
\left(1+ \rho^{2k} \right)^{N_k} \le \exp \left(n^2 \rho^6 /2\right) = 1+o(1)$  by \prettyref{eq:2lower2} and 
$\rho \le s =n^{-1/2+o(1)}$ in view of \prettyref{eq:dense_nps}. 
Similarly,
$$
\prod_{\orbit \in \calO_1 }  \left(1+\rho^{2|\orbit|} \right)  =  \left(1+\rho^{2} \right)^{\binom{n_1}{2}} \le  \exp \left( s^2 n_1^2 /2\right).
$$
Hence, 
\begin{align*}
 \Expect_{\calQ}\left[\left(\frac{\calP'(A,B) }{\calQ(A, B) } \right)^2\right] 
& \le   (1+o(1)) \expect{ \exp \left( s^2 n_1^2 /2 + s^2 n_2 + s^4 N_2 \right)
\indc{n_1 \le \alpha n}  } \\
& + (1+o(1)) \expect{ 
\exp\sth{ s^2 n_2 + s^4 N_2 - mps^2 (2-p ) + \zeta \log \frac{ems^2}{\zeta} +o(\zeta) } \indc{n_1 > \alpha n} }.
\end{align*}
We upper bound the two terms separately. For the first term, we apply \prettyref{eq:n1n2_a_0} in \prettyref{prop:n1n2_weak_detection} with $\mu=s^2/2$, $\nu=0$, $\tau=s^2$, $a=0$, and $b = \alpha n =  p(\log \frac{1}{p} - 1 +p)n = n^{1-o(1)}= \omega(1) $. By the assumption that $ n s^2 p(\log \frac{1}{p} - 1 +p)  \le (2-\epsilon)\log n/n$, it follows that 
$$
\mu b + \nu+2 -\log b  
=\frac{1}{2} n ps^2 \left(\log \frac{1}{p} - 1 +p \right) + 2 - (1-o(1)) \log n 
\le \frac{1}{2}\left( - \epsilon +o(1)  \right)  \log n +2 \le 0.
$$
Thus, it follows from \prettyref{eq:n1n2_a_0} in \prettyref{prop:n1n2_weak_detection} that
\begin{align*}
    \expect{ \exp \left( s^2 n_1^2 /2 + s^2 n_2 + s^4 N_2 \right) \indc{n_1 \le \alpha n}  } \le 1+o(1).
\end{align*}
%

For the second term, we further divide into three cases according to \prettyref{lmm:flucuation_W_function}. 
We define 
$$
\beta= \frac{ \log^2 (nps^2)}{nps^2} \quad \text{ and } \quad \beta' = \frac{ \log (nps^2)}{100 nps^2}.
$$
Recall $\gamma$ as defined in \prettyref{eq:def_gamma}.
By \prettyref{lmm:flucuation_W_function}, we have that
\begin{enumerate}[label=(\alph*)]
\item If $\beta n \le n_1\le n$, then $ \gamma = o(1)$ and $\zeta = (1+o(1)) mps^2 $;
\item If $ \beta' n \le n_1\le \beta n$, then $\gamma \le 200 +o(1)$ and $ \zeta=\Theta\left( mps^2\right)$; 
\item If $\alpha n \le n_1\le \beta' n$, then $\gamma \ge  200 +o(1)$ and $\zeta = O\left(n_1\log \left(n/n_1\right)/\log(\gamma)\right)$. 
\end{enumerate}

{\bf Case 2 (a) }:  $ \beta n \le n_1 \le n$. In this case, $\zeta=(1+o(1)) mps^2$ and hence
$$ 
\zeta \log \frac{ems^2}{\zeta} +o(\zeta) = (1+o(1)) mps^2 \log \frac{e}{p}. 
$$ 
Therefore, 
\begin{align*}
&\expect{ \exp\sth{ s^2 n_2 + s^4 N_2- mps^2 (2-p ) +  \zeta \log \frac{ems^2}{\zeta} +o(\zeta)  }\indc{n_1 \ge \beta n} }\\
& \le \expect{
\exp \sth{ s^2 n_2 + s^4 N_2  + \frac{1+o(1) }{2} n_1^2 p s^2  \left( \log \frac{1}{p}  -1 +p \right)  }\indc{n_1 \ge \beta n} }
=o(1),
\end{align*}
where the first inequality uses the fact that $m \le n_1^2/2$; 
the last equality holds by invoking \prettyref{eq:n1n2_a_omega_1} in \prettyref{prop:n1n2_weak_detection} with  $\mu= \frac{1+o(1) }{2} p s^2  \left( \log \frac{1}{p}  -1 +p \right)$, $\nu = 0$, $\tau=s^2$, $a=\beta n = \omega(1)$,  $b=n= \omega(1)$, and 
\begin{align*}
\mu b+\nu+2 - \log b = \frac{1+o(1) }{2} n p s^2  \left( \log \frac{1}{p}  -1 +p \right) + 2 - \log n =-\frac{1}{2} (\epsilon+o(1)) \log n +2 \le 0,
\end{align*}
where the last equality follows from the assumption that 
$nps^2 \left( \log \frac{1}{p}  -1 +p \right) = (2-\epsilon) \log n$. 

\medskip
When $p$ is a fixed constant bounded away from $1$,
as $\alpha=\Theta(1)$ and $\beta=o(1)$, we have $\beta=o(\alpha)$; thus the regime 
$\alpha n \le n_1 \le \beta n$ is vacuous. Thus henceforth we assume $p=o(1)$.

{\bf Case 2 (b) }:  $ \beta' n \le n_1 \le \beta n$.  
In this case,  $\gamma \le 200 +o(1)$ and 
$
\zeta = O\left( mps^2 \right),
$
so that 
$$ 
\zeta \log \frac{ems^2}{\zeta} +  o(\zeta)  
\le c_1 mps^2 \log \frac{1}{p}
$$ 
for a universal constant $c_1$. 
Therefore, 
\begin{align*}
&\expect{ \exp \sth{ s^2 n_2 + s^4 N_2  - mps^2 (2-p ) +  \zeta \log \frac{ems^2}{\zeta} +o(\zeta)  }\indc{ \beta' n \le n_1 \le \beta n} }\\
& \le \expect{
\exp \sth{ s^2 n_2 + s^4 N_2  + c_1 n_1^2 ps^2 \log \frac{1}{p}  } \indc{\beta'n\le   n_1 \le \beta n}} = o(1),
\end{align*}
where the last equality holds by invoking  \prettyref{prop:n1n2_weak_detection} with $\mu=c_1 ps^2 \log \frac{1}{p}$, $\nu=0$, $\tau=s^2$, $a=\beta'n = \omega(1)$, and $b=\beta n =\omega(1)$. 
Note that the conditions of \prettyref{prop:n1n2_weak_detection} for $a = \omega(1)$ are satisfied, in particular, $\mu b +2 - \log b \le 0$ for $n$ sufficiently large.
To see this, on the one hand,
$$
\mu b = c_1 \beta nps^2 \log \frac{1}{p}  = O \left( \beta \log n  \right)= o(\log n),
$$
where we used the fact that $nps^2 \log \frac{1}{p} = O(\log n)$ and $\beta=o(1)$; 
on the other hand, 
$\log(b) = \log (\beta n) =(1+o(1)) \log n$, as $\log (\beta)=o(\log n)$ by our choice
of $\beta$. 


\medskip 
{\bf Case 2 (c) }:  $ \alpha n \le n_1 \le \beta' n$.  In this case,
$\gamma \ge 200+o(1)$ and 
$$
\zeta = O \left( n_1 \log (n/n_1) / \log (\gamma)  \right), 
$$ 
so that 
$$ 
\zeta \log \frac{ems^2}{\zeta} +  o(\zeta)  
\le c_1 n_1 \frac{\log (n/n_1)}{ \log \gamma} \log \frac{n_1 s^2 \log \gamma}{ \log (n/n_1) }   \triangleq n_1 \psi
$$ 
for a universal constant $c_1$. 
Therefore, 
\begin{align*}
&\expect{ \exp \sth{ s^2 n_2 + s^4 N_2 - mps^2 (2-p ) +  \zeta \log \frac{ems^2}{\zeta} +o(\zeta)  }\indc{ \alpha n \le n_1 \le \beta' n} }\\
& \le \expect{
\exp \sth{ s^2 n_2 + s^4 n n_2  + n_1 \psi  } \indc{ \alpha n \le n_1 \le \beta' n} } =o(1),
\end{align*}
where the last equality holds 
by invoking  \prettyref{eq:n1n2_a_omega_1} in \prettyref{prop:n1n2_weak_detection}  with $\mu=0$, $\nu= \max_{  \alpha n \le n_1 \le \beta' n} \psi(n_1)$, $a=\alpha n$, $b=\beta'n$, 
and the claim that 
$\nu =o \left( \log (\alpha n)  \right) = o(\log a)$ .
Note that the conditions of \prettyref{prop:n1n2_weak_detection} for $a=\omega(1)$ are satisfied, in particular, $\nu+3 \le \log a$ and $ \mu b +\nu+2 -\log b =  \nu +2  -\log b \le \nu +2 - \log a \le 0$ for $n$ sufficiently large. 




To finish the proof, it remains to verify the claim that $\psi =o \left( \log (\alpha n ) \right)$ for $ \alpha  n \le n_1 \le \beta' n$. 
Note that $\log (\alpha n)= (1+o(1)) \log n$, as $\log (\alpha)=o(\log n)$ by the choice
of $\alpha$ and the assumption that $p=n^{o(1)}$. Thus it suffices to show that 
$\psi=o(\log n)$, \ie, 
%
%
%
%
$$
 \frac{\log (n/n_1)}{ \log \gamma} \log \frac{n_1 s^2 \log \gamma}{\log (n/n_1) }   = o \left( \log n \right) .
$$
Note that 
$$
\frac{\log (n/n_1)}{ \log \gamma}  \log \log \gamma \le \log (n/n_1) \le \log \frac{1}{\alpha} = o(\log n).
$$
Thus it suffices to show
$$
 \frac{\log (n/n_1)}{ \log \gamma} \log \frac{n_1 s^2}{\log (n/n_1) } = o \left( \log n \right)  , 
$$
which further reduces to proving that 
$\max_{x \in [1/\beta', 1/\alpha]} g(x) = o(\log n)$,
where 
$$
g(x) = \frac{\log x}{ \log \frac{  x \log x}{ nps^2} } \log \frac{n s^2}{x \log x },
$$
since $\gamma = O\left(\frac{x \log x}{nps^2}\right)$ for $x=n/n_1\in [1/\beta',1/\alpha]$.
To this end, let 
$$
\delta= \frac{1}{\log \left( \log n/ \log (1/\alpha)  \right) } =o(1),
$$ 
where the last equality holds because $\log(1/\alpha) = (1+o(1)) \log (1/p)=o(\log n)$ by the choice of $\alpha$ and the assumption that $p=n^{-o(1)}$. 
Let
$$
\phi(x) =  \log x \log \frac{n s^2}{x \log x } -  \log \frac{ x \log x}{ nps^2} \delta \log n. 
$$
It is sufficient to show that  for all $x \in [1/\beta', 1/\alpha]$, $  \phi(x) \le 0$, which immediately implies 
that $g(x) \le \delta \log n = o(\log n)$. 

Taking derivative of $\phi(x)$, we get that 
\begin{align*}
\phi'(x) & = \frac{1}{x} \log \frac{n s^2}{x \log x } -\frac{ \log x }{x} \left( 1+ \frac{1}{\log x} \right) 
- \frac{1}{x} \left( 1+ \frac{1}{\log x} \right)   \delta \log n \\
& = \frac{1}{x} \left( \log (ns^2) - 2\log x - \log\log x - 1 - \delta \log n - \frac{\delta \log n}{\log x}   \right).
\end{align*}
By assumption that $ns^2 \alpha = (2-\epsilon) \log n$, 
we get that 
$$
\log (ns^2) \le \log (2 \log n) + \log \frac{1}{\alpha} = o( \delta \log n),
$$
where the last equality uses the fact that 
$
    \frac{\log (1/\alpha)  }{\delta \log n} = \frac{\log (1/\alpha) }{\log n} \log \frac{\log n}{\log (1/\alpha) } = o(1).
$

Moreover, $\beta'=o(1)$ and hence $x=\omega(1)$.  Therefore, $\phi'(x) \le 0$ for all $x \in [1/\beta', 1/\alpha]$. 
Hence, to show $\phi(x) \le 0$, it suffices to prove $\phi(1/\beta') \le 0$. 
Note that 
\begin{align*}
\phi (1/\beta') & =  \log \frac{1}{\beta'}  \log \frac{n s^2}{ (1/\beta') \log (1/\beta') } -  \log \frac{ (1/\beta') \log (1/\beta') }{ nps^2} \delta \log n \\
& = \log \frac{1}{\beta'}  \log \frac{n s^2}{(200+o(1)) nps^2 } -  \log \frac{ (200+o(1)) nps^2  }{ nps^2} \delta \log n \\
& = \log \frac{1}{\beta'} \log \frac{1}{(200+o(1)) p } - \log \left( 200+o(1) \right) \delta \log n \le 0,
\end{align*} 
where the last inequality holds because
$$
\frac{1}{\delta}   \log \frac{1}{\beta'} 
 \le \log \frac{200 nps^2}{\log (nps^2)} \log \frac{ \log n } { \log (1/\alpha)}
 \le  O \left(  \left( \log \frac{ \log n}{ \log (1/p)} \right)^2    \right)  = o \left( \frac{ \log n}{ \log (1/p) } \right),
$$
where we used the fact that $nps^2 =O(\log n/\log (1/p))$ and $\log (1/\alpha) = (1+o(1)) \log (1/p)$ when $p=o(1)$. 

\section{Supplementary Proofs for \prettyref{sec:lb-sparse}}\label{sec:supp_lb_sparse}
\subsection{Proof of \prettyref{prop:con1}: Long orbits}\label{sec:con1}

Fix any $\sigma=\pi^{-1}\circ \tpi$.
Since $\{X_{\orbit}\}_{\orbit\notin \calO_k}$ are mutually independent, it follows that 
\begin{align*}
 \Expect_{\calQ}\left[\prod_{\orbit\notin \calO_k}X_{\orbit}\right]
 & = \prod_{\orbit\notin \calO_k}\Expect_{\calQ} \left[X_{\orbit}\right] = \prod_{\orbit\notin \calO_k} \left(1+\rho^{2|\orbit|}\right).
\end{align*}
For any $\orbit_{ij} \notin \calO_k$, we have $i$ or $j$ is from node orbit with length larger than $k$, or  $\orbit_{ij}$ has length larger than $k$. By \prettyref{sec:classification_edge_orbits}, we know that $|\orbit_{ij}|\ge  \ceil{\frac{k+1}{2}}{}$. It follows that  
\begin{align*}
 \Expect_{\calQ}\left[\prod_{\orbit\notin\calO_k}X_{\orbit}\right]  \le  \prod_{m=\ceil{\frac{k+1}{2}}{}}^{\binom{n}{2}}  \left(1+\rho^{2|\orbit|}\right)^{N_m} 
 & \overset{(a)}{\le} \left(1+\rho^k \right)^{\sum_{m\ge \ceil{\frac{k+1}{2}}{}}^{\binom{n}{2}} N_m }\\
 & \overset{(b)}{\le} \left(1+\rho^{k}\right)^{\frac{n^2}{k}}. 
\end{align*}
where $(a)$ holds, since $1+\rho^{2m}$ decreases when $m$ increases, and $1+\rho^{2m} \le 1+\rho^{k}$ for any $m \ge  \ceil{\frac{k+1}{2}}{}$; $(b)$ holds since the total number of edges is $\binom{n}{2}$, and the total number of edge orbits with length at least $\ceil{\frac{k+1}{2}}{}$ is at most $\frac{\binom{n}{2}}{\ceil{\frac{k+1}{2}}{}}\le \frac{n^2}{k}$.

\subsection{Proof of \prettyref{prop:con2}: Incomplete orbits}\label{sec:con2}

Fix any $\sigma=\pi^{-1}\circ \tpi$.
Since edge orbits are disjoint, and $A_{i,j}$ and $B_{ij}$ are i.i.d.\ $\Bern(ps)$ across all $1\le i<j \le n$ under $\calQ$, it follows that $\{A_{ij},B_{\pi(i)\pi(j)}\}_{(i,j)\in \orbit}$ are 
mutually independent across different edge orbits $\orbit$ under $\calQ$. 
Recall that an edge orbit $\orbit \in \calJ_k$ if and only if $A_{ij}=B_{\pi(i)\pi(j)} =1$ for all $(i,j) \in \orbit$.
Therefore, conditional on $\calJ_k=\calJ$, $\{A_{ij},B_{\pi(i)\pi(j)}\}_{(i,j)\in \orbit}$ 
are independent across all edge orbits $\orbit \in \calO_k\backslash \calJ$. 
In particular, the distribution of $\{A_{ij},B_{\pi(i)\pi(j)}\}_{(i,j)\in \orbit}$ for $\orbit \notin \calJ$ conditional on $\calJ_k=\calJ$ is the
same as that conditional on $\orbit \notin \calJ_k$.
Since $X_{\orbit}$ is a function of $\{A_{ij},B_{\pi(i)\pi(j)}\}_{(i,j)\in \orbit}$, it follows that 
\begin{align*}
    \Expect_{\calQ}\left[\prod_{\orbit \in \calO_k \backslash \calJ_k }X_{\orbit} \; \Big|\; \calJ_k =\calJ  \right]
     = \prod_{\orbit \in \calO_k \backslash \calJ  }  \Expect_{\calQ}\left[ X_{\orbit} \; \Big|\; \calJ_k =\calJ \right]
    =\prod_{\orbit \in \calO_k \backslash \calJ  }  \Expect_{\calQ}\left[ X_{\orbit} \; \Big|\;  \orbit \notin \calJ_k \right].
\end{align*}
Therefore, to prove \prettyref{prop:con2}, it suffices to show $\Expect_{\calQ}\left[ X_{\orbit} \; \Big|\;  \orbit \notin \calJ_k \right] \le 1$. 
Note that 
\begin{align*}
    \Expect_{\calQ}\left[X_{\orbit}| \orbit \notin \calJ_k \right]
    & = \frac{\Expect_{\calQ}\left[X_{\orbit}\indc{ \orbit \notin \calJ_k }\right]}{\Prob\left( \orbit \notin \calJ_k  \right)} =\frac{\Expect_{\calQ}\left[X_{\orbit}\right]-\Expect_{\calQ}\left[X_{\orbit}\indc{\orbit \in \calJ_k }\right]}{1-\Prob\left( \orbit \in \calJ_k \right)}.
\end{align*}
Recall that  $\orbit \in \calJ_k$ if and only if $ A_{ij}=1,B_{\pi(i)\pi(j)}=1$ for all $(i,j)\in \orbit$, in which case $X_{\orbit}=\left(\frac{1}{p}\right)^{2|\orbit|}$. 
Thus $\prob{\orbit \in \calJ_k}=\left(ps\right)^{2|\orbit|}$ and 
$$
\Expect_{\calQ}\left[X_{\orbit}\indc{\orbit \in \calJ_k }\right] = \left(\frac{1}{p}\right)^{2|\orbit|} \left(ps\right)^{2|\orbit|}= s^{2|\orbit|}.
$$
Recall that  $\Expect_{\calQ}\left[X_{\orbit}\right]=1+\rho^{2|\orbit|}$, where
$\rho = \frac{s(1-p)}{1-ps}$. Combining this with the last two displayed equation  yields that 
\begin{align*}
    &  \Expect_{\calQ}\left[X_{\orbit}| \orbit \notin \calJ_k \right] = \frac{1+\rho^{2|\orbit|}- s^{2|\orbit|} }{1-\left(ps\right)^{2|\orbit|}}
    =\frac{1-s^{2|\orbit|}
    \left(1 - \left(\frac{1-p}{1-ps}\right)^{2|\orbit|} \right)}{1-\left(ps\right)^{2|\orbit|}} \le 1,
\end{align*}
where the last inequality holds by
the following claim: if $p \le 1/2$
and $s\le 1/2$, then
\begin{align}
1-\left(\frac{1-p}{1-ps}\right)^{2|\orbit|}\ge  p^{2 |\orbit|},  \quad \forall |\orbit| \ge  1. \label{eq:X_O_bound_1}
\end{align}

Indeed, as $|\orbit|$ increases, $1-\left(\frac{1-p}{1-ps}\right)^{2|\orbit|}$ increases while $p^{2|\orbit|}$ decreases, so it suffices to
verify \prettyref{eq:X_O_bound_1} 
for $|\orbit|=1$. When $|\orbit|=1$, we have $ 1-\left(\frac{1-p}{1-ps}\right)^{2} \ge  p^{2} \iff   \left(1-ps\right)^2\ge  \frac{1-p}{1+p}$,
which holds when $p\le \frac{1}{2}$ and $s\le \frac{1}{2}$. 
\subsection{Proof of \prettyref{prop:jkbound}: Averaging over orbit lengths} \label{sec:jkbound}

In the following proof, for any $t \in \naturals$,  denote the $t$th harmonic number by $\harmonic_t \triangleq \sum_{\ell=1}^t \frac{1}{\ell}$ and $\harmonic_{0} \triangleq 0$.

Under the assumption of $s\le 0.1$, we have $2m s^{m} \le 0.04$ when $m\ge  2$ and $2ms^m \le 0.2$ for $m=1$. Thus, 
for any $1\le m\le k$, 
\begin{align}
1 +  s^{m} n_{m}  \indc{m:\mathrm{even}}  +  2 s^{2m} \sum_{\ell \le  m} \ell n_{\ell}+ s^{4m} m n_{2m}  \indc{2m \le k}  
\le 1+ 1.04 s^m \sum_{\ell \le  2m }  n_\ell \indc{\ell \le k}.  \label{eq:1.04upperbound}
\end{align}
Hence, to prove \prettyref{eq:jkbound1}, it suffices to show, for $s\le 0.1$:
\begin{align}
  \Expect \left[\prod_{m=1}^k \left(1+ 1.04 s^m \sum_{\ell \le  2m }  n_\ell \indc{\ell \le k}\right)^{n_m}\right] = O(1). \label{eq:target1}
\end{align} 
To prove \prettyref{eq:jkbound2},  it suffices to show, for $s=o(1)$:
\begin{align}
  \Expect \left[\prod_{m=1}^k \left(1+ 1.04 s^m \sum_{\ell \le  2m }  n_\ell \indc{\ell \le k}\right)^{n_m}\right] = 1+o(1). \label{eq:target2}
\end{align} 
Pick $\eta = \log k \sqrt{\log\left(n/k\right)}$. 
We can write $\Expect \left[\prod_{m=1}^k \left(1+ 1.04 s^m \sum_{\ell \le  2m }  n_\ell \indc{\ell \le k}\right)^{n_m}\right]$ as $\termI+\termII$ where
\begin{align*}
 \termI
    & =  \Expect \left[\prod_{m=1}^k \left(1+ 1.04 s^m \sum_{\ell \le  2m }  n_\ell \indc{\ell \le k}\right)^{n_m} \indc{\sum_{\ell=1}^k n_{\ell}<\eta \harmonic_k} \right],  \\
  \termII  & = \Expect \left[\prod_{m=1}^k \left(1+ 1.04 s^m \sum_{\ell \le  2m }  n_\ell \indc{\ell \le k}\right)^{n_m} \indc{\sum_{\ell=1}^k n_{\ell} \ge \eta \harmonic_k} \right]  .
\end{align*}
To bound $\termI$, we first apply \prettyref{eq:cycle_decomposition_poisson_distribution_total_v} in \prettyref{lmm:cycle_decomposition_poisson_distribution_total_v} and get: 
$$
\mathrm{TV} \left(  \calL\left( n_1, \ldots, n_k\right) , \calL\left( Z_1, \ldots, Z_k\right) \right) \le F\left( \frac{n}{k} \right),
$$
where $\calL$ denotes the law of random variables, $Z_\ell \inddistr \Pois(\frac{1}{\ell})$, and $\log F(n/k)=-(1+o(1)) \frac{n}{k}\log \left(\frac{n}{k}\right)$ when $k=o(n)$, with $o(1)$ depending only on $k/n$. 
Then, we have
\begin{align}
    \termI 
     & \le \Expect \left[\prod_{m=1}^k \left(1+ 1.04 s^m \sum_{\ell \le  2m }  Z_\ell \indc{\ell \le k}\right)^{Z_m} \right]+
     2 F\left(\frac{n}{k}\right) \left(1+ 1.04 s\eta \harmonic_k\right)^{\eta \harmonic_k}.
		\label{eq:termI}
\end{align}
Here the second term satisfies
\begin{align*}
     F\left(\frac{n}{k}\right) \left(1+1.04  s\eta \harmonic_k\right)^{\eta \harmonic_k}
    &  \overset{(a)}{\le} \exp\left(- (1+o(1)) \frac{n}{k}\log \left(\frac{n}{k}\right) + 1.04  s \eta^2 \harmonic_k^2 \right) \overset{(b)} =o(1),
\end{align*}
where $(a)$ holds because for $x \ge 0$, $\log(1+x) \le x$; $(b)$ holds since $\harmonic_k \le \log k+1$, and 
$k(\log k)^2 \eta^2 = o\left(n \log \left(\frac{n}{k}\right)\right)$ under our assumption of $k (\log k)^4 = o(n)$. 

To bound $\termII$, applying \prettyref{eq:jointdis2} in \prettyref{lmm:jointpmf}, we get that 
\begin{align}
    \termII
    & \le \Expect \left[\prod_{m=1}^k \left(1+ 1.04 s^m \sum_{\ell \le  2m }  Z_\ell \indc{\ell \le k}\right)^{Z_m} \indc{\sum_{\ell=1}^k Z_{\ell}\ge  \eta \harmonic_k}\right]e^{\harmonic_k}\nonumber\\
    & \le \Expect \left[\prod_{m=1}^k \left(1+ 1.04 s^m \sum_{\ell \le  2m }  Z_\ell \indc{\ell \le k}\right)^{Z_m} \exp\left({\frac{\sum_{\ell=1}^k Z_{\ell}}{\eta}}\right)\right].
		\label{eq:termII}
\end{align}
Since $\eta = \log k \sqrt{\log (n/k)} = \omega(\log k)$, for $s=0.1$, the desired \prettyref{eq:target1} follows from applying \prettyref{eq:s_constant} in \prettyref{lmm:poisson_approximation} to \prettyref{eq:termI} and \prettyref{eq:termII}; for $s=o(1)$, the desired \prettyref{eq:target2} follows from applying \prettyref{eq:s_little_o_1} in \prettyref{lmm:poisson_approximation} to \prettyref{eq:termI} and \prettyref{eq:termII}. Hence, our desired result follows.

\begin{lemma}\label{lmm:poisson_approximation}
Suppose $\eta  = \omega(\log k)$. If $s\le 0.1$,
        \begin{align}
            \Expect \left[\prod_{m=1}^k \left(1+ 1.04 s^m \sum_{\ell \le  2m }  Z_\ell \indc{\ell \le k}\right)^{Z_m} \exp \left( \frac{\sum_{\ell=1}^k Z_{\ell}}{\eta} \right)\right] = O(1). \label{eq:s_constant}
        \end{align}
In particular, if $s=o(1)$,
        \begin{align}
            \Expect \left[\prod_{m=1}^k \left(1+ 1.04 s^m \sum_{\ell \le  2m }  Z_\ell \indc{\ell \le k}\right)^{Z_m} \exp \left( \frac{\sum_{\ell=1}^k Z_{\ell}}{\eta} \right)\right] = 1+o(1). \label{eq:s_little_o_1}
\end{align}
\end{lemma}

To show \prettyref{lmm:poisson_approximation} we need the following elementary result (proved at the end of this subsection):
\begin{lemma}\label{lmm:fact}
Let $a, b, d, \alpha,\lambda>0$ such that $be^{\alpha+1} \lambda < \frac{1}{4}$.
Let $X \sim \Pois(\lambda)$. Then
\begin{align}
    \Expect\left[\left(a+bX\right)^{X+d}\exp\left(\alpha X\right)\right]
   & \le   a^d\left(1+2b e^{\alpha+1}\lambda \right)^d \exp \left(\lambda \left(a e^{\alpha}\left(1+2be^{\alpha+1}\lambda\right)-1\right)\right) \nonumber\\
   & + 27 \exp\left(-\lambda \right)\max\big\{  \left(4 b d\right)^d , a^d \big \} \label{eq:lmm_fact_1}.
\end{align}
Moreover, if $d=0$, and $a,b,\alpha,\lambda$ are fixed constants such that $b e^{\alpha+1}\lambda<1$,
\begin{align}
    \Expect\left[\left(a+bX\right)^{X}\exp\left(\alpha X\right)\right] =O(1)\label{eq:lmm_fact_2}.
\end{align}
In particular, if $d=0$, $a=1$, $b = o(1)$, $\alpha=o(1)$, and $\lambda$ is some fixed constant, 
\begin{align}
    \Expect\left[\left(1+bX\right)^{X}\exp\left(\alpha X\right)\right]  = 1+o(1).  \label{eq:lmm_fact_3}
\end{align}
\end{lemma}

Before proceeding to the proof of \prettyref{lmm:poisson_approximation}, we pause to note that a direct application of \prettyref{eq:lmm_fact_2} (with $X=Z_1+\ldots+Z_k$, $a=1$, $b=1.04s$, $\alpha=1/\eta$) yields the condition $s \log k = O(1)$. Since $k$ will be chosen to be $\Theta(\log n)$ eventually, this translates into the statement that strong detection is impossible if $s=O(\frac{1}{\log \log n})$. In order to improve this to $s=O(1)$, the key idea is to partition the product over  $[k]$ in \prettyref{eq:s_constant} into subsets and recursively peel off the expectation backwards by repeated applications of \prettyref{lmm:fact}.
\begin{proof}[Proof of \prettyref{lmm:poisson_approximation}]
Let $m_0 =0$, $m_1 = 1$, and iteratively define 
\begin{align}
     m_{\ell} = \floor{\exp\left(3 \times 2^{-2\ell+1}s^{-\frac{m_{\ell-1}}{2}}\right)}{}, \quad  2 \le \ell \le r, \label{eq:m_ell}
\end{align}
where $r$ is the integer such that $m_{r-1}<k \le m_r$. Note that for $s\le 0.1$, we have 
$m_2\ge  3 = 3 m_1$; moreover, if $s=o(1)$, we have $m_2 =\omega(1)$. 
Note that it is possible that $m_{r-1}$ is very close to $k$ and  $m_r$ is much larger than $k$, especially when $s=o(1)$. To simplify the argument,
define $K= \min\{k^2, m_r\}$. Since the quantity inside the expectation in \prettyref{eq:s_constant} increases in $k$, it suffices to bound 
$$
 \Expect \left[\prod_{m=1}^K \left(1+ 1.04 s^m \sum_{\ell \le  2m }  Z_\ell \indc{\ell \le K}\right)^{Z_m} \exp \left( \frac{\sum_{\ell=1}^K Z_{\ell}}{\eta} \right)\right]
$$
and we can assume, without loss of generality, that $k=\omega(1)$. 




Next we partition $[K]$ into the following $r$ subsets: 
$$\{m_{0}+1,\cdots, m_{1}\},\{m_{1}+1,\cdots, m_2\},\cdots, \{m_{r-1}+1,\cdots, K\}.$$
For $1\le \ell\le r$, define 
$$
A_{\ell} \triangleq \sum_{t=m_{\ell-1}+1}^{m_{\ell}} Z_t\indc{t \le K} , \quad  B_{\ell} \triangleq \sum_{t=m_{\ell-1}+1}^{2m_{\ell-1}}Z_{t}\indc{t \le K} , \quad C_{\ell} \triangleq \sum_{t=2m_{\ell-1}+1}^{m_{\ell}} Z_{t}\indc{t \le K}, \quad M_{\ell} \triangleq \sum_{t=1}^{\ell} A_{t},
$$
and $B_{r+1}=0$.
Note that $A_{\ell} = B_{\ell}+C_{\ell}$ and, by the definition of $r$, $M_r = \sum_{t=1}^K Z_t$. Furthermore, for $1 \le \ell \le r$,
\begin{align*}
     C_{\ell}+B_{\ell+1} & \inddistr \Pois\left(\lambda_{\ell}\right),\quad \mathrm{where\ }\lambda_{\ell} \triangleq \harmonic_{2m_{\ell}\wedge K} - \harmonic_{2m_{\ell-1}\wedge K} \,.
\end{align*}
Then we can upper bound \prettyref{eq:s_constant} as 
\begin{align}
    \prettyref{eq:s_constant}
    & \le \Expect \left[\prod_{\ell=1}^{r}\left(\prod_{m = m_{\ell-1}+1}^{m_{\ell}} \left(1+1.04s^{m_{\ell-1}+1} \left(M_{\ell}+B_{\ell+1}\right)\right)^{Z_m} \right)\exp\left(\frac{M_r}{\eta}\right)\right] \nonumber \\ 
    & = \Expect \left[\prod_{\ell=1}^{r} \left(1+1.04 s^{m_{\ell-1}+1} \left(M_{\ell}+B_{\ell+1}\right)\right)^{A_{\ell}}\exp\left(\frac{M_r}{\eta}\right) \right]. \label{eq:abm}
\end{align}
Define the sequence $\{\alpha_\ell\}_{\ell=1}^r$ and $\{\beta_\ell\}_{\ell=1}^r$ backward recursively by 
$ \alpha_r =\frac{1}{\eta}$, $\beta_r = 0$ and for $2\le \ell \le r$, 
\begin{align}
            \alpha_{\ell-1} 
            & \triangleq \alpha_{\ell} + 1.04 \lambda_{\ell} e^{\alpha_{\ell}}s^{m_{\ell-1}+1} \left(1+2.08 s^{m_{\ell-1}+1} e^{\alpha_{\ell}+1} \lambda_{\ell} \right) + \log \left(1+2.08 s^{m_{\ell-1}+1} e^{\alpha_{\ell}+1}\lambda_{\ell}\right)\label{eq:alpha_ell-1},\\
            \beta_{\ell-1} & \triangleq \beta_{\ell} + \lambda_{\ell} \left(e^{\alpha_{\ell}}\left(1+2.08s^{m_{\ell-1}+1} e^{\alpha_{\ell}+1} \lambda_{\ell} \right)-1\right). \label{eq:beta_ell-1}
\end{align}
For $1\le \ell\le r$, define
\begin{align}
    S_{\ell} 
    &\triangleq \Expect \biggl [\left(\prod_{t=1}^{\ell-1} \left(1+1.04 s^{m_{t-1}+1}\left(M_{t}+B_{t+1}\right)\right)^{A_{t}}\right) \nonumber \\
    &~~~~ \left(1+1.04s^{m_{\ell-1}+1}\left(M_{\ell}+B_{\ell+1}\right)\right)^{A_{\ell}+B_{\ell+1}}  e^{\alpha_{\ell}\left( M_{\ell}+B_{\ell+1}\right)+\beta_{\ell}}\biggr ] \label{eq:sell_1}, 
\end{align}
Since $B_{r+1}=0$, it follows that $S_r$ is precisely the RHS of \prettyref{eq:abm} and our goal is to show $S_r=O(1)$. 
This is accomplished by the following sequence of claims:
\begin{enumerate}[label=(C\arabic*)]
    \item \label{C:1} For any $2\le \ell \le r$, $ \left(\log m_{\ell}\right) s^{m_{\ell-1}+1}  \le \left(\log m_{\ell}\right)^2 s^{m_{\ell-1}+1} \le 9 s \times 2^{-4\ell+2}$; 
    \item \label{C:2}  For any $2\le \ell \le r$, $m_{\ell} \ge m_{\ell-1}^2$ and $m_{\ell} \ge \left(m_2\right)^{2^{\ell-2}}$; 
    \item \label{C:3} For any $2 \le \ell \le r$, $ \left( \log m_{\ell+1} \right) s^{m_{\ell}} \le \frac{1}{8} \left( \log m_{\ell}  \right) s^{m_{\ell-1}}$;
    \item \label{C:lambda} For any $2\le \ell \le r$, $\lambda_{\ell}\le \log m_{\ell}$, and $\sum_{\ell=2}^r \exp\left(-\lambda_{\ell}\right)=O(1)$, in particular, if  $s=o(1)$, $ \sum_{\ell=2}^r \exp\left(-\lambda_{\ell}\right)=o(1)$; 
    \item \label{C:alpha_beta} For $1\le\ell \le r$, $\alpha_{\ell} \le \frac{2}{5}$, and $\beta_1 \le 3$, in particular, if  $s=o(1)$, for $1\le\ell \le r$, $\alpha_{\ell} = o(1)$, and $\beta_1 = o(1)$; 
     \item \label{C:sell_sell-1} For any $2 \le \ell\le r$,
    $ S_{\ell} \le S_{\ell-1} \left(1+27\exp\left(-\lambda_{\ell} \right)\right)$, and $S_r =O(1)$, in particular, if $s=o(1)$, $S_r =1+o(1)$.
\end{enumerate}
We finish the proof by verifying these claims:
\begin{itemize}
    \item Proof of \ref{C:1}: This follows from the definition of $m_\ell$ given in \prettyref{eq:m_ell}.
    \item Proof of \ref{C:2}: 
    It suffices to prove the first inequality. We proceed by induction. 
    For the base case of  $\ell=2$, recall that we have shown under the assumption $s \le 0.1$, we have $m_2 \ge  3 m_1 \ge m_1^2$. By \prettyref{eq:m_ell}, we can get that 
    $$
    \frac{m_3}{m_2^2}\ge  \frac{1}{m_2^2} \left(\exp\left(3\times 2^{-5}0.1^{-\frac{m_2}{2}} \right)-1 \right) \ge  2,
    $$ 
    where  the last inequality holds because $m_2 \ge 3$ and $x \mapsto \frac{1}{x^2}\left(\exp\left(3 \times 2^{-5}0.1^{-\frac{x}{2}}\right)-1\right)$ increases for $x\ge 3$. Then we have $m_3 \ge 2 m_2^2 \ge 6 m_2$ given $m_2 \ge 3$. 
    
    Fix any $4 \le \ell \le r$,
    suppose we have shown for every $2 \le t \le \ell-1$, $m_{t} \ge m_{t-1}^2$. By \prettyref{eq:m_ell}, $ m_{\ell-1} \le  \exp\left(3\times 2^{-2\ell+3}s^{-\frac{m_{\ell-2}}{2}}\right)$ and
    $$
    m_{\ell} \ge \exp \left(3\times 2^{-2\ell+1}s^{-\frac{m_{\ell-1}}{2}}\right)-1 \ge   \exp\left(3\times 2^{-2\ell+1}s^{-\frac{6m_{\ell-2}}{2}}\right)-1,
    $$
    where the last inequality holds because $m_{\ell-1} \ge 6 m_{\ell-2}$ for $\ell \ge 4$
    following from $m_3\ge 6 m_2$ and the induction hypothesis $m_{\ell-1} \ge m_{\ell-2}^2$. Then we have 
    \begin{align*}
        \frac{m_{\ell}}{m_{\ell-1}}
        \ge  \exp\left(3\times 2^{-2\ell+3}s^{-\frac{m_{\ell-2}}{2}} \left(2^{-2}s^{-\frac{5}{2}m_{\ell-2}}-1\right)\right) -1 & \overset{(a)}{\ge } \exp\left(3\times 2^{-2\ell+3}s^{-\frac{m_{\ell-2}}{2}}\times 2\right)-1\\
        &\overset{(b)}{\ge } m_{\ell-1}^2 -1 \overset{(c)}{\ge } m_{\ell-1},
    \end{align*}
    where $(a)$ holds by $2^{-2}s^{-\frac{5}{2} m_{\ell-2}} \ge  3$ given $s\le 0.1$ and $m_{\ell-2} \ge  3$; $(b)$ holds by $m_{\ell-1} \le  \exp\left(3\times 2^{-2\ell+3}s^{-\frac{m_{\ell-2}}{2}}\right)$; $(c)$ holds by $m_{\ell-1} \ge 6m_{\ell-2} \ge 18$ for $\ell\ge  4$ given $m_2\ge 3$. Hence, 
    \ref{C:2} follows.

     \item Proof of \ref{C:3}:
    We prove a stronger statement: $\left(\log m_{\ell+1 }\right)^2 s^{m_{\ell}} \le \frac{1}{8}\left(\log m_{\ell}\right)^2 s^{m_{\ell-1}}$ for $2 \le \ell \le r$. 
    Since  $\left(\log m_{\ell+1}\right)^2 s^{m_{\ell}}  \le 9 \times 2^{-4\ell-2}$ by \ref{C:1}, it suffices to show
    $\left(\log m_{\ell}\right)^2 s^{m_{\ell-1}}  \ge  9 \times  2^{-4\ell+1}$  for $\ell\ge 2$.  
    For $\ell=2$, we have $\left(\log m_{2}\right)^2 s^{m_{1}}  \ge  9 \times  2^{-7}$, since $m_2 \ge 3$ and $s\le 0.1$.
    By \prettyref{eq:m_ell}, we have 
    $$ 
    m_{\ell} \ge  \exp \left(3 \times 2^{-2\ell+1}s^{-\frac{m_{\ell-1}}{2}}\right)-1 \ge  \exp\left(3 \times 2^{-2\ell+\frac{1}{2}}s^{-\frac{m_{\ell-1}}{2}}\right),
    $$ 
    where the last inequality holds because  $\exp(x)-1\ge  \exp(x/\sqrt{2})$ for $x\ge  1.22$, and   $3 \times 2^{-2\ell+1}s^{-\frac{m_{\ell-1}}{2}} \ge \ln m_{\ell}
    \ge 2^{\ell-2} \ln m_2 \ge 2$ for $\ell\ge  3$, by \ref{C:1}, \ref{C:2}, and $m_2 \ge 3.$
    

    \item Proof of \ref{C:lambda}:
    Note that $2m_{r-1} \le m_r$ by \ref{C:2}. Moreover, $2 m_{r-1} < k^2$ as $m_{r-1}<k$ and $k=\omega(1)$. Since $K=\min\{k^2, m_r\}$,  it follows that $2m_{r-1} \le K \le m_r$.  Therefore,
    $\lambda_{r} = \harmonic_{K} - \harmonic_{2m_{r-1}} $ and 
    $ \lambda_{\ell}= \harmonic_{2m_\ell} - \harmonic_{2m_{\ell-1}}$ for $1 \le \ell \le r-1.$
    
    Since for any $n,k \in \naturals$, $ \harmonic_n - \harmonic_k \le \int_k^n \frac{1}{x}  \mathrm{d}x  \le \log\frac{n}{k},$ 
    then for any $2 \le \ell \le r$,
    $$
    \lambda_{\ell} \le \harmonic_{2m_\ell} - \harmonic_{2m_{\ell-1}} \le \log \frac{m_\ell}{m_{\ell-1}} \le \log m_{\ell},
    $$
    where the last inequality holds due to $m_{\ell-1} \ge 1$ for $\ell \ge 2.$

   Conversely, since for any $n,k \in \naturals$, $ \harmonic_n - \harmonic_k \ge \int_{k+1}^{n+1}\frac{1}{x}  \mathrm{d}x  \geq \log \frac{n+1}{k+1} \geq \log \frac{n}{k}-1$, it follows that for any $2 \le \ell \le r-1$,
  $$ 
  \lambda_{\ell} = \harmonic_{2m_\ell} - \harmonic_{2m_{\ell-1}} \ge  \log\left(\frac{m_{\ell}}{m_{\ell-1}}\right)-1,
  $$
 and  $\lambda_r \ge \log \frac{K}{2m_{r-1}} -1$. 
    If $K=k^2$, then $K \ge k m_{r-1}$
    and thus $\lambda_r \ge \log \frac{k}{2}-1$; otherwise, $K=m_r$
    and thus
    $\lambda_r  \ge \log \frac{m_r}{2 m_{r-1}} - 1$. Hence, we get that
    \begin{align*}
    \sum_{\ell=2}^r \exp\left(-\lambda_{\ell}\right) 
        & \le   \sum_{\ell=2}^{r}\frac{2 e m_{\ell-1}}{m_{\ell}}  + \frac{2 e }{k} \overset{(a)}{\le} \frac{2e}{m_2 } +  \sum_{\ell=3}^{r} \frac{2e}{m_{\ell-1}} + o(1) \overset{(b)}{\le} \frac{3 e}{m_{2}}+ o(1) = O(1),
    \end{align*}

    where $(a)$ holds by $m_1=1$, $k=\omega(1)$, and \ref{C:2} so that $m_{\ell} \ge m_{\ell-1}^2$ for $\ell\ge 2$; $(b)$ holds because in view of \ref{C:2}, $m_{\ell} \ge \left(m_2\right)^{2^{\ell-2}}$ and
    $m_2 \ge 3$, so that
    $$
    \sum_{\ell=3}^{r-1} \frac{1}{m_{\ell-1}} \le \sum_{\ell=3}^{r-1} \frac{1}{ (m_2)^{2^{\ell-2}}} 
    \le \sum_{\ell=2}^{\infty} \frac{1}{(m_2)^\ell}
    = \frac{1}{m_2^2} \frac{1}{1-m_2^{-1}}
    \le
    \frac{1}{2m_2}.
    $$
    In particular, if $s=o(1)$, then $m_{2}=\omega(1)$ and we have 
    $$
    \sum_{\ell=2}^r \exp\left(-\lambda_{\ell}\right) \le \frac{3 e}{m_{2}}+ o(1) = o(1).
    $$
    
    Hence, \ref{C:lambda} follows.
		
        \item Proof of \ref{C:alpha_beta}:
        For $2 \le \ell \le r$, let $c=1.04$, we define 
        $$
        \psi_{\ell} \triangleq  (2+4e)c \left( \log m_{\ell}  \right) s^{m_{\ell-1}+1}+ 8e c^2\left(\log m_{\ell}\right)^2 s^{2m_{\ell-1}+2}.
        $$
        We prove $\alpha_{\ell} \le \frac{2}{5}$ for $1\le \ell\le r$ by induction. 
        Since $\eta=\omega(1) \ge  6$, we have  $\alpha_r=\frac{1}{\eta}< \frac{2}{5}$. 
        Fix any $2 \le \ell \le r$.
        Suppose we have shown for any $\ell \le t\le r$, $\alpha_{t}\le \frac{2}{5}$. 
        Then $e^{\alpha_t}\le 1+2\alpha_t$, $e^{2\alpha_t}\le 1+4\alpha_t$. Since $\log(1+x) \le x$ for $x\ge  0$, by \prettyref{eq:alpha_ell-1} we have 
        \begin{align*}
            \alpha_{t-1} 
            & \le \alpha_{t} + \left(1 + 2e \right)c \lambda_{t} e^{\alpha_{t}} s^{m_{t-1}+1} + 2e c^2 s^{2m_{t-1}+2} e^{2\alpha_{t}} \lambda_{t}^2 \\
            & \le \alpha_{t} + \left(1 + 2e \right)c \lambda_{t} \left(1+2\alpha_{t}\right) s^{m_{t-1}+1}+ 2e c^2  s^{2m_{t-1}+2}  \left(1+4\alpha_{t}\right) \lambda_{t}^2 \\
            & \le \alpha_{t} \left(1+\psi_t\right) + \frac{1}{2} \psi_t,
        \end{align*}

        where the last inequality holds by \ref{C:lambda}.
        By the induction hypothesis, 
        \begin{align}
            \alpha_{\ell-1} 
             & \le  \alpha_r \prod_{t=\ell}^r \left(1+\psi_t \right) + \frac{1}{2} \sum_{t=\ell}^r \psi_t \prod_{j=\ell}^{t-1} \left( 1 + \psi_j \right) \nonumber  \\
             & \le  \left(  \alpha_r +  \frac{1}{2} \sum_{t=\ell}^r \psi_t \right) \exp \left( \sum_{t=\ell}^r \psi_t \right) \overset{(a)}{<} \frac{9}{7}\alpha_r+\frac{36}{49} \psi_{\ell} \overset{(b)}{<} \frac{2}{5}, \label{eq:alpha_ell_upper_bound}
        \end{align}

         where $(a)$ holds by $ \exp\left(\sum_{t=\ell}^r \psi_t \right) \le \frac{9}{7}$, since $\sum_{t=\ell}^r \psi_t \le \frac{8}{7} \psi_{\ell}$ following from $\psi_{t+1} \le \frac{1}{8} \psi_{t}$ by \ref{C:3}, and $\psi_{\ell} \le 0.2$ for $\ell\ge  2$ because $\left(\log m_{\ell} \right) s^{m_{\ell-1}+1} \le 9s \times 2^{-6} \le 2^{-6}$ by \ref{C:1}; $(b)$ holds because $\alpha_r= \frac{1}{\eta} \le \frac{1}{6}$ by $\eta = \omega(1)$ and $\psi_{\ell} \le 0.2$.
         
         In particular, if $s=o(1)$, by \ref{C:3}, we have $\sum_{t=\ell}^r \psi_{t} \le \frac{8}{7} \psi_\ell = o(1)$
         because $\psi_{\ell} = o(1)$ for $\ell\geq 2$ following from $\left(\log m_{\ell} \right) s^{m_{\ell-1}+1} \le 9 s \times 2^{-4\ell+2}  = o(1)$ by \ref{C:1} and $s=o(1)$. Then for $2 \le \ell \le r $, we have 
         $$
         \alpha_{\ell-1} \le  \left(  \alpha_r +  \frac{1}{2} \sum_{t=\ell}^r \psi_t \right) \exp \left( \sum_{t=\ell}^r \psi_t \right) = \left(\alpha_r +o(1) \right) \left(1+ o(1)\right) = o(1),
         $$ 
         where the last equality holds because $\alpha_r = \frac{1}{\eta} = o(1)$ given $\eta = \omega(\log k)$. 
         
        Since we have shown that $\alpha_{\ell}\le \frac{2}{5}$ for $1\le \ell\le r$, it follows that 
        $e^{\alpha_\ell}\le 1+2\alpha_\ell$ and $e^{2\alpha_\ell}\le 1+4\alpha_\ell$.
        Then by \prettyref{eq:beta_ell-1}, we can get that for $2 \le \ell\le r$,  
        \begin{align*}
            \beta_{\ell-1}
            & \le \beta_{\ell} + \lambda_{\ell} \left\{  2\alpha_\ell + 2ec s^{m_{\ell-1}+1} (1+4\alpha_\ell) \lambda_{\ell} \right\} \\
            & \overset{(a)}{\le }\beta_{\ell} +  2\alpha_{\ell}\lambda_{\ell} + \left(1 + \frac{8}{5}\right) 2ec\lambda_{\ell}^2 s^{m_{\ell-1}+1} \\
            &\overset{(b)}{\le }\beta_{\ell} +  2\left(\frac{9}{7}\alpha_r+\frac{36}{49} \psi_{\ell+1}\right)\lambda_{\ell} + 
            \frac{26}{5} e c\lambda_{\ell}^2 s^{m_{\ell-1}+1} \\
            & \overset{(c)}{<} \beta_{\ell} +\frac{18}{7} \alpha_r \lambda_{\ell} + \frac{9}{49} \psi_\ell \log (m_\ell) 
            + \frac{26}{5} e c \left( \log m_{\ell} \right)^2  s^{m_{\ell-1}+1},
        \end{align*}
        where $(a)$ holds by $\alpha_{\ell} \le \frac{2}{5}$; 
        $(b)$ holds by \prettyref{eq:alpha_ell_upper_bound} so that $\alpha_{\ell}\le \frac{9}{7}\alpha_r+\frac{36}{49}\psi_{\ell+1}$; 
        $(c)$ holds because $\lambda_\ell \le \log m_\ell$ for $\ell \ge 2$ by  \ref{C:lambda} and
        $\psi_{\ell+1}\le \frac{1}{8}\psi_{\ell}$ by \ref{C:3}. Since $ \left( \log m_{\ell} \right) s^{m_{\ell-1}+1} \le 9s \times 2^{-4 \ell+2} \le 2^{-6}$ by \ref{C:1}, it follows that 
        $$
           \psi_{\ell} \le (2+4e) c \left( \log m_{\ell} \right) s^{m_{\ell-1}+1} + \frac{8ec^2}{2^6} \left( \log m_{\ell} \right) s^{m_{\ell-1}+1} 
           \le 14\left(  \log m_{\ell} \right) s^{m_{\ell-1}+1}.
        $$  
       Combining the last two displayed equation yields that 
         \begin{align*}
          \beta_{\ell-1}  \le  \beta_{\ell} + \frac{18}{7} \alpha_r \lambda_{\ell} + 18 \left( \log m_{\ell} \right)^2  s^{m_{\ell-1}+1}
           \le \beta_{\ell} + \frac{18}{7} \alpha_r \lambda_{\ell}  + 18  \times 9s \times 2^{-4 \ell+2},
         \end{align*}
     where the last inequality holds 
     in view of $\left(\log m_{\ell}\right)^2 s^{m_{\ell-1}+1} \le   9 s\times 2^{-4\ell+2}$ by \ref{C:1}.

    By the telescoping summation of the last displayed equation 
        and $\beta_r=0$, we get 
        \begin{align*}
            \beta_1
            & \le \frac{18}{7}\alpha_r  \sum_{\ell=2}^{r} \lambda_{\ell}+ 18  \times 9 s\times\sum_{\ell=2}^r 2^{-4\ell+2}\le 3,
        \end{align*}
        where the last equality holds by since $  \sum_{\ell=2}^{r} \lambda_{\ell} \le \log K \le 2\log k$ and $\alpha_r = \frac{1}{\eta} = \frac{1}{\omega\left(\log k\right)}$. 
        In particular if $s=o(1)$, we have 
        $
             \beta_1 = o(1). 
        $
    
		\item Proof of \ref{C:sell_sell-1})
        First, we prove for any $2 \le \ell\le r$, $ S_{\ell} \le S_{\ell-1} \left(1+27\exp\left(-\lambda_{\ell} \right)\right)$. 
        Since $M_{{\ell}} = M_{{\ell}-1}+B_{{\ell}}+C_{{\ell}}$ and $A_{\ell} = B_{{\ell}} + C_{{\ell}}$, by \prettyref{eq:sell_1}, letting $c=1.04$, we have  
        \begin{align}
            S_{\ell} 
            &\le \Expect \bigg[\left(\prod_{t=1}^{\ell-1} \left(1+c s^{m_{t-1}+1}\left(M_{t}+B_{t+1}\right)\right)^{A_{t}}\right) \nonumber\\ 
            &~~~~\left(1+c s^{m_{\ell-1}+1}\left(M_{\ell-1}+B_{\ell}+C_{\ell}+B_{\ell+1}\right)\right)^{B_{\ell}+C_{\ell}+B_{\ell+1}}\exp\left(\alpha_{\ell}\left( M_{\ell-1}+B_{\ell}+C_{\ell}+B_{\ell+1}\right)+\beta_{\ell}\right)\bigg ]\nonumber \\
            & = \Expect \bigg[\left(\prod_{t=1}^{\ell-1} \left(1+c s^{m_{t-1}+1}\left(M_{t}+B_{t+1}\right)\right)^{A_{t}}\right) \exp\left(\alpha_{\ell}\left( M_{\ell-1}+B_{\ell}\right)+\beta_{\ell}\right)\nonumber \\
            &~~~~ \left(1+c s^{m_{\ell-1}+1}\left(M_{\ell-1}+B_{\ell} + C_{\ell}+B_{\ell+1}\right) \right)^{B_{\ell}+C_{\ell}+B_{\ell+1}} \exp\left(\alpha_{\ell}\left(C_{\ell}+B_{\ell+1}\right)\right)\bigg] \label{eq:sell_2}. 
        \end{align} 
 Note that   $\{B_t, A_t \}_{t=1}^{\ell-1}$ and $B_\ell$ are independent from  $ C_{\ell}+B_{\ell+1}$.
 To proceed, we condition on $\{B_t, A_t \}_{t=1}^{\ell-1}$ and $B_\ell$, and take expectation over
 $ C_{\ell}+B_{\ell+1}$ by applying \prettyref{eq:lmm_fact_1} in \prettyref{lmm:fact}. 
 In particular, let $X = C_{\ell}+B_{\ell+1} \sim \Pois\left(\lambda_{\ell}\right)$, $a=1+cs^{m_{\ell-1}+1}\left(M_{\ell-1}+B_{\ell}\right)$, $b=c s^{m_{\ell-1}+1}$, $d=B_{\ell}$, $\lambda=\lambda_{\ell}$, and $\alpha = \alpha_{\ell} \le \frac{2}{5}$ by \ref{C:alpha_beta}. 
 By \ref{C:lambda}, $\lambda_{\ell} \le \log m_{\ell}$ for $\ell \ge  2$, and thus
 $$
 be^{\alpha+1} \lambda = c s^{m_{\ell-1}+1} e^{\alpha_{\ell}+1} \lambda_{\ell} \le c e^{1.4}  s^{m_{\ell-1}+1} \log m_{\ell} 
 \le c e^{1.4} \times 2^{-6} <\frac{1}{4},
 $$
 where the second-to-the-last inequality holds by \ref{C:1}.
For ease of notation, let $\xi=1+2be^{\alpha+1} \lambda$ and note that $a=1+b\left(M_{\ell-1}+B_{\ell}\right)$.
Then  it follows from \prettyref{eq:lmm_fact_1} in \prettyref{lmm:fact} that 
    \begin{align*}
          & \Expect_{C_{\ell}+B_{\ell+1}} 
          \left[ \left(1+cs^{m_{\ell-1}+1}\left(M_{\ell-1}+B_{\ell}  + C_{\ell}+B_{\ell+1}\right) \right)^{B_{\ell}+C_{\ell}+B_{\ell+1}} \exp\left(\alpha_{\ell}\left(C_{\ell}+B_{\ell+1}\right)\right) \mid M_{\ell-1}, B_{\ell} \right]\\
          & \le
           \left( 1+b  \left(M_{\ell-1}+B_{\ell}\right)\right)^{B_{\ell}} \xi^{B_{\ell}}  \exp\left\{ \lambda_{\ell}\left[ \left( 1+ b \left(M_{\ell-1}+B_{\ell}\right)\right) e^{\alpha_{\ell}}\xi - 1\right]\right\} \\
          &~~~~+ 27\exp\left(-\lambda_{\ell} \right)\max\big\{  \left(4 b B_{\ell}\right)^{B_{\ell}} , \left(1+b \left(M_{\ell-1}+B_{\ell}\right)\right)^{B_{\ell}}\big \} \\
          & \le \left( 1+b \left(M_{\ell-1}+B_{\ell}\right)\right)^{B_{\ell}} \xi^{M_{\ell-1}+B_{\ell}}
           \exp\left \{ \lambda_{\ell}\left[ \left( 1+ b \left(M_{\ell-1}+B_{\ell}\right)\right) e^{\alpha_{\ell}} \xi -1\right] \right\} \\
          &~~~~ + 27 \exp\left(-\lambda_{\ell} \right) \left(1+ 4 b \left(M_{\ell-1}+B_{\ell}\right)\right)^{B_{\ell}}\\
          & \le \left( 1+ 4 b \left(M_{\ell-1}+B_{\ell}\right)\right)^{B_{\ell}}\exp\left \{
         \left( \lambda_{\ell} b e^{\alpha_{\ell}} \xi+ \log \xi \right) \left(M_{\ell-1}+B_{\ell}\right)  + \lambda_{\ell}\left( e^{\alpha_{\ell}} \xi-1\right) \right \} 
          \left(1+27\exp\left(-\lambda_{\ell}\right) \right)\\
          & \le \left( 1+ c s^{m_{\ell-2}+1}\left(M_{\ell-1}+B_{\ell}\right)\right)^{B_{\ell}}\exp\left\{\left(\alpha_{\ell-1}-\alpha_{\ell}\right)\left(M_{\ell-1}+B_{\ell}\right)+\left(\beta_{\ell-1}-\beta_{\ell}\right)\right\} \left(1+27 \exp\left(-\lambda_{\ell} \right) \right),
        \end{align*}
    where the last inequality holds in view of \prettyref{eq:alpha_ell-1} and \prettyref{eq:beta_ell-1}, and
    the fact that $4 s^{m_{\ell-1}} \le s^{m_{\ell-2}}$ by \ref{C:2} and $s\le 0.1$.
        
    Combining the above result with \prettyref{eq:sell_2}, we get 
        \begin{align*}
            S_{\ell}
            & \le \Expect \bigg[ \left(\prod_{t=1}^{\ell-1} \left(1+ c s^{m_{t-1}+1}\left(M_{t}+B_{t+1}\right)\right)^{A_{t}}\right)\nonumber \\
            &~~~~\left( 1+ c s ^{m_{\ell-2}+1}\left(M_{\ell-1}+B_{\ell}\right)\right)^{B_{\ell}} \exp\left(\alpha_{\ell-1}\left( M_{\ell-1}+B_{\ell}\right)+\beta_{\ell-1}\right) \bigg] \left(1+27\exp\left(-\lambda_{\ell} \right)\right)\\
            & =  S_{\ell-1}\left(1+27\exp\left(-\lambda_{\ell} \right)\right).
        \end{align*}
        Applying the last displayed equation recursively, we get that
        \begin{align*}
            \prettyref{eq:s_constant}=S_r 
            & \le S_1 \prod_{\ell=2}^r \left(1+27\exp\left(-\lambda_{\ell}\right)\right) \le S_1 \exp\left(27\sum_{\ell=2}^r \exp\left(-\lambda_{\ell}\right) \right)  = O(S_1),
        \end{align*}
        where the last equality holds by \ref{C:lambda}, in particular if $s=o(1)$, we have 
        $$
        S_r = S_1 \left(1+ o(1)\right).
        $$
         
        It remains to calculate $S_1$. Note that
\begin{align}
    S_1
    & = \expect{(1+c s (M_1+B_2))^{A_1+B_2}\exp\left(\alpha_1(M_1+B_2)+\beta_1\right)} \nonumber \\
    & = \expect{(1+c s(C_1+B_2))^{C_1+B_2}\exp\left(\alpha_1(C_1+B_2)\right)} \exp\left(\beta_1\right) \nonumber,
\end{align}
where the last equality holds due to $B_1=0$ and $M_1=A_1=C_1$.
We apply \prettyref{eq:lmm_fact_2} in \prettyref{lmm:fact}, with $X= C_1+B_2 \sim \Pois(\lambda_1)$, $a=1$, $b=1.04 s \le 0.104 $, $\lambda = \lambda_1=\frac{3}{2}$ and $\alpha =\alpha_1 \le \frac{2}{5}$ by \ref{C:alpha_beta}. Noting that $be^{\alpha+1}\lambda \le 1.04 \cdot0.1 \cdot e^{1.4} \cdot \frac{3}{2} < 0.633$ and $\beta_1 \le 3$ by \ref{C:alpha_beta}, we conclude $
S_1 =O(1)$ and hence $S_r= O(1)$.

In particular, if $s=o(1)$, by \prettyref{eq:lmm_fact_3} in \prettyref{lmm:fact} with $a = 1$, $b=1.04 s=o(1)$, $d=0$, $\alpha = \alpha_1= o(1)$ by \ref{C:alpha_beta}, and $\lambda= \frac{3}{2}$, we have 
$S_1 = 1+o(1)$ and hence $S_r = 1+o(1)$.

        \end{itemize}

\end{proof}

\begin{proof}[Proof of \prettyref{lmm:fact}]
First, we show \prettyref{eq:lmm_fact_1}: Let $\gamma =  2 e^{\alpha+1}a$. Write
\begin{align*}
\Expect \left[\left(a+bX\right)^{X+d}\exp\left(\alpha X\right)\right] 
    & = \Expect \left[\left(a+bX\right)^{X+d}\exp\left(\alpha X\right)\indc{X < \gamma\lambda}\right]+\Expect \left[\left(a+bX\right)^{X+d}\exp\left(\alpha X\right)\indc{X \ge  \gamma\lambda}\right] \\
    & =  \termI +\termII.
\end{align*}
Then we have 
\begin{align*}
    \termI 
    & \le \Expect \left[\left(a+b\gamma\lambda\right)^{X+d}\exp\left(\alpha X\right)\right]\\
    & = \left(a+b\gamma\lambda\right)^d \Expect\left[\exp\left( X \left(\log \left(a+b\gamma\lambda\right)+\alpha \right)\right)\right]\\
    & \overset{(a)}{=}  \left(a+b\gamma\lambda\right)^d \exp \left(\lambda \left(e^{\alpha}\left(a+b\gamma\lambda\right)-1\right)\right)\\
    & = \left(a+2a b e^{\alpha+1}\lambda \right)^d \exp \left(\lambda \left(a e^{\alpha}\left(1+2be^{\alpha+1}\lambda\right)-1\right)\right),
\end{align*}
where $(a)$ holds by the moment generating function $\Expect\left[\exp\left(t X\right)\right]=\exp \left(\lambda \left(e^t-1\right)\right)$. Then, directly substituting the Poisson PMF into $\termII$, we get 
\begin{align*}
    \termII
    & \le e^{-\lambda}\sum_{k \ge  \gamma\lambda}\frac{\lambda^k}{k!}\left(a+bk\right)^{k+d}\exp\left(\alpha k\right)\\
    & \overset{(a)}{\le} e^{-\lambda}\sum_{k \ge  \gamma\lambda}\left(\frac{e\lambda}{k}\right)^{k} \left(a+bk\right)^{k+d}\exp\left(\alpha k\right)\\
    & \le e^{-\lambda}\sum_{k \ge  \gamma\lambda}\left[ \frac{ae^{\alpha+1}}{\gamma }+b e^{\alpha+1}\lambda\right]^k\left(a+bk\right)^{d} \\
    & \overset{(b)}{\le} e^{-\lambda}\sum_{k \ge  \gamma\lambda}\left(\frac{3 e^{\frac{1}{4}}}{4} \right)^k\exp\left(f\left(k\right)\right),
\end{align*}
where $(a)$ holds because $k! \ge  \left(\frac{k}{e}\right)^k$; $(b)$ holds by the choice of 
 $\gamma =  2 a e^{\alpha+1}$, our assumption that $b e^{\alpha+1}\lambda <\frac{1}{4}$, and defining:
$$
f(x) \triangleq d \log \left(a+b x\right)- \frac{x}{4}.
$$
As $b,d>0$, $f(x)$ is  concave and 
$
f'(x)= \frac{bd}{a+bx}-\frac{1}{4},
$
which equals $0$ when $x =4 d-\frac{a}{b}$.
Therefore,  
if $4d \ge  \frac{a}{b}$, 
then $f(x)$ for $x\ge 0$ is maximized at $x= 4d-\frac{a}{b}$, and $f(4d-\frac{a}{b}) =  d \log \left(4bd\right)- d+\frac{a}{4b}\le d \log \left(4 bd\right)$. Otherwise, $f(x)$ for $x\ge0$ is maximized at $x=0$,  and
$
f(0)= d \log a. 
$
In conclusion, we have 
$
\max_{k\ge  0} f(k) \le \max\{ d\log \left( 4 bd\right) , d\log a\}.
$
Then we get an upper bound on $\termII$ as
\begin{align*}
    \termII 
    & \le \exp\left(-\lambda \right) \sum_{k\ge  \gamma \lambda} \left(
    \frac{3e^{\frac{1}{4}}}{4}\right)^k \max\big\{  \left(4 bd\right)^d , a^d \big \}\\
    & \le 27 \exp\left(-\lambda \right)\max\big\{  \left(4 b d\right)^d , a^d \big \}.
\end{align*}
where the last inequality follows by $\frac{ \frac{3}{4} e^{1/4}  }{1 - \frac{3}{4} e^{1/4} } <27$. 

Next we prove \prettyref{eq:lmm_fact_2}. Substituting the Poisson PMF into \prettyref{eq:lmm_fact_2} yields that 
 \begin{align*}
 \Expect\left[\left(a+bX\right)^{X}\exp\left(\alpha X\right)\right] 
        & = \sum_{k=0}^{\infty}\frac{\lambda^k e^{-\lambda}}{k!}  (a+b k)^{k} \exp (\alpha k) \\
        & \overset{(a)}{\le } e^{-\lambda} \left(1+\sum_{k=1}^{\infty}\left(\frac{a\lambda e^{\alpha+1}}{k}+b e^{\alpha+1}\lambda\right)^{k} \right) \\
        & \overset{(b)}{=} O(1).
\end{align*}
where $(a)$ holds due to $k!\ge  \left(\frac{k}{e}\right)^k$ for $k\ge  1$, and $(b)$ holds because $a, b, \alpha, \lambda$ are fixed constants such that $b e^{\alpha+1}\lambda< 1$.  
 
It remains to prove \prettyref{eq:lmm_fact_3}. Given $b = o(1)$, $\alpha=o(1)$, we pick $t=\omega(1)$  such that $b t^2 +\alpha t =o(1)$ and get:
\begin{align*}
 \Expect\left[\left(1+bX\right)^{X}\exp\left(\alpha X\right)\right] 
        & = \sum_{k=0}^{t}\frac{\lambda^k e^{-\lambda}}{k!}  (1+b k)^{k}e^{\alpha k} + \sum_{k=t+1}^{\infty}\frac{\lambda^k e^{-\lambda}}{k!}  (1+b k)^{k}e^{\alpha k} \\
        & \overset{(a)}{\le } \sum_{k=0}^{t}\frac{\lambda^k e^{-\lambda}}{k!}  \exp\left(bk^2+\alpha k\right)+ e^{-\lambda} \sum_{k=t+1}^{\infty}\left(\frac{ \lambda e^{\alpha+1}}{k}+b e^{\alpha+1}\lambda\right)^{k}\\
        & = 1+o(1),
\end{align*}      
where the last equality holds because $\exp\left(bk^2+\alpha k\right)= 1+o(1)$ for any $0 \le k\le t$ given $bt^2+\alpha t=o(1)$, and $\sum_{k=t+1}^{\infty}\left(\frac{ \lambda e^{\alpha+1}}{k}+b e^{\alpha+1}\lambda\right)^{k} = o(1) $ for any $k>t$ given $t=\omega(1)$, $b=o(1)$, $\alpha= o(1)$, and $\lambda$ is some constant. 
\end{proof}

\section{Concentration Inequalities for Gaussians and Binomials}
\label{app:concentration}

\begin{lemma}[Hanson-Wright inequality] \label{lmm:hw}
Let $X, Y \in \reals^n$ are standard Gaussian vectors  such that the pairs 
$(X_i, Y_i)\sim \calN  \Big( \left(\begin{smallmatrix} 0\\ 0\end{smallmatrix}\right) , \left(\begin{smallmatrix} 1 & \rho \\ \rho & 1 \end{smallmatrix}\right)  \Big )
$ are independent for $i=1,\ldots,n$.
Let $M \in \reals^{n \times n}$ be any deterministic matrix. There exists some universal constant $c>0$
such that 
with probability at least $1-2\delta$,
\begin{align}
 \left| X^\top M Y -  \rho \Tr(M) \right| \le c \left( \|M\|_F \sqrt{ \log (1/\delta) } +  \| M \| \log (1/\delta) \right) \label{eq:hw_bilinear}. 
\end{align}
\end{lemma}


\begin{proof}
When $\rho=1$, i.e.~$X=Y$,  the bilinear form reduces to a quadratic form and this lemma is the original Hanson-Wright inequality \cite{hanson1971bound,rudelson2013hanson}. 
In general, note that 
\begin{align*}
X^\top M Y=\frac{1}{4}(X+Y)^\top M( X + Y)
 -\frac{1}{4}(X-Y)^\top M(X- Y ).
\end{align*}
Thus, it suffices to analyze the two terms separately.
Note that 
 \[
\expect{(X \pm Y)^\top M( X \pm  Y) } = (2 \pm 2\rho)  \Tr(M).
 \]
Applying the Hanson-Wright inequality, we get that with probability at least $1-\delta$, 
$$
\left| (X \pm Y )^\top  M (X \pm Y) - 
\expect{ ( X \pm Y )^\top  M ( X \pm Y) } \right|
\le c \left( \|M\|_F \sqrt{ \log (1/\delta) } +  \| M \| \log (1/\delta) \right),
$$
where $c$ is some universal constant. The conclusion readily follows by combining the last three displayed equations.
\end{proof}

\begin{lemma}[{C}hernoff's inequality for Binomials]
Suppose  $X \sim \Binom(n,p)$ with mean $\mu=np$.
Then for any $\delta>0$,
\begin{align}
\prob{ X \ge (1+\delta) \mu} \le \exp \left( - \mu \left( (1+\delta) \log (1+\delta) - \delta \right) \right), \label{eq:chernoff_binom_right}
\end{align}
and 
\begin{align}
\prob{ X \le (1-\delta) \mu } \le \exp\left( - \frac{\delta^2}{2} \mu \right). \label{eq:chernoff_binom_left}
\end{align}
In particular, it follows from \prettyref{eq:chernoff_binom_right} that  
\begin{align}
\prob{X \ge  \tau }\le \exp\left(-t\right), \quad \forall t>0,\label{eq:chernoff_binom_right_Lambert}
\end{align}
where $\tau = \mu \exp \sth{ 1+ W \left(\frac{t} {e\mu}-\frac{1}{e}\right)}$ and 
$W(x)$ is the Lambert W function defined on $[-1,\infty]$ as the unique solution of $W(x) e^{W(x)} =x$ for $x \ge -1/e$. 
\end{lemma}
\begin{proof}
Note that \prettyref{eq:chernoff_binom_right} and \prettyref{eq:chernoff_binom_left} are direct consequences of \cite[Theorems 4.4 and 4.5]{ProbabilityComputing05}. 
To derive \prettyref{eq:chernoff_binom_right_Lambert} from \prettyref{eq:chernoff_binom_right}, we use 
 $$
 y = \exp \left( 1+ W(x/e) \right) \iff \log \frac{y}{e}  = W(x/e) \iff \log \frac{y}{e}   \exp \left(  \log \frac{y}{e}   \right) = \frac{x}{e}  \iff y \log y - y = x,
 $$
 and let $x=t/\mu -1$, $y=1+\delta$, and $\tau= y \mu$. 
\end{proof}

\section{Facts on Random Permutation}
\label{app:perm}
In this appendix we collect several useful facts about random permutation (cf.~\cite{arratia1992cycle}). For any $\ell\in\naturals$, let $n_\ell$ denote the number of $\ell$-cycles in a uniform random permutation $\sigma \in \calS_n$. Let $\{Z_{\ell}\}_{1\le  \ell \le k}$ denote a sequence of independent Poisson random variables where $Z_{\ell} \sim \Pois\left(\frac{1}{\ell}\right)$. 
\begin{lemma}
\label{lmm:jointpmf}
For any $k\in[n]$ and $a_1,a_2,\cdots,a_k \in\integers_+$,
\begin{align}
 \prob{n_1 \ge a_1,n_2 \ge a_2,\cdots,n_k \ge a_k}
 & \le \frac{1}{\prod_{\ell=1}^k \ell^{a_\ell} a_\ell!}.
\label{eq:jointdis1} 
\end{align} 
Consequently, for any nonnegative function $g$, 
\begin{equation}
\Expect[g(n_1,\ldots,n_k)] \le \Expect[g(Z_1,\ldots,Z_k)] \exp\left(\harmonic_k\right) 
\label{eq:jointdis2}
\end{equation}
where $\harmonic_k = \sum_{1\le \ell\le k} \frac{1}{k}$ denote the harmonic number. 
\end{lemma}
\begin{proof}
It suffices to check the first inequality.
Note that for all $\sum_{\ell=1}^k \ell a_{\ell} \le n$,
$\expect{\prod_{1\le \ell\le k}\binom{n_{\ell}}{a_{\ell}}} = \frac{1}{\prod_{\ell=1}^k \ell^{a_\ell} a_\ell!}$ (see~e.g.~\cite[Eq.~(5)]{arratia1992cycle}). Then \prettyref{eq:jointdis1} follows due to $\indc{n_\ell\ge a_\ell}  \le \binom{n_\ell}{a_\ell}$ for $1\le \ell \le k$.
\end{proof}

\begin{lemma}[{\cite[Theorem 2]{arratia1992cycle}}]
\label{lmm:cycle_decomposition_poisson_distribution_total_v}
For any $1 \le k< n$,
the total variation distance between the law of $\{n_{\ell}\}_{ 1\le \ell\le k}$ and the law of $\{Z_{\ell}\}_{1\le  \ell \le k}$ satisfies:
\begin{align}
\mathrm{TV} \left( \calL\left( n_1, n_2, \ldots, n_k \right), \calL\left( Z_1, Z_2, \ldots, Z_k \right)\right)  \le  F\left(\frac{n}{k} \right), 
\label{eq:cycle_decomposition_poisson_distribution_total_v}
\end{align}
where 
$
 F(x) = \sqrt{2\pi m} \frac{2^{m-1}}{\left(m-1\right)!}+\frac{1}{m!}+3\left(\frac{x}{e}\right)^{-x},
$
with $m\triangleq \ceil{x}{}$, so that $\log F(x)= -x\log x(1+o(1))$ as $x\diverge$.
\end{lemma}

\section*{Acknowledgment}
The authors thank Cristopher Moore for simplifying the proof of \prettyref{prop:cycle}.
J.\ Xu would like to thank Tselil Schramm for helpful discussions
on the hypothesis testing problem
at the early stage of the project.
\bibliography{detection}
\bibliographystyle{alpha}

\end{document}